\documentclass[11pt,leqno,a4paper]{amsart}

\usepackage{amssymb,amsmath,amscd,graphicx,fontenc,amsthm,mathrsfs}
\usepackage[frame, cmtip, arrow, matrix, line, graph, curve]{xy}
\usepackage{graphpap, color}
\usepackage[mathscr]{eucal}
\usepackage{ifthen}
 
\DeclareMathOperator{\spn}{span}
 
\usepackage{fancyhdr}
\usepackage{indentfirst}
\usepackage{paralist}
\usepackage{pstricks}
 \usepackage{bbm}

\newtheorem{theorem}{Theorem}[section]
\newtheorem{definition}[theorem]{Definition}
\newtheorem{lemma}[theorem]{Lemma}
\newtheorem{proposition}[theorem]{Proposition}
\newtheorem{corollary}[theorem]{Corollary}

\theoremstyle{remark}
\newtheorem{remark}[theorem]{Remark}

\newcommand{\I}{\mathsf{I}}
\newcommand{\p}{p}

\newcommand{\D}{{\mathcal D}}

\newcommand{\NN}{{\mathbb{N}}}
\newcommand{\ZZ}{{\mathbb{Z}}}

\newcommand{\TT}{{\mathcal{T}}}

\newcommand{\RR}{{\mathbb{R}}}
\newcommand{\PP}{{\mathbb{P}}}
\newcommand{\QQ}{{\mathbb{Q}}}

\newcommand{\GG}{{\mathbf{G}}}

\renewcommand{\SS}{{\mathbb{S}}}
\renewcommand{\PP}{{\mathbb{P}}}
\newcommand{\U}{\mathbb{U}}

\DeclareMathOperator{\rk}{rank}

\DeclareMathOperator{\codim}{codim}
\DeclareMathOperator{\vol}{vol}

\newcommand{\GL}{\mathrm{GL}}
\newcommand{\SL}{\mathrm{SL}}

\newcommand{\PSL}{\mathrm{PSL}}

\newcommand{\SU}{\mathrm{{SU}}}
\newcommand{\SO}{{\mathrm{SO}}}

\newcommand{\Or}{\mathrm{O}}

\renewcommand{\c}{\diamond}
\newcommand{\pr}{\pi}

\renewcommand{\t}{\mathsf{t}}

\DeclareMathOperator{\Lie}{Lie}
\newcommand{\Sf}{\prod_{p\in S_f}}

\newcommand{\ox}{\mathbf x}
\newcommand{\oy}{\mathbf y}
\newcommand{\ow}{\mathbf w}
\newcommand{\ov}{\mathbf v}
\newcommand{\ou}{\mathbf u}
\newcommand{\oc}{\mathbf c}
\providecommand{\ve}{\mathbf{ e}}

\newcommand{\ct}[2]{{\mathbf{N}}_{#1}(#2)}

\newcommand{\vt}[2]{{\mathbf{V}}_{#1}(#2)}

\newcommand{\T}{{\mathsf{T}}}

\newcommand{\leb}{\lambda}

\newcommand{\sh}{\mathsf{h}}
\newcommand{\sg}{\mathsf{g}}
\newcommand{\sy}{\mathsf{y}}

\DeclareMathOperator{\supp}{supp}
\providecommand{\vol}{\mathrm{vol}}

\newcommand{\tran}{\mathrm{tran}}
\newcommand{\diag}{\mathrm{diag}}

\renewcommand{\red}{\mathrm{red}}

\newcommand{\q}{{\mathbf q}}

\newcommand{\Sing}{\mathscr{S}}
\newcommand{\Gen}{\mathscr{G}}

\newcommand{\SK}{\mathsf{K}}
\newcommand{\SH}{\mathsf{H}}
\newcommand{\ST}{\mathbb{T}}
\newcommand{\sk}{\mathsf{k}}
\renewcommand{\SU}{\mathsf{U}}

\newcommand{\SV}{\mathsf{V}}
\newcommand{\sz}{\mathsf{z}}
\newcommand{\SG}{\mathsf{G}}

\newcommand{\SP}{\mathsf{P}}
\newcommand{\prT}{\| \T \|}
\newcommand{\prTi}{\| \T_i \|}

\renewcommand{\1}{\mathbf{1}}

\renewcommand{\sl}{\mathfrak{sl}}

\renewcommand{\O}{\mathcal{O}}
\newcommand{\SKK}{\mathscr{K}}
\newcommand{\SF}{\mathscr{F}}

\newcommand{\SD}{\mathscr{D}}
\newcommand{\CA}{{{\mathbb{U}}_p}}
\newcommand{\evit}{\mathscr{E}}
\newcommand{\sing}{\mathscr{S}}
\newcommand{\N}{\mathcal{N}}
\newcommand{\w}{\mathbf{w}}
\newcommand{\m}{\mathbf{m}}
\newcommand{\M}{\mathbf{M}}
\DeclareMathOperator{\Rep}{Rep}

\newcommand{\GP}{{\mathbf{P}}}
\newcommand{\GQ}{{\mathbf{Q}}}

\newcommand{\F}{\mathcal{F}}

\providecommand{\RR}{\mathbb{R}} \providecommand{\ZZ}{\mathbb{Z}}
\providecommand{\NN}{\mathbb{N}}

\providecommand{\ox}{\overrightarrow x}
\providecommand{\oy}{\overrightarrow y}
\providecommand{\ow}{\overrightarrow w}

\renewcommand{\L}{\mathcal{L}}

\DeclareMathOperator{\card}{card}

\newcommand{\GN}{{\mathbf{N}}}
\newcommand{\GX}{{\mathbf{X}}}

\newcommand{\ol}{\mathfrak{O}}

\providecommand{\ox}{\mathbf x}
\providecommand{\oy}{\mathbf y}
\providecommand{\ow}{\mathbf w}

\providecommand{\ov}{\mathbf v}
\providecommand{\oq}{\mathbf q}

\providecommand{\M}{\mathbf{M}} \providecommand{\m}{\mathbf{m}}

\definecolor{han}{rgb}{1.0, 0, 0}

\begin{document}
\title[Values of Isotropic quadratic forms at $S$-integral points]{Asymptotic distribution of values of isotropic quadratic forms at $S$-integral points}
\author{Jiyoung Han}
\address{Department of Mathematical Sciences, Seoul National University, Kwanak-ro 1, Kwanak-gu, Seoul
08826} \email{jiyoung.han.math@snu.ac.kr}

\author{Seonhee Lim}
\address{Department of Mathematical Sciences, Seoul National University, Kwanak-ro 1, Kwanak-gu, Seoul
08826} \email{slim@snu.ac.kr}
\author{Keivan Mallahi-Karai}
\address{Department of Mathematics, Jacobs University, Bremen, Germany} \email{k.mallahikarai@jacobs-university.de}


\maketitle
\begin{abstract}
We prove an analogue of a theorem of Eskin-Margulis-Mozes \cite{EMM}: suppose we are given a finite set of places $S$ over $\QQ$ containing the archimedean place and excluding the prime $2$, an irrational isotropic form $\q$ of rank $n\geq 4$ on $\QQ_S$, a product of 
 $p$-adic intervals $\I_p$, and a product $\Omega$ of star-shaped sets.
 
We show that unless $n=4$ and $\q$ is split in at least one place, the number of $S$-integral vectors $\ov \in \T \Omega$ satisfying simultaneously $\q(\ov) \in I_p$  for $p \in S$ is asymptotically given by $$
  \lambda(\q, \Omega) |\I| \cdot \prT^{n-2},$$ as $\T$ goes to infinity, where $|\I |$ is the product of Haar measures of the $p$-adic intervals $I_p$.

The proof uses dynamics of unipotent flows on $S$-arithmetic homogeneous spaces; in particular, it relies 
on an equidistribution result for certain translates of orbits applied to test functions with a controlled growth at infinity, specified by an $S$-arithmetic variant of the $ \alpha$-function introduced in \cite{EMM}, and an $S$-arithemtic version of a theorem of Dani-Margulis \cite{DM}.
 \end{abstract}

\tableofcontents

\section{Introduction}

The Oppenheim conjecture, settled by Margulis in 1986 \cite{Mar}, states that for any non-degenerate indefinite quadratic form $\q$ over $\RR^n$,  $n \ge 3$, the set $\q(\ZZ^n)$ of values of integral vectors is a dense subset of $\RR$ if $\q$ is \emph{irrational}, i.e. if $\q$ is not proportional to a form with rational coefficients. Both indefiniteness and irrationality conditions are easily seen to be necessary. Margulis' proof is based on a study of orbits of certain unipotent one-parameter groups on the space of lattices in $\RR^3$. {The idea of using} unipotent flows dates back in an explicit form to Raghunathan and in implicit forms to Cassels and Swinnerton-Dyer. More precisely, Raghunathan observed that Oppenheim conjecture follows from the assertion that every bounded orbit of the orthogonal group $\SO(2,1)$ 
on $\SL_3(\RR)/\SL_3(\ZZ)$ is compact. Raghunathan {further} conjectured in the mid-seventies that if $G$ is a connected Lie group, $\Gamma$
a lattice in $G$, and $U$ a unipotent subgroup of $G$, then the closure of any orbit $Ux$, for $x \in G/\Gamma$, is itself an orbit $Lx$, where $L$
is a closed connected subgroup of $G$ containing $U$. Note that the subgroup $L$ depends on $x$. Raghunathan's conjecture was later proved by Ratner~\cite{Ra}.

\subsection{The Quantitative Oppenheim Conjecture}
The Quantitative Oppenheim conjecture, a refinement of the Oppenheim conjecture, is {about} the asymptotic distribution of the values of $\q(\ov)$, where $\ov$ runs over integral vectors in a {large} ball. To be more precise, let $\Omega$ be a radial set defined by $\Omega = \{ \ox \in \RR^n: ||\ox|| < \rho (\ox/||\ox||) \},$ where $\rho$ is a positive continuous function on the unit sphere and $T\Omega$ {is} the dilation of $ \Omega$ by a factor $T>0$.
Let us also denote by $\ct{(a,b),\q, \Omega}{T}$ the number of integral vectors
$\ox \in T \Omega$ such that $\q(\ox) \in (a,b)$.  

In \cite{DM}, Dani and Margulis obtained
the asymptotic exact lower {bound}:
for $\q$ an irrational indefinite quadratic form {on} $\RR^n$ with $n \ge 3$, for a given interval $(a,b)$ and a set $ \Omega$, we have
\[ \liminf_{T \to \infty} \frac{\ct{(a,b),\q, \Omega}{T}}{ \lambda_{\q, \Omega}(b-a) T^{n-2} }
\ge 1, \]
where $ \lambda_{\q, \Omega}$ is a constant such that denominator gives the asymptotic volume of the set 
$\{\ov \in T\Omega: \q(\ov)\in(a,b)\}$. The proof of the above theorem is based on studying the distribution of the shifted orbits of the form $u_tKx$, where $u_t$ is a certain one-parameter unipotent subgroup, $K$ is a maximal compact subgroup of the orthogonal group {$\SO(\q)$ associated to $\q$}, and $x \in \SL_n(\RR)/\SL_n(\ZZ)$. A suitable choice of a function $f$ on $\RR^n$ allows one to approximate the number of the integral vectors $\ov$ in $T \Omega$ satisfying $\q(\ov) 
\in (a,b)$ by an integral of the form
\[ \int_K \widetilde{ f}(u_tkx) d\mu(x), \]
where $d\mu$ is the invariant probability measure on the space of unimodular lattices $\SL_n(\RR)/\SL_n(\ZZ)$, and $
\widetilde{ f}: \SL_n(\RR)/\SL_n(\ZZ) \to \RR$ is the Siegel transform of $f$ defined by $ \widetilde{f}(g\Gamma)= \sum_{\ov \in g\ZZ^n} f( \ov)$. 
Although $ \widetilde{ f}$ is an 
unbounded function to which the uniform version of Ratner's theorem proved in \cite{DM}
does not directly apply, it is yet possible to establish asymptotically sharp lower bounds by approximating $ \widetilde{f}$ from below by bounded functions.  

The question of establishing an asymptotically sharp upper bound turns out to be substantially subtler. Note that, in general, the equidistribution results in \cite{DM} do not apply to arbitrary unbounded function.  In the groundbreaking work \cite{EMM}, it is shown that if $f$ has 
a controlled growth at infinity, then one can still establish this equidistribution result. The growth at infinity is measured by a family of functions $ \alpha_1, \dots, \alpha_n$ with origin in geometry of numbers. If $ \Lambda$ is a lattice in $\RR^n$, then 
$\alpha_i( \Lambda)^{-1}$ is the smallest value for the covolume of $ \Lambda \cap L$, where $L$ ranges over all $i$-dimensional subspaces of $\RR^n$ for which $ \Lambda \cap L$ is a lattice in $L$. 

Setting $ \alpha= \max( \alpha_i)$, one of the main results of \cite{EMM} is that the equidistribution results of \cite{DM} continue to hold as long as the test function is majorized by $ \alpha^s$ for some $s \in (0,2)$. This is then used to show that if $\q$ is a form of signature $(r,s)$ with $r \ge 3$ and $s \ge 1$, then
\[ \lim_{T \to \infty} \frac{\ct{(a,b),\q, \Omega}{T}}{ \lambda_{\q, \Omega}(b-a) T^{n-2} }
= 1. \]

For quadratic forms of signature $(2,1)$ and $(2,2)$, certain quadratic forms were constructed in \cite{EMM} so that by choosing $ \Omega$ to be the unit ball, for any $ \epsilon>0$, \[ \ct{(a,b),\q_i, \Omega}{T_j} \gg T_j^{i} (\log T_j)^{1- \epsilon} \]
for an infinite sequence $T_j \to \infty$. These counterexamples are very well approximable by rational forms. 
When $\q$ has signature $(2,2)$, it is proven in \cite{EMM2} that if forms of a class called EWAS are excluded, then a similar asymptotic statement continues to hold. 

\subsection{The $S$-arithmetic Oppenheim Conjecture}
The generalization we study in this article involves considering various places. 

Let $S_f$ be a finite set of odd prime numbers
and $S=\{ \infty \} \cup S_f$. Each element of $S$ ($S_f$,
respectively) is called a \emph{place} (a finite place, respectively).
For $p \in S_f$, the $p$-adic norm on $\QQ$ is defined by $|x|_p=p^{-v_p(x)}$, where $v_p(x)$ is the $p$-adic valuation of $x$. The completion of $\QQ$ with respect to $| \cdot |_p$ is
the field of $p$-adic numbers and is denoted by $\QQ_p$. We will sometimes write $\QQ_{ \infty}$ for $\RR$. When $p \in S_f$, $\QQ_p$ contains a maximal 
compact subring $\ZZ_p$ consisting of elements with norm at most $1$. We also write $\U_p= \ZZ_p - p \ZZ_p$ for the set of multiplicative units in $\ZZ_p$. When $p=\infty$, 
we set $\U_p= \{ \pm 1\}$. We will use the notation $\U_p^n = \ZZ_p^n - p(\ZZ_p^n)$ for $p \in S_f$ and $\U_\infty^n =\SS^{n-1}$, the set of vectors of length $1$ in $\RR^n$.

For $p \in S_f$, $\lambda_p$ stands for the Haar measure on $\QQ_p$ normalized so that 
$\lambda_p(\ZZ_p)=1$.  We denote by $\QQ_S=\prod_{p \in S} \QQ_p$ the direct product of $\QQ_p$, $p \in S$, and write $\ZZ_S=\{ x \in \QQ_S :  |x|_p \le 1, \; \forall p \not\in S\}$ for the ring of $S$-adic integers. 
The function on $\QQ_S^n$ defined by
$$ \|v \|= \prod_{p \in S} \| v_p \|_p,$$
where $\| v_p \|_p = \max |(v_p)_i|_p$ for $p \in S_f$, plays a role of the norm on $\RR^n$ \cite{KT}. Also, $\| \cdot \|_{ \infty}$ denotes the usual Euclidean norm. We will also denote by $\lambda_S$ the product measure
$\otimes_{ p \in S} \lambda_p$. 
The Haar measures on $\QQ_p^n$ and $\QQ_S^n$ are denoted by $\lambda_p^n$ and $\lambda_S^n$, respectively.

A quadratic form $\q$ on $\QQ_S^n$ is an $S$-tuple $(\q_p)_{p \in S}$, where $\q_p$ is a quadratic form on $\QQ_p^n$. The form $\q$ is called \emph {isotropic}, if each $\q_p$ is isotropic, that is, if there exists a non-zero vector $\ov \in \QQ_p^n$ such that $\q_p(\ov)=0$. A quadratic form over $\RR$ is isotropic if and only if it is indefinite. $\q$ is called \emph{non-degenerate} if each $\q_p$ is non-degenerate, that is, the symmetric matrix $\left(\beta_{\q_p}(\ov_i, \ov_j)\right)_{1\le i,j \le n}$ is invertible, where $\{\ov_1, \ldots, \ov_n\}$ is a basis of $\QQ_p^n$ and $\beta_{\q_p}(\ov_i, \ov_j)=(\q_p(\ov_i+\ov_j)-\q_p(\ov_i)-\q_p(\ov_j))/2$. Finally, we call $\q$ {\it rational} if there exist a single quadratic form $\q_0$ (defined over $\QQ$) and an invertible element $\lambda=(\lambda_p)_{p\in S} \in \QQ_S$ satisfying that $\q_p= \lambda_p \q_0$ for all $p \in S$. Otherwise $\q$ is called {\it irrational}. 

Borel and Prasad proved the $S$-arithmetic Oppenheim conjecture {over} any number field $k$. In the case when $k=\QQ$, the theorem says that 
for a non-degenerate irrational isotropic quadratic form $\q$ on $\QQ_S^n$ ($n \geq 3$) and any given $ \epsilon>0$, there exists a {non-zero} vector $\ox \in \ZZ_S^n$ with 
$|\q_p(\ox)|< \epsilon$ for all $p \in S$.

\subsection{Statements of results}
The goal of this paper is to prove a quantitative version of the theorem of Borel and Prasad mentioned above.
Note that due to the presence of $|S|$ valuations, we need to consider $|S|$ parameters for the divergence to the infinity, etc. 

Let $\T=(T_p)_{p \in S}$ be an $S$-tuple of positive real numbers with the components $T_p \in p^\ZZ$ for $p \in S_f$. Such an $S$-tuple $\T$ will be called an $S$-time. The set of all $S$-time vectors is denoted by $\TT_S$. For $\T \in 
\TT_S$, write $\prT= \prod_{p \in S} T_p.$, and $m(\T)=\min_{p \in S} T_p $. For $\T=(T_p), \T'=(T'_p) \in \TT_S$, define $\T \succeq \T'$ if $T_p \ge T'_p$ for all $p \in S$. Similarly, $\T_i =(T_{p,i})_{ p\in S} \to \infty$ means that $T_{p,i} \to \infty$ for each $p \in S$. Since $\T$ is of the form
$(e^{t_\infty}, p_1^{t_1}, \cdots, p_s^{t_s})$, let us denote $\t=(t_{\infty}, t_1, \cdots, t_s)$ and consider $S$-parameter groups 
$$\ST = \RR \times \prod_{ p \in S_f} \ZZ, \qquad 
\ST^+ = \RR^+ \times \prod_{ p \in S_f} \ZZ^+.$$ 
{}{which we will use later. The notions $\t \succeq \t'$ and $\t \rightarrow \infty$ are defined accordingly.}

For $p \in S$, let $\rho_p: \U_p^n \to \RR$ be a positive continuous function.
For $p \in S_f$, we will assume throughout the paper that $\rho$ satisfies the following condition:
\begin{equation}\label{nu-s}\tag{$\mathbf{I_{\rho}}$}
\rho_p( u\ox)= \rho_p(\ox), \quad \forall \ox \in \U_p^n, \quad \forall u \in \U_p.
\end{equation}
Define 
\begin{equation}\label{omg}
\Omega = \prod_{p \in S}  \Omega_p \subseteq  \QQ_S^n,
\end{equation}
where $\Omega_p$ is the set of vectors $\ov_p$ whose norm is bounded by the value of $\rho_p$ in the direction of $\ov_p$. Condition \eqref{nu-s} is indeed very mild and is satisfied by many sets of interest (e.g., the unit ball, which is defined by the constant function).
Having fixed an $S$-time $\T$, we denote by $\T \Omega = \{ (z_p \ov_p)_{p \in S}: z_p \in 
 \QQ_p,  |z_p| \le T_p, \ov_p \in \Omega_p  \}$ the dilation of $\Omega$ by $\T$. A $p$-adic interval of length $p^{-b}$ is a set of the form {}{$I_p=a+p^b\ZZ_p$,} where $a \in \QQ_p$ and $b \in \ZZ$. An $S$-adic interval $\I$
is a product of $p$-adic {intervals} $I_p$, with $p \in S$. Let us denote $|\I |=\prod \lambda_p(I_p)$.
\begin{definition}[Counting and volume functions]
Fixing $ \Omega$ as above, the counting and volume functions are defined respectively by
\begin{equation*}
\begin{split}
 \ct{\I,\q, \Omega}{\T} &= \card \{ \ov \in \ZZ_S^n \cap \T \Omega : \q_p(\ov_p) \in I_p, \forall p \in S \},  \\
 \vt{\I,\q, \Omega}{\T} & = \vol \{ \ov \in \QQ_S^n \cap \T \Omega: \q_p(\ov_p) \in I_p, \forall p \in S \},
\end{split}
\end{equation*}
where $\card$ and $\vol$ denote the cardinality and the Haar measure $\lambda_S^n$. 
\end{definition}

We start with the following statement about the asymptotic {}{volume of} the set of vectors in $T_p \Omega_p \subseteq \QQ_p^n$ with the condition  $\q(\ov) \in \I$. This proposition will be proven in Section \ref{passage}. 
\begin{proposition}\label{volume-asym} For a constant $\lambda=\lambda(\q,\Omega)$, we have,
\[ \vt{\I,\q, \Omega}{\T}\sim \lambda(\q,\Omega) \cdot | \I | \cdot \prT^{n-2},
\]
as $\T\rightarrow \infty$.
\end{proposition}
The 
asymptotic behavior of $\ct{\I,\q, \Omega}{\T}$ is more intricate than the volume asymptotics. The issue for real forms of signature $(2,1)$ and $(2,2)$ persists in the $S$-arithmetic setup. 
%However, $p$-adic forms of rank $4$ behave more nicely than the corresponding real forms of signature $(2,2)$, which allow us to prove a partial result even in this case.

\begin{definition}Let $\q=(\q_p)_{ p \in S}$ be an isotropic quadratic form on $\QQ_S^n$. We say that $\q$ is \emph{exceptional} if either (1) $n\leq 3$, or (2) $n=4$ and for some $p$, $\q_{p}$ is split, i.e., it is equivalent to the form $x_1x_4 + x_2^2 - x_3^2$.
\end{definition}
See Section~\ref{subsec:quad} for more details.
Note that the set $\O_S(n)$ of non-degenerate quadratic forms on $\QQ_S^n$ can be identified with a Zariski open subset of $ \prod_{p \in S} \text{sym}_n(\QQ_p)$,
where $\text{sym}_n(\QQ_p)$ denotes the set of $n \times n$ symmetric matrices over $\QQ_p^n$. In this correspondence, a form $\q=(\q_p)_{ p \in S}$ is 
associated to its Gram matrix.

\begin{theorem}\label{uniform-upper-bound} Let $\D$ be a compact subset of non-exceptional quadratic forms on
$\QQ_S^n$, $\I$ an $S$-interval, and  $ \Omega$ as in \eqref{omg}.  Then there exists a constant
$C=C(\mathcal D,\I,\Omega)$ such that for
any $\q \in \mathcal D$ and any sufficiently large $S$-time $\T$ we have
\[\ct{\I, \q, \Omega}{\T} \le C  \prT^{n-2}.
\]
\end{theorem}

Note that this theorem is effective in the sense that the constant $C$ can be explicitly given in terms of the {}{data}. The next theorem
is our main result which
provides asymptotically sharp bounds for $\ct{\I, \q, \Omega}{\T}$. For the sake of simplicity, the theorem below is stated for an individual irrational 
form. Modifications of the proof, along the lines of the proofs in \cite{EMM}, can be made to establish a uniform version when $\q$ runs over a compact set 
of forms.

\begin{theorem}\label{main:asymptotics} Let $\q$ be an irrational isotropic form in $\O_S(n)$, $\I$ be an $S$-interval, 
and $ \Omega$ be as in \eqref{omg}. If $\q$ is not exceptional,  then
\[ \lim_{\T \to \infty} \frac{\ct{\I, \q, \Omega}{\T}}{ \lambda(\q, \Omega) |\I| \cdot \prT^{n-2}}=1. \] 
\end{theorem}

\begin{remark}
The proof given here establishes both the lower and the upper bound at the same time. It seems likely to us that an adaptation of the arguments in \cite{DM} can be used to establish the lower bound even when $\q$ is exceptional. 
\end{remark}

%:1
The proof of Theorem \ref{main:asymptotics} rests upon
a number of ingredients. We will use the dynamics on $\SG/\Gamma$, where $\SG=\SL_n(\QQ_S)$ and $\Gamma=\SL_n(\ZZ_S)$,
which can be identified with a set $\mathcal L_S$ of unimodular $S$-lattices. First, we will relate the counting problem to a question about the asymptotic behavior of integrals of the form
\[ \int_\SK \widetilde{ f}(a_\t \sk  \Delta) dm(\sk),\]
where $\SK$ is a maximal compact subgroup of $\SO(\q)$, $x$ is an element in 
%$ \SL_n(\QQ_S)/\SL_n(\ZZ_S)$
$\SG/\Gamma$ related to $\q$, $m$ is the normalized Haar measure of $\SK$, and $a_\t$ is a 1-parameter diagonal subgroup of $\SO(\q)$ (see Equation \eqref{def_a_t} in Section \ref{Orthogonal groups}).
Here, $\widetilde{ f}$ is the Siegel transform of a compactly supported function $f$ defined on $\QQ_S^n$ (see Definition \ref{def_Siegel}). Such integrals, when $ \widetilde{ f}$ is replaced by a bounded  continuous function can be dealt with using an $S$-arithmetic version of the results in \cite{DM}, which we will state and prove in Section
\ref{s:DM} of this paper.

Since $ \widetilde{ f}$ is unbounded, in order to use the equidistribution result just described, one needs to also control 
the integral of $ \widetilde{f} ( a_\t \sk \Delta)$ when $ a_\t \sk \Delta$ is far into the cusp. As in \cite{EMM}, this is dealt with 
using the function $ \alpha_S$ (see section \ref{alphas}), which is an analog of $ \alpha$ introduced in the previous subsection. We prove

\begin{theorem}\label{bound} With the notation as above, assume that $\q$ is not exceptional. For $0 < s< 2$ and for any $S$-lattice $ \Delta \in \L_S$, 
{ \[ \sup_{\t \succ 1} \int_\SK \alpha_S( a_\t \sk \Delta)^s dm(\sk) < \infty. \]}
Moreover, the bound is uniform as $ \Delta$ varies over a compact subset $\mathcal C$ of $ \L_S$.
\end{theorem}

We will prove Theorem \ref{main:asymptotics} using Theorem \ref{main-DM} and the following result which uses Theorem \ref{bound} as an important ingredient:

\begin{theorem}\label{ergodic-a}
Set $\SG=\SL_n(\QQ_S), \Gamma=\SL_n(\ZZ_S), \SH=\SO(\q)$, and let $\phi: \SG/\Gamma \to \RR$ be a continuous function, and $\nu$ be a positive continuous 
function on $\prod_{p\in S} \U_p^n$.  Assume that for some $0<s<2$, 
we have $|\phi( \Delta)|< C \alpha_S( \Delta)^s$ for all $ \Delta \in \SG/\Gamma$. Let 
$x_0 \in \SG/\Gamma$ be such that $\SH x_0$ is not closed. Then 
\[ \lim_{ \t \to \infty} \int_\SK \phi(a_\t \sk x_0) \nu(\sk) dm(\sk) =
\int_{\SG/\Gamma} \phi(y) d\mu(y) \int_\SK \nu \, dm(\sk). \]
\end{theorem}

%
%Similarly to \cite{EMM}, the proof of this theorem relies upon Theorem \ref{bound} and
%a uniform version of Ratner's theorem. 

It is natural to inquire what happens in the case of exceptional forms. In this direction, we can prove the following theorem:

\begin{theorem}\label{counter-example}
Let $\I \subseteq \QQ_S$ be an $S$-interval and  $ \Omega$ be the product of unit balls in $\QQ_p^3$ for $p \in S$. 
Then, there exist an isotropic irrational quadratic form on $\QQ_S^3$, a constant $c>0$, and a sequence $\T_i \to \infty$ 
such that 
\[\ct{\I, \q,\Omega}{\T_i} > c \prTi ( \log \prTi )^{1 - \epsilon}. \]
\end{theorem}

Since {}{scalar} multiples of a single form $\q$ cover all possible equivalence classes of isotropic quadratic forms over $\QQ_p^3$ (this can be easily seen from the proof of Proposition 
\ref{same-ortho}), the above theorem shows that the counterexamples can be constructed in any equivalence class of quadratic forms. This theorem, however, leaves out the case of forms with four variables. It would be interesting to study these forms in details {}{in the future}.

\begin{remark} 
 In this paper, sans serif roman letters are reserved for $S$-adic objects.
Algebraic groups will be denoted by bold letters. For instance, when $\GG=\mathbf{SL}_n$, then $\GG(\QQ_p)=\SL_n(\QQ_p)$ and $\SG= \prod_{p \in S } \GG(\QQ_p)$. Elements of $\SG, \SH$, etc. are denoted by lower case $\sg, \sh$, etc.
\end{remark}

\begin{remark} In order to streamline the process of normalizing a vector across different places, we will use the following notation. For $p \in S_f$ and a real number $x=p^n$, we set $x^{\c}=p^{-n}$, viewed as a {\it $p$-adic} number. If $x$ is a positive real number, then, by definition $x^{\c}=x$. If $\T=(T_p)_{p \in S}$ is an $S$-time (which, as defined above, means that $T_p$ is a power of $p$ for every $p \in S_f$), we define $T^{\c}=(T_p^{\c})_{p \in S}$; in particular, if $\ov \in \QQ_S^n$, then each component of the vector $\ov/\|\ov\|^{\c}:= (\ov_p/\| \ov_p \|_p^{\c})$ has norm one. 
\end{remark}
\begin{remark}
It would be interesting to extend the result to include the prime $p=2$ in $S$. We expect some modifications to be necessary, for example, Proposition~\ref{L:h} and classification of quadratic forms depend on  whether $p$ is 2 or odd.
\end{remark}
\begin{remark}
It is likely that the results of this paper can be generalized to the case of arbitrary number fields. This, however, will inevitably involve using more number theory. For instance, one will need a classification of quadratic forms over finite extensions of $p$-adic fields, which is more complicated than that over $p$-adic fields.
\end{remark}
%Prime powers $p^n$ can be considered both real and $p$-adic numbers. In some occasions (see Lemma \ref{L:JF}) we need to transform a real $p^n \in \RR$ to $p^n \in \QQ_p$, which will be denoted by: $(p^n)^{\c}=p^{-n} \in \QQ_p$. This notation also renders some formulae more familiar; for instance, the unit $p$-adic vector in the same direction as a vector $\ov$ 
%is $\ov/\| \ov \|^{\c}$. When $p= \infty$, we set $x^\c=x$. For an $S$-vector \seon{$x=(x_p)_{p \in S} \in \RR\times \underset{p\in S_f}{\prod} p^{\ZZ}$}, we set $x^\c=(x_p^\c)_{p \in S} \in \QQ_S$. \seon{$\| \|^{\c}$ is not defined since so far $\| \|$ is a real number, not a vector.}

\section*{Acknowledgement} 
We would like to thank G. Margulis, A. Mohammadi and A. Eskin for helpful discussions.
We thank H. Oh and D. Kleinbock for letting us know about R. Cheung's Ph. D. thesis: he has obtained partial results,  where he proves Theorem \ref{bound} and uses it to 
establish a result analogous to Theorem \ref{uniform-upper-bound}. However, the upper bound he obtained
is not sharp, and no lower bound was established. Methods used in \cite{Rex} are also somewhat
different from ours. The general line of argument is, however, similar. Moreover, his proof covers all local fields
of characteristic zero. 

The first author was supported by Basic Science Research Program through the National Research Foundation of Korea (NRF) funded by the Ministry of Education (2016R1A6A3A01010035), the second author is supported by Samsung Science and Technology Foundation under Project No. SSTF-BA1601-03.
We also acknowledge the support of Seoul National University, MSRI and Jacobs University during our visits. We are also grateful to the referees whose detailed comments helped us improve the paper.

\section{Preliminaries}
In this section, we recall some definitions about $\QQ_p$-vector spaces and quadratic forms defined on them.

\subsection{Norms on exterior products}
Let $ 1 \le i \le n$. One equips the $i$th exterior product $\bigwedge^i\RR^n$ with an inner product defined by
\begin{equation*}\label{inner} \langle \ox_1\wedge \cdots \wedge \ox_i, \oy_1\wedge \cdots \wedge \oy_i \rangle =\det\left(\ox_k \cdot \oy_j\right)_{1\leq k, j\leq i}.
\end{equation*}
For $\ov \in \bigwedge^i\RR^n$, the induced norm is given by 
$\|\ov \|= | \langle \ov, \ov \rangle|^{1/2}$. 
Denote by $\ve_1, \dots, \ve_n$ the canonical basis for $\RR^n$. For $J=\{1 \leq j_1 < \ldots < j_i \leq n\}$, denote $\ve_J= \ve_{j_1}\wedge \cdots \wedge \ve_{j_i}$. As $J$ runs over all subsets of $\{1, 2 , \dots, n \}$ of $i$ elements,  the set $\{ \ve_J \}$ forms an orthonormal basis for $\bigwedge^i \RR^n$ and hence
$\| \sum_J a_J \ve_J  \|=\left(\sum_J a_J^2\right)^{1/2}.$ 
The inner product defined above is preserved by the action of the orthogonal group $\Or(n)$ on $\RR^n$. 

%One can define the inner product on $\bigwedge^i\QQ_p^n$ using \eqref{inner}. 
The most convenient norm to work with, for our purpose, is given by 
$\| \ov \|=\max_{J} |a_J|_p.$ 
The action of $\SL_n(\ZZ_p)$ on $\bigwedge^i \QQ_p^n$ preserves this norm. Indeed, since the entries of $k \in \SL_n(\ZZ_p)$ are $p$-adic integers, we immediately have the inequality $
  \|k \ov \|_p\leq\| \ov \|_p.$ Equality follows by applying this inequality to $k^{-1}$ instead of $k$.

\subsection{Quadratic forms}\label{subsec:quad}
In this subsection, we review the classification of quadratic forms over the $p$-adic fields $\QQ_p$ for odd primes $p$ following \cite{Se}.
Recall that a quadratic form $q$ on $\RR^n$ has signature $(r,s)$ if there is $g\in \GL_n(\RR)$ such that
  \[q(\ox)=\q(g\ox),
  \]
where $\q(\ox)=\sum_{i=1}^r x_i^2- \sum_{i=r+1}^{r+s} x_{i}^2$.  
Two quadratic forms over the field $\QQ_p$ are \emph{equivalent} if and only if they have the same rank, {}{the} same discriminant, and the same Hasse invariant. Recall that the discriminant of a non-degenerate quadratic form over a field $F$ is the image of the determinant of the Gram matrix $G$ in 
the group $F^{\ast}/(F^{\ast})^2$. This group contains $4$ elements when $F=\QQ_p$ and $p$ is an odd prime, representing the parity of $v_p(\det G)$ and whether the image of $\det G$ {}{in} the residue {}{field} is a quadratic residue or not. Hasse invariant can take two values and is {}{invariant under scalar multiplication}. It follows that there are at most $8$ different quadratic forms over $\QQ_p$ of any given dimension. A non-exceptional isotropic quadratic form $\q$ over $\QQ_p$ is called \emph{standard} if it is of the form
\[ \q(\ox) = x_1x_n+ a_2x_2^2+ \cdots + a_{n-1}x_{n-1}^2, \]
where $a_i \in \{ 1,p,u, pu \}$ and for some $j,k$,
$$-a_j/a_k \; \mathrm{is \; not \; a \; square.}$$
Here $u \in \U_p$ is a fixed element such that its image in $\ZZ_p/p\ZZ_p$ is a non-residue. Any non-exceptional isotropic quadratic form is equivalent to a standard one.

\subsection{Orthogonal groups}\label{Orthogonal groups}

For a quadratic form $\q=(\q_\p)_{p \in S}$ over $\QQ_S$, we define the orthogonal group of $\q$ by $\SO(\q)=\prod_{p\in S} \SO(\q_\p)$, where $\SO(\q_\p)$ is {}{the $p$-adic analytic group} consisting of matrices of determinant $1$ preserving the quadratic form $\q_p$. 
Let $K_\infty$ be the maximal compact subgroup {}{of $\SO(\q_\infty)$} which is isomorphic to $\SO(r)\times \SO(s)$, where $(r,s)$ is the signature of $\q_\infty$. For $p \in S_f$, we define 
$K_p=\SL_n(\ZZ_p)\cap \SO(\q_p)$. It is easy to see that $K_p$ is a maximal subgroup of $\SO(\q_p)$. Finally, we set $\SK= \prod_{p\in S}  K_\p$ and let $m_\SK, m_{K_p}$ be the Haar measures on $\SK$, $K_p$, respectively. 
Since $K_p$ is 
a subgroup of $\SL_n(\ZZ_p)$,  the $p$-adic max norm on $\QQ_p^n$ is invariant by $K_p$.
We will also need the following one-parameter subgroups. For each $p \in S_f$, define  
\begin{equation}\label{eqn:a_t}
 a_t^p=
\begin{cases}
  \diag(p^{t}, 1, \ldots, 1, p^{-t}):  \quad t\in \ZZ   & \, p \in S_f \\
   \diag(e^{-t},1, \dots, 1, e^t):  \quad t \in \RR   & \, p= \infty. 
\end{cases}  \end{equation}

We denote by $A_p$ the group consisting of all $a_t^p$ with $t \in \ZZ$ when $p \in S_f$, and $t \in \RR$ when $p= \infty$. 

Then we can define subgroups of $\SO(\q)$ for a standard quadratic form $\q$ as \begin{equation}\label{def_a_t}A = \prod A_p=\{ a_{\t} = a_{t_\infty}^\infty  \cdot \Sf a_{t_p}^p   \} = \{a_{\t} : \t \in \ST\}, \qquad A^+ = \{a_{\t} : \t \in \ST^+ \}.
\end{equation}
%$\mathsf{A}$
%Certain elements of the $S$-adic group {}{$\SA= \prod_{p \in S } A_p$} will be used later. For $p \in S$, we denote by $ \epsilon_p$ the element of $\ST$ for which $t_p=1$ and $t_q=0$ for $q \neq p$. Hence, for $p \in S_f$, 
%$a_{\epsilon_p}$ is the element of $\SA$ which is equal to $\diag\left(p, 1, \ldots, 1, p^{-1}\right)$ in place $p$ and identity in other place. 
%Similarly, $a_{\epsilon_\infty}$ the element of $\SA$ which is equal to
%$\diag\left(e^{-1}, 1, \ldots, 1, e \right)$ in place $ \infty$, and identity matrix in other places.  

%\seon{$\epsilon_p, a_{\epsilon_p}$ omitted since they are not used later.}

The following simple fact about ternary forms will be used later. 

\begin{proposition}\label{same-ortho}
Let $\q$ be a non-degenerate isotropic ternary quadratic form over $\QQ_p$. Then 
the group $\SO(\q)$ is locally isomorphic to $\PSL_2(\QQ_p)$.  
\end{proposition}

\begin{proof}
This is well-known for $\RR$, so assume $p \in S_f$. The conjugation action of $\SL_2(\QQ_p)$ on the space of $2 \times 2$ matrices of trace zero over $\QQ_p$ preserves the determinant, which can be viewed as a non-degenerate isotropic ternary quadratic form $\q_0$. Since the kernel is $\pm I$, and the corresponding Lie algebras are isomorphic, the orthogonal group of $\q_0$ is locally isomorphic to $\PSL_2(\QQ_p)$. For an
arbitrary non-degenerate isotropic ternary quadratic form $\q$ and $c \in \QQ_p \backslash \{0\}$, the form $ c \q$ has the same orthogonal group as $\q$, and for an appropriate value of $c$ has the same discriminant as $\q_0$. 
\end{proof}

The proof of the following fact is deferred to the appendix:

%%%%%%%%%%%%%%%%%%%%%%%%%%%%%%%%%%%%%%%%%%%%%%%%%%%%%%%%%%%%%%%%%%%%%
%transitivity of K.
\begin{proposition}\label{transitivity of K} For given $c_1 \in \QQ_p$ and $c_2 \in p^\ZZ$,
$ K_p$ acts transitively on $$\{\ov_p\in \QQ_p^n : \q(\ov_p)=c_1 \ \text{and} \
\| \ov_p\|_p=c_2\}.$$ 
%\seon{$C_1, C_2$ changed to $c_1, c_2$ to be compatible with section 3.2.}
\end{proposition}

%%%%%%%%%%%%%%%%%%%%%%%%%%%%%%%% S-adic geometry of numbers %%%%%%%%%
\section{The $S$-arithmetic Geometry of Numbers} 
In this section, we prove some basic results on the $S$-arithmetic geometry of numbers. Due to lack of reference in the $S$-arithmetic case, we provide proofs.
\subsection{$S$-lattices and a generalization of the $\alpha$ function}\label{alphas}
The proof of Oppenheim conjecture is based on the dynamics of unipotent flow on the space of lattices, thus we need an analogous $S$-arithmetic notion.  
\begin{definition}\label{def:3.1} Let $V$ be a free $\QQ_S$-module of rank $n$. A $\ZZ_S$-module $\Delta$ is an \emph{$S$-lattice} in $V$ if there exist $\ox_1, \ldots, \ox_n \in V$ such that
\begin{equation*}
 \Delta=\ZZ_S\ox_1\oplus \cdots \oplus \ZZ_S\ox_n
 \end{equation*}
and $V$ is generated by $\ox_1, \ldots, \ox_n$ as a $\QQ_S$-module. 
\end{definition}

Note that an $S$-lattice is a discrete and cocompact subgroup of $\QQ_S^n$. The base lattice $\Delta_0$ is the $S$-lattice generated by the canonical basis $\ve_1, \dots, \ve_n$. 
Note that $\SL_n(\QQ_S)$ acts on the set of lattices in $\QQ_S^n$ and the 
orbit of $ \Delta_0$ under this action can be identified with $\L_S:=\SL_n(\QQ_S)/\SL_n(\ZZ_S)$.  
For an $S$-lattice $\Delta$, we say that a subspace $L\subseteq \QQ_S^n$ is \emph{$\Delta$-rational} if {the} intersection $L \cap \Delta$ is an $S$-lattice in $L$, that is, $L$ is a subspace of $\QQ_S^n$ generated by finite number of elements in $\Delta$. Suppose that $L$ is $i$-dimensional and $L \cap \Delta=\ZZ_S\ov^1 \oplus \cdots \oplus \ZZ_S\ov^i$. Define $d_{\Delta}(L)$ or simply $d(L)$ by
\begin{equation}\label{eqn:d}
d(L)=d_{\Delta}(L)= \prod_{p \in S} \| \ov^1 \wedge \cdots \wedge \ov^i \|_p.
\end{equation}

The definition of $d(L)$ is independent of the choice of basis. 
This is an analogue of the volume of the quotient space $L/(L\cap \Delta)$ for lattices
in $\RR^n$. If $L=\{0\}$, we write $d(L)=1$. 
We now show that there is indeed a more intimate connection between the $\L_S$ and the space $\L$ of lattices (of covolume $1$) in $\RR^n$. More precisely, we show that the $\L_S$ fibers over $\L$ in such a way that many natural functions 
on $\L_S$ factor through the projection map. 

\begin{definition}[The projection map]
Let $ \Delta \in \L_S$. Define $ \pi( \Delta)$ to be the set of vectors $\ov \in \RR^n$ for which 
there exists $\ow \in \Delta$ with $ \ow_{ \infty}= \ov$ and $ \ow_p \in \ZZ_p^n$ for all $p \in S_f$. 
\end{definition}

\begin{remark}\label{uniqueness}
For any vector $\ov \in \pi(\Delta)$ there is exactly one vector $\ow \in \Delta$ satisfying $\ow_{ \infty}= \ov$. Indeed, if both $\ow_1$ and $\ow_2$ satisfy this property, then 
$\ow=\ow_1-\ow_2 \in \Delta$ is a non-zero vector with real component zero. This implies that $( \prod_{p \in S_f} p^m )\ow \to 0$ as $m \to \infty$, which contradicts the discreteness of $ \Delta$. 
\end{remark}

\begin{proposition}\label{fibration}
$ \pi( \Delta) $ is a lattice of covolume $1$ in $\RR^n$. 
\end{proposition}

\begin{proof}
Since the $p$-adic components of the elements of $ \pi( \Delta)$ belong to the compact sets $ \ZZ_p^n$, for 
$p \in S_f$, the vectors $\ow=\ov_{ \infty}$ cannot have an accumulation point, since then by passing to 
a subsequence, we can find an accumulation point for a sequence of vectors in $ \Delta$. 

Start with a $\ZZ_S$-basis $\ov^1, \dots, \ov^n$ for $ \Delta$. By replacing the vectors $\ov^i$ with 
certain $\ZZ_S$ linear combinations of them, we may assume that the determinant of the matrix $A_p$ with columns
$\ov^1_p, \dots, \ov^n_p$ is a unit in $\ZZ_p$, for every $p \in S_f$. Note that this is equivalent to the 
condition 
\begin{equation}\label{uni}
\| \ov^1_p \wedge \cdots \wedge \ov^n_p\|_p=1
\end{equation}
for every $p \in S_f$. We now claim that $\ov^1_{ \infty}, \dots, \ov^n_{ \infty}$ form a $\ZZ$-basis for $\pi( \Delta)$.
If $\ow \in \pi( \Delta)$, then $\ow=\ov_{ \infty}$ for some $\ov \in \Delta$ with $\ov_p \in \ZZ_p^n$. Write 
$\ov= \sum_{k=1}^{n} c_k \ov^k$, for $c_k \in \ZZ_S$, and let $\oc$ be the column vector with entries $c_1, \dots, c_n$. 
Since $A_p\oc=\ov_p$, $A_p^{-1} \in \SL_n(\ZZ_p)$, and $\ov_p \in \ZZ_p^n$, we deduce that $\oc \in \ZZ_p^n$ for all $p \in S_f$. This implies that $\oc \in \ZZ^n$, which proves our claim. 
Since $ \Delta$ is unimodular, \eqref{uni} implies that $\| \ov^1_{ \infty} \wedge \cdots \wedge \ov^n_{ \infty}\|_{ \infty}=1$, showing that $ \pi( \Delta)$ has covolume $1$. 
\end{proof}

\begin{definition}\label{alpha} For an $S$-lattice $\Delta$, define
\[\alpha_i^S(\Delta)=\sup\left\{\frac{1}{d(L)} : L \text{ is a } \Delta\text{-rational subspace of dimension }  i \right\},
\]
where $1\leq i \leq n$ and
\[\alpha^S(\Delta)=\max\left\{\alpha_i(\Delta) : 0\leq i \leq n \right\}.
\]
\end{definition}
The function $ \alpha^S$ is not only an analogue of the original function $ \alpha$, but is in fact intimately related
to it. 
\begin{proposition}\label{alphas}
For any $ \Delta \in \L_S$, we have 
\[ \alpha^S( \Delta)= \alpha(\pi( \Delta) ). \]
\end{proposition}

\begin{proof} Assume that $ \alpha_i^S( \Delta)= 1/d(L)$
where $L$ is the $\ZZ_S$-span of $ \ov^1, \dots, \ov^i$. Similarly to the proof of Proposition \ref{fibration}, after 
possibly replacing $\ov^1, \dots, \ov^i$ by another basis, we can assume that $\|\ov_p^1 \wedge \cdots \wedge \ov_p^i \|_p=1$
for all $ p \in S_f$. This implies that  $d(L)= \| \ov_{ \infty} ^1 \wedge \cdots \wedge \ov_{ \infty}^i  \|_{ \infty}$. 
As before, $ \ov_{ \infty} ^1, \dots,  \ov_{ \infty}^i$ form a $\ZZ$-basis for $\pi(L)$, implying that $\alpha_i(\pi(\Delta) )
\ge \alpha_i^S( \Delta)$. The converse can be proven similarly.
\end{proof}
Since $S$ is fixed throughout the discussion, we will henceforth drop the superscript $S$ in $\alpha_i^S$ and $ \alpha^S$. We now define an $S$-adic Siegel transform. 

\begin{definition}\label{def_Siegel}
Let $f:\QQ_S^n \to \RR$ be a bounded function vanishing
outside a compact set. The \emph{$S$-adic Siegel transform} of $f$ is a function on $\SG/\Gamma$ defined as follows: for $g\in \SL_n(\QQ_S)$,
$$
 \tilde{f}(g)=\sum_{\ov\in \ZZ_S^n} f(g\ov).
$$
 More generally, $\tilde{f}$ is defined on the space of $S$-lattices $\Delta$ as
 $
 \tilde{f}(\Delta)=\underset{\ov\in \Delta}{\sum} f(\ov).
 $
\end{definition} 
 
Since the diagonal embedding of $\ZZ_S^n$ in $\QQ_S$ is discrete, the defining sum involves only finitely many non-zero terms. The following lemma is a natural extension of Schmidt's theorem (\cite{Sch}, Lemma 2) in geometry of numbers. 
%We will supply two proofs: the first one is self-contained proof, while the second one 
%reduces the statement to the one for lattices in $\RR^n$. 

\begin{lemma}{(S-adic Schmidt lemma)}\label{SSchmidt}  Let $f:\QQ_S^n \to \RR$ be a bounded function vanishing outside a compact subset. Then there exists a positive constant $c=c(f)$ such that
$\tilde{f}(\Delta)<c\alpha^S(\Delta)$ for any $S$-lattice $\Delta$ in $\QQ_S^n$.
\end{lemma}
%\seon{First proof omitted.}

%\begin{proof}[First Proof]
%It suffices to prove the lemma for $f={\mathbf 1}_C$, where $C \subseteq \QQ_p^n$ is a compact set. Since 
%For $\ov, \ow \in \QQ_S^n$, define
%\begin{equation*}  
%\ell(\ov,\ow):= ||\ov||\card \left(C \cap(\ow+\ZZ_S.\ov)\right). 
%  \end{equation*}
%Since $C$ is compact and $\ZZ_S$ is discrete, we have $\ell(\ov, \ow)< \infty$. 
%A compactness argument shows that 
%   \[\ell :=\max_{\ov, \ow \in \QQ_p^n} \ell(\ov, \ow)<\infty.
%   \]
%Roughly speaking, the quantity $\ell$ is an analogue of the diameter of $C$ in the real case.  
%For any $S$-lattice $\Delta$, let $\{\ov_1, \ldots, \ov_n\}$ be a $\ZZ_S$-basis of $\Delta$, ordered in  such a way that $||\ov_j||<1$ for $1\leq j\leq i$ and $||\ov_j||\geq1$ for $i+1\leq j\leq n$. Since the number of $S$-(sub)lattice points are inversely proportional to the volume of the (sub)lattice parallelepiped, by definitions of $\alpha_i$, $\alpha$ and $\ell$, we have
%
%  \begin{align*}
%\tilde{f}(\Delta)&  = \card \left(\Delta\cap C \right)\\
%& \le \max_{\ow\in\Delta}\card \left(\ow+\ZZ_S\ov_1\oplus\cdots
%\oplus\ZZ_S\ov_i\cap C \right)\\
%&\hspace{2cm}\cdot\max_{\ow\in\Delta}\card \left(\ow+\ZZ_S\ov_{i+1}\oplus\cdots \oplus\ZZ_S\ov_n\cap C\right)\\
%& \le \left(\ell^i\alpha_i(\Delta)\right)\cdot \ell^{n-i}   \ell^n\alpha(\Delta),
%\end{align*}
%\end{proof}
\begin{proof}
%[Second Proof]
Without loss of generality, we can assume that $f$ is the characteristic function
of $B= \prod_{p \in S} B_p$, where $B_p$ is a ball of radius $p^{n_p}$, for $p \in S_f$ and a ball of radius $R$ for $p = \infty$. Set $D= \prod_{p \in S_f} p^{n_p}$ and 
consider the {}{scalar} map $ \delta: \Delta \to \Delta$ defined by $ \delta(\ox)= D \cdot \ox$. Note that if 
$\ox \in  \Delta \cap B$, then $ \| \ox \|_p \le 1$ for all $ p \in S_f$ and
$\| \ox \|_{ \infty} \le RD$. In particular, $\pi( \delta( \ox ) )$ will be inside
a ball $B'$ of radius $RD$. Note that Schmidt's lemma states that 
\[ \card ( \Delta' \cap  B') =O( \alpha( \Delta')), \]
holds for any lattice $ \Delta'$ in $\RR^n$, where the implies constant depends on $B'$. 
Using this fact and Remark \ref{uniqueness} we obtain 
\[ \card ( \Delta \cap B) \le \card ( \pi( \Delta) \cap B') =O( \alpha( \pi( \Delta) )= O(\alpha^S( \Delta)). \]

%
%
% of the p. It suffices to show that the existence of a constant $c_n>0$ (only depending on $n$) such that for all $R>0$ and
%$m \in \ZZ$ and for any lattice $ \Delta \subseteq \RR^n \times \QQ_p^n$, we have
%\[ |  \Delta \cap \pball{R,m}| \le C  \max(1, (Rp^m)^n) \alpha( \Delta).\]
%For $m \ge 1$ write
%\[ A_{R,m}= \{(x,y) \in  \Delta: |x| \le R, |y|_p = p^m \}. \]
%also define
%\[ A_{R,0}=  \{(x,y) \in  \Delta: |x| \le R, |y|_p \le 1\}. \]
%Note that $\pball{R,m}= \bigcup_{j=0}^{m} A_{R,m}$.
%As $ \Delta$ is a $\ZZ[1/p]$-module, the dilation map $ \delta_m: \Delta \to \Delta$ defined by
% $ \delta_j(x,y)=(p^j x, p^j y)$ is an automorphism of $ \Delta$. Moreover, it is easy to see that for $j \ge 1$,
%  \[  \delta_j( A_{R,j})= A_{p^j R, 1}. \]
%The projection map $\prj_1$ sends $A_{p^j R, 1}$ into $ \pi(  \Delta) \cap \{ x \in \RR^n: |x|\le p^jR \}$. By the infectivity of this map, we have
%\[ |A_{p^j R, 1}| \le | \pi(  \Delta) \cap \{ x \in \RR^n: |x|\le (p^jR)^n \}| \le c_n \max(1,(p^j R)^n) \alpha( \pi( \Delta) )= C_n
%\max(1,(p^jR)^n) \tilde{\alpha}( \Delta). \]
%By a similar argument, $\prj_1$ maps $A_{R,0}$ injectively into $  \pi(  \Delta) \cap \{ x \in \RR^n: |x|\le R \}$. From these, we obtain
%\[|  \{(x,y) \in  \Delta: |x| \le R, |y|_p \le p^m \}| \le c_n \max \left[1,\left( \sum_{j=0}^{m}(p^jR)^n \right) \right] \alpha( \Delta). \]
%which proves the claim.

\end{proof}

\begin{lemma}\label{lem 5:06}
Let $\Delta$ be an $S$-lattice in $\QQ_S^n$ and let $L$ and $M$ be two $\Delta$-rational subspaces. Then we have that
\begin{equation}d(L)d(M)\geq d(L\cap M)d(L+M),\label{eq5:06}
\end{equation}
where $d=d_{\Delta}$ is defined in \eqref{eqn:d}.
\end{lemma}
\begin{proof} We first verify that if $L$ and $M$ are $\Delta$-rational, then $L\cap M$ and $L+M$ are also $\Delta$-rational. Since a subspace of $\QQ_S^n$ is $\Delta$-rational if and only if it is generated by elements of $\Delta$, $L+M$ is clearly a $\Delta$-rational subspace of $\QQ_S^n$. On the other hand, consider a projection map $\pi : \QQ_S^n \rightarrow \QQ_S^n/\Delta$. Note that the projection image $\pi(H)$ of a subspace $H$ in $\QQ_S^n$ is closed if and only if $H \cap \Delta$ is a lattice subgroup in $H$, that is, $H$ is $\Delta$-rational. Since $\pi$ is proper, if we set $H=\pi^{-1}(\overline{\pi(L \cap M)})=\pi^{-1}(\overline{\pi(L)\cap\pi(M)})$, we can easily check that $H$ is $\Delta$-rational and $H=L\cap M$.

  Let $p : \QQ_S^n \rightarrow \QQ_S^n/(L\cap M)$. Since $d_{\Delta}(H)=d_{p(\Delta)}(p(H))d_{\Delta}(L \cap M)$ for any $\Delta$-rational subspace $H$, the inequality \eqref{eq5:06} is equivalent to
  \[d_{p(\Delta)}(L)d_{p(\Delta)}(M)\geq d_{p(\Delta)}(L+M).
  \]
  Let $\{\ov_1, \ldots, \ov_{\ell}\}$ and $\{\ow_1, \ldots, \ow_{m}\}$ be bases of $p(L)$ and $p(M)$ consisting of elements in $p(\Delta)$, respectively. Then since $(p(L)\cap p(\Delta))+(p(M)\cap p(\Delta)) \subseteq p(L+M)\cap p(\Delta)$,
\begin{align*}
d_{p(\Delta)}(L)d_{p(\Delta)}(M)&=\|\ov_1\wedge \cdots \wedge\ov_{\ell}\| \|\ow_1\wedge\cdots\wedge\ow_m\|\\
&\geq\|\ov_1\wedge\cdots\wedge\ov_{\ell}\wedge
\ow_1\wedge\cdots\wedge\ow_m\|\geq d_{p(\Delta)}(L+M).
\end{align*}
\end{proof}

The following Lemma is the $S$-arithmetic version of Lemma 3.10 in \cite{EMM}. 
\begin{lemma}\label{integrable}
The function $ \alpha$ belongs to $L^r( \SL_n(\QQ_S)/\SL_n(\ZZ_S))$ for all values 
of $1 \le r< n$. 
\end{lemma}

\begin{proof}
Consider the map $ \pi: \L_S \to \L$. Denote the invariant probability measures on $ \L_S$ and $ \L$ by $\mu_S$ and $\mu$, respectively. We claim that $ \pi$ is measure-preserving, that is, for any measurable
$A \subseteq \L_{\RR}$, we have $\mu_S(\pi^{-1}(A))= \mu(A)$. To see this, define $ { \nu}(A)= \mu_S(\pi^{-1}(A))$. One can easily see that ${  \nu}$ is an $\SL_n(\RR)$-invariant probability measure on $\L$, implying that $ \nu= \mu$. The result will immediately follows from Proposition \ref{alphas}. 
\end{proof}

We can now prove the following generalization of the classical 
Siegel's formula in geometry of numbers. 
\begin{proposition}\label{Siegel integral formula}($S$-arithmetic Siegel integral formula)
 Let $f$ be a continuous function with compact support on $\QQ_S^n$,
 $  \SG= \SL_n(\QQ_S)$ and $ \Gamma= \SL_n(\ZZ_S)$. Then
 \[ \int_{\SG/ \Gamma} \widetilde{ f} d\mu = \int_{\QQ_S^n} f d\lambda_S^n + f(0). \]
\end{proposition}

The proof is based on the following lemma. For each $p \in S$, we denote by $\leb^n_p$ the Haar measure on $\QQ^n_p$ normalized such that $ \leb^n_p(\ZZ^n_p)=1$. Denote by $\delta_p$ the delta measure at the zero point of $\QQ_p^n$. 

\begin{lemma}\label{measures}
Let $\nu$ be an $\SL_n(\QQ_S)$-invariant Radon measure on $\QQ_S^n$. 
Then $\nu$ is a linear combination of measures of the form $ \bigotimes_{p \in S} \nu_p$, where each 
$\nu_p$ is either the Haar measure  $\leb^n_{p}$ or the delta measure $\delta_{p}$. 
\end{lemma}

\begin{proof} Observe that if $\nu$ is a positive Radon measure on a locally compact space $X$ invariant under the action of a group $H$, and $H$ has a finite number of orbits, namely $X_1, \dots, X_m$, then $\nu$ naturally decomposes into the sum 
$ \nu=\sum_{j=1}^{m} \nu_i$, where $\nu_i(A)= \nu(A \cap X_i)$ is $H$-invariant and supported on $X_i$, for $1 \le i \le m$. Let us now consider the case at hand, 
namely, the action of $\SL_n(\QQ_S)$ on $\QQ_S^n$. Since the natural action of $\SG=\SL_n(\QQ_p)$ on $\QQ_p^n$ has two orbits on $\QQ_p^n$, the $\SG$-orbits in the action on $\QQ_S^n$ are 
of the form $ \prod_{p \in S} X_p$, where $X_p$ is either the zero point or $ \QQ_p^n - \{ 0 \}$. Each orbit is a homogenous space $\SG/L$, where $L$ is the stabilizer of an arbitrary point in the respective orbit. The invariant measure on the homogeneous space $\SG/L$, if exists, is unique up to a factor. The result follows. 
\end{proof}

\begin{proof}[Proof of Proposition \ref{Siegel integral formula}]
Note that by Lemmas \ref{SSchmidt} and \ref{integrable}, $ \widetilde{ f}$ is integrable. This implies that the
map $f \mapsto \int_{\SG/ \Gamma} \widetilde{  f} d\mu$ defines a positive linear functional on the space of 
compactly supported continuous functions on $\QQ_S^n$, hence defines a Radon measure $\nu$ on $\QQ_S^n$. 
Since $\mu$ is $\SG$-invariant, so is $\nu$. Using Lemma~\ref{measures}, $\nu$ splits 
as a linear combination of measures of the form $ \bigotimes_{p \in S} \nu_p$, where each $\nu_p$ is either $\leb^n_{p}$ or the delta measure $ \delta_p$. The coefficient of 
$\bigotimes_{p \in S} \leb^n_{p}$ in this linear combination is one. Indeed, taking an increasing sequence of compactly supported non-negative continuous functions $f_n$ approximating the normalized characteristic functions of increasingly larger balls in $\QQ_S^n$, one obtains
$ \widetilde{ f_n} \to 1 $, and $  \int_{\QQ_S^n} f_n d\lambda_S^n \to 1$. 
Similarly, by choosing a family $f_n$ of non-negative continuous functions with support in a ball of radius $1/n$ with $|f_n| \le 1$ and $f_n(0)=1$, and applying the Lebesgue's dominated convergence theorem, we have
\[ \lim_{n \to \infty} \int_{\SG/ \Gamma} \widetilde{ f_n} d\mu = 1, \]
which implies that the coefficient of $f(0)$ is $1$. It remains to be shown that the rest of the coefficients are zero. This can be obtained in a similar way from the fact that 
a lattice $ \Delta$ in $\QQ_S^n$ cannot contain a non-zero point $(\ov_p)_{p \in S}$ with $\ov_{p_0}=0$ for some $p_0 \in S$. Otherwise, one can find an appropriate sequence
$n_i$ of $S$-integers, such that $|n_i|_q \to 0 $ for all $q \neq p_0$ and hence $n_i\ov \to 0$ contradicting the discreteness of $ \Delta$. From here, and by choosing appropriate sequences of continuous functions as above, we can deduce that the rest of the coefficients are zero. 
\end{proof}

\subsection{Integration on subvarieties}

  Now let us introduce volume forms of submanifolds in $\QQ_S^n$ inherited from the volume form $d\lambda_S^n$ of $\QQ_S^n$ in order to measure the volumes of orbits of maximal compact subgroups of orthogonal groups in $\QQ_S^n$. Since such maximal compact subgroups are products of maximal compact subgroups at each place, it suffices to deal with volume forms of $p$-adic submanifolds in $\QQ_p^n$. 
  
Let us introduce two equivalent definitions of volume form on submanifolds. Recall that any orbit of the $p$-adic maximal compact subgroup $K_p=\SL_n(\ZZ_p)\cap \SO(\q_p)$ of $\SO(\q_p)$ for a given quadratic form $\q_p$ on $\QQ_p^n$ is of the form 
  \[\{ \ow \in \QQ_p^n : ||\ow||_p=p^{c_1}, \q_p(\ow)=c_2\},
  \]
  where $c_1 \in \ZZ$ and $c_2 \in \QQ_p$ are some constants. Note that the above $K_p$-orbit is an open subset of the $(n-1)$-dimensional $p$-adic variety in $\QQ_p^n$ given by
  \[\{\ow \in \QQ_p^n : \q_p(\ow)=c_2 \}.
  \]
  
  Hence from now on we will consider $Y$ as a compact open subset of a $d$-dimensional $p$-adic variety in $\QQ_p^n$ or a parallelepiped $\ZZ_p \ov_1 \oplus \cdots \oplus \ZZ_p \ov_d$ where $\ov_1, \ldots, \ov_d \in \QQ_p^n$. Then $Y$ is contained in $p^{-r}\ZZ_p^n$ for some $r\in \NN$.

\begin{def}\label{subvol1} Consider the projection $\pi_{\ell} : p^{-r}\ZZ_p^n \rightarrow p^{-r}\ZZ_p^n/p^{\ell}\ZZ_p^n$ and let $Y_{\ell}=\pi_{\ell}(Y)$. Then the volume form $\nu_d$ defined over $Y$ is defined by
\begin{equation*}
\nu_d(Y)=\lim_{\ell \rightarrow \infty} \frac {\# Y_{\ell}}{p^{d\ell}}.
\end{equation*}
\end{def}

  As in the real case, we can define the volume form of a $p$-adic submanifold by describing the volume form of its tangent space.
\begin{def}\label{subvol2} Let $Y$ be a parallelepiped $\ZZ_p \ov_1 \oplus \cdots \oplus \ZZ_p \ov_d$, where $\ov_1, \ldots, \ov_d$ are linearly independent in $\QQ_p^n$. Then the volume form $\nu'_d$ of the parallelepiped is defined as
\begin{equation*}\nu'_d(\ZZ_p \ov_1 \oplus \cdots \oplus \ZZ_p \ov_d)=|| \ov_1 \wedge \cdots \wedge \ov_d ||_p,
\end{equation*}
where $|| \cdot ||$ is the maximum norm of $\bigwedge^d(\QQ_p^n)$.
\end{def}
  
We will use the fact (see \cite[Theorem 9]{Se3}) that the normalized measures $\nu_d$ and $\nu'_d$ are equal.

\section{Passage to homogeneous dynamics}\label{passage}
 In this section, we define a $p$-adic analogue of $J_f$ function introduced in Section 3 in \cite{EMM} for the real case. This function is used in Section \ref{sec:counting} to relate the counting problem to the asymptotic formula {of} certain integrals on the space of $S$-lattices. Fix a prime $p \in S_f$, and let $\q$ be a standard quadratic form on $\QQ_p^n$. 
%\[ \q(x)= x_1x_n+ a_2x_2^2+ \cdots  + a_{n-1}x_{n-1}^2 \]
%where $a_2, \dots, a_{n-1} \in \QQ_p$. \seon{$\q$ is already defined.} We will assume that $a_i \in \{ 1, p, u, pu \}$,
%where $u \in \ZZ_p$ is fixed element whose image in $\ZZ_p/p\ZZ_p$ is a quadratic non-residue.  Let $a_t$ be defined as above. 
%Consider the orthogonal subgroup $ \SH=\SO(\q)$ of $\q$ and let $ \SK$ be the maximal compact subgroup $\left(\SO(n)\times\Sf\SL_n(\ZZ_p)\right)\cap  \SH$.
%In the following two propositions, we denote $\q_p$, $ K_p$ and $a(t_p)\in H_p=SO(\q_p)$ by $\q$, $ K$ and $a_t\in  H$ respectively.\\
Let $\pi_1 : \QQ_p^n \rightarrow \QQ_p$ be the projection to the first coordinate and $\QQ^n_{p,+}$ be the set $\{\ox\in \QQ_p^n : \pi_1(\ox)\neq 0 \}.$ We denote by $p^{\ZZ}$ the subset of $\QQ_p$ consisting of powers of $p$.

%%%%%%%%%%%%%%%%%%%%%%%%%%%%%%%%%%%%%%%%%%%%%%%%%%%%%%%%%%%%%%%%%%%%%%%%%%%%
\begin{lemma}\label{L:JF} Let $f$ be a continuous function with compact support on
$\QQ_{p,+}^n$ which satisfies the invariance property
\begin{equation}\tag{\textcolor{blue}{${\mathbf I_f}$}}
f(ux_1, x_2, \ldots,x_{n-1}, u^{-1}x_n)=f(x_1, x_2, \ldots, x_n)\label{IF}
\end{equation}
for all units $u\in \U_p$. Let $\nu$ be a non-negative continuous
function on the unit sphere $\U_p^n$ such that
\begin{equation}\tag{\textcolor{blue}{${\mathbf I_\nu}$}}
\nu(u\ox)=\nu(\ox),\label{IV}
\end{equation} for all $u\in\mathcal \U_p$.
Let $J_f$ be the real-valued function on $p^{\ZZ}\times\QQ_p \subset \QQ_p^2$ defined by
\begin{equation}
J_f(p^{-r},\zeta)=\frac 1 {p^{r(n-2)}}\int_{\QQ_p^{n-2}} f(p^{-r},
x_2, \ldots, x_{n-1}, x_n) \ dx_2\cdots
dx_{n-1},\label{JFE}
\end{equation}
where in the integral, $x_n=p^{r}\left(\zeta - \q(0, x_2,\ldots,
x_{n-1},0)\right)$. Then for any $\epsilon>0$, if $t$ and $\| \ov \|$ are
sufficiently large, we have
\begin{equation}
\left|c(K_p)p^{t(n-2)}\int_{ K_p}
f(a_tk\ov)\nu(k^{-1}\ve_1)dm(k) - J_f\left(
{p^t}{\|\ov\|^{\c}},\zeta\right)\nu(\frac{\ov}{\|\ov\|^{\c}})\right|<\epsilon,\label{Jf
ineq}
\end{equation}
where $c(K_p)=\vol(K_p.\ve_1)/(1-1/p)$ and $m$ is the normalized Haar measure on $ K_p$.
\end{lemma}
\begin{proof} Let $\pi_i$ denote the projection onto the $i$th coordinate in $\QQ_p^n$, and $a_t:=a^p_t=
{\diag}\left(p^t,1, \ldots, 1, p^{-t}\right)$ as before. Since $f$ has compact support, for $t \gg 1$, if $f(a_t\ow)\neq0$, we have  
\begin{equation*}\|\ow\|_p=|\pi_1(\ow)|_p.\label{eq3:11}
\end{equation*}
This implies that on the support of the function $k \mapsto f(a_tk\ov)$, for $t$ large enough we have
\begin{equation}\label{eqn:at}
 a_tk\ov = \left({\|\ov\|_p^{\c}}p^t\cdot u(k\ov), \pi_2(k\ov), \ldots, \pi_{n-1}(k\ov), \ov_n u(k\ov)^{-1}\right). 
\end{equation}

Using \eqref{IF} we have
  \begin{equation*}f(a_tk\ov) = f\left( {\|\ov\|_p^\c}\;p^t, \pi_2(k\ov), \ldots, \pi_{n-1}(k\ov),
  v_n\right).
  \end{equation*}
where $v_n$ is determined by
  $  \q\left( {\|\ov\|_p^\c}\;p^t, \pi_2(k\ov), \ldots, \pi_{n-1}(k\ov), v_n
  \right)=\q(\ov)$. 
 Since $K_p$ preserves the norm, given $\delta>0$, for $t>t_0( \delta)$, if $f(a_t k\ov)\neq0$, we have $$\left| {k\ov}/{\|\ov\|_p^{\c}} - u(\pi_1(k\ov))\ve_1\right|_p<\delta$$ for a unit
  $u(\pi_1(k\ov))\in\mathcal \U_p$. This implies that,
   \begin{equation*}
  \left| \frac {\ov}{\| \ov\|_p^\c} - k^{-1}(u(\pi_1(k\ov))\ve_1)\right|_p= \left| \frac {\ov} {\|\ov\|_p^\c} -
  u(\pi_1(k\ov))(k^{-1}\ve_1)\right|_p<\delta. 
  \end{equation*}

  Since $\nu$ is continuous and satisfies \eqref{IV}, by choosing $\delta>0$ small enough we can assure that if $t>t_0( \delta)$, and $f(a_tk \ov) \neq 0$, then the following holds:
  \begin{equation}\left|\nu(k^{-1}\ve_1) -
  \nu(\frac {\ov}{\|\ov\|_p^\c})\right|<\epsilon. \label{eq3:12}
  \end{equation}

For sufficiently large $t$, if $f(a_t kv) \neq 0$, then the vector
$ \left(  {\|\ov\|_p^\c}\;p^t, \pi_2(k\ov), \ldots, \pi_{n-1}(k\ov),
  v_n\right)$ will be at a bounded distance from {}{$\ve_1$.}

  Since the $ K_p$-invariant measure on the orbit $ K_p.\ov$ is the push-forward of the normalized Haar measure on $K$, one can approximate the integral in \eqref{JFE} by the integration on the $K_p$-orbit $K_p.(
  {\|\ov\|_p^\c}\;{\ve_1})\subset\QQ_p^n$. For
  $t>\max\{t',t''\}$,
\begin{align}
&\left|\int_{ K_p} f(a_tk\ov)\nu(k^{-1}\ve_1)dm(k)\right.\label{lem3:6eq11}\\
&\hspace{1.5cm}\left.-\int_{ K_p.(\|\ov\|_p^\c\;\ve_1)}f\left( {\|\ov\|_p^{\c}}\;p^t,
\pi_2(k\ov),\ldots,\pi_{n-1}(k\ov),v_n\right)\nu(\frac {\ov}{\|\ov\|^\c})d\lambda_{n-1}\right|_p\leq \epsilon,\nonumber
\end{align}
where $\lambda_{n-1}$ is the normalized Haar measure on $ K.(\|\ov\|^{\c}_p\;\ve_1)$.
\end{proof}
%%%%%%%%%%%%%%%%%%%%%%%%%%%%%%%%%%%%%%%%%%%%%%%%%%%%%%%%%%%%%%%%%%%%%%%%%%%%
We remark that \eqref{IV} is an essential condition but \eqref{IF} is not a real restriction since we can easily modify the given function to satisfy \eqref{IF}.

%%%%%%%%%%%%%%%%%%%%%%%%%%%%%%%%%%%%%%%%%%%%%%%%%%%%%%%%%%%%%%%%%%%%%%%%%%%%
%lemma 3:08

\begin{proposition} \label{L:h}
Suppose $h$ is a real-valued continuous function with compact support defined on $(\QQ_p^n - \{ 0 \}) \times \QQ_p$. Then
\begin{equation}\label{for:h}
  \lim_{t \to \infty} \frac{1}{p^{(n-2)t}} \int_{ \QQ_p^n} h(p^t\ov, \q(\ov)) d\ov \nonumber 
 =
\vol( K\cdot\ve_1)  \sum_{z\in \ZZ} p^{(n-2)z} \int_{ K} \int_{\QQ_p} h(p^{-z} k \ve_1, \zeta)  d\zeta dm(k). 
\end{equation}
where $\vol=\vol_{n-1}$ denote the $(n-1)$-dimensional volume induced from $\QQ_p^n$.
%\[ c_Q= \mu_{n-1}( \{ v \in \mathcal U_p^n: Q_0(u)=0 \} \]
%is the $(n-1)$-dimensional volume of the intersection of the light cone
%with the set of vectors of norm one.
\end{proposition}
\begin{proof}
%Recall that
%the restriction of the Haar measure $\mu$ of $\QQ_p$ to $\ZZ_p$ can also be defined as the inverse limit of a counting
%measure on $\ZZ/p^k\ZZ$. More precisely, for a compact set $Y\in \ZZ_p$, we have
%\[\mu(Y)=\lim_{k\rightarrow \infty}\frac{|Y_k|}{p^{k}},
%\]
 %where $Y_k$ is the image of $Y$ under the projection from $\ZZ_p$ to $\ZZ_p/p^k\ZZ_p$. 

Let us make the substitution
$\ow= p^t\ov$. This gives $ d\ow= 1/p^{nt} d\ov$, hence we have
\begin{equation*}
\frac{1}{p^{(n-2)t}} \int_{ \QQ_p^n} h(p^t\ov, \q(\ov)) d\ov =   \frac{1}{p^{(n-2)t}} \int_{\QQ_p^n}  h(\ow, p^{-2t} \q(\ow))
p^{nt} d\ow\nonumber 
 =
p^{2t}  \int_{\QQ_p^n}  h(\ow, p^{-2t} \q(\ow)) d\ow. \label{lem3:08eq01}\end{equation*}
%Of course, it suffices to replace $  \varpi^{2T}$ by $  \varpi^{T}$
%and consider the integrals
Replacing $p^{-2t}$ by $p^{-t}$, the left integral in \eqref{for:h} can be changed to
\[ L_1(h):=\lim_{t \to \infty} p^{t}  \int_{\QQ_p^n}  h(\ow, p^{-t} \q(\ow)) d\ow. \]
With the decomposition $\QQ_p^n= \bigcup_{z \in \ZZ} p^{-z} \mathcal \U_p^n$, write 
\begin{equation}\label{decomp} h= \sum_{z \in \ZZ} h_z,
\end{equation}
 where $h_z(\ov,x)= h(\ov,x) \1_{p^{-z}\U_p^n}(\ov)$. Since $h$ has compact support, all but finitely many of $h_z$ are zero.
We first claim that for $h$ with $\supp h \subseteq \mathcal \U_p^n \times \QQ_p$,
\begin{equation*}
\lim_{t \to \infty} \frac{1}{p^{(n-2)t}} \int_{ \QQ_p^n} h(p^t\ov, \q(\ov)) d\ov=
\vol( K.\ve_1)    \int_{ K} \int_{\QQ_p} h(k \ve_1, \zeta)  d\zeta dm(k).
 \end{equation*}
Then, for $h$ with $\supp h \subseteq p^{-z_0}\mathcal \U_p^n$ where $z_0 \in \ZZ$, define $h'(\ov,x)=h( p^{-z_0}\ov, p^{-2z_0}x)$, which has
compact support in $\mathcal \U_p^n \times \QQ_p$. By a change of variables, we have
%\begin{equation}
\begin{align*}
 L_1(h)& =\lim_{t \to \infty} p^t \int_{\QQ_p^n}  h(\ov, p^{-t} \q(\ov) ) d\ov  =\lim_{t \to \infty} p^t \int_{ p^{-z_0} \mathcal \U_p^n} h(\ov, p^{-t}\q(\ov)) d\ov  \nonumber \\
&= p^{n z_0} \lim_{t \to \infty} p^t  \int_{\mathcal \U_p^n}  h(p^{-z_0} \ov, p^{-t- 2 z_0} \q(
\ov)) d\ov \nonumber\\
&= p^{n z_0} \lim_{t \to \infty} p^t \int_{\mathcal \U_p^n} h'( \ov, \q(\ov)) d\ov
 = \vol( K.\ve_1)    \int_{ K} \int_{\QQ_p} h'(  k \ve_1, \zeta)  d\zeta dm(k)\nonumber\\
&=\vol( K.\ve_1) p^{n z_0}   \int_{ K} \int_{\QQ_p} h(   p^{- z_0} k \ve_1, p^{-2 z_0}\zeta)  d\zeta dm(k) \nonumber \\
&= \vol( K.\ve_1)   p^{(n-2) z_0} \int_{ K} \int_{\QQ_p} h(   p^{- z_0} k\ve_1,  \zeta') d \zeta' dm(k),
\end{align*}
%\end{equation}
where the last equality relies on the change of variables $ \zeta'= p^{-2 z_0} \zeta$.
Now, using the decomposition \eqref{decomp}, Proposition \ref{L:h} follows for an arbitrary $h$.\\

Now let us prove the claim. Let $h$ satisfy $\supp h \subseteq \mathcal \U_p^n \times \QQ_p$. Furthermore, since we put $\q(\ov)$ in the second variable and $\q$ is the quadratic form with $p$-adic integral coefficients, we can add the assumption that $\supp h\subseteq \mathcal \U_p^n\times \ZZ_p$.

Let us denote by ${\mathcal A}_k$ the subspace of the space of continuous functions on $ \mathcal \U_p^n \times \ZZ_p$ which is spanned by the functions of the form $h(\ov,x)=f(\ov)g(x)$, with $f(\ov)=\1_A( \ov)$ and $g(x)=\1_B(x)$, where
$A \subseteq   \U_p^n$ and $B \subseteq  \ZZ_p$ are respectively the preimages
of subsets $ \bar{ A} \subseteq (\ZZ/p^k\ZZ)^n$ and $ \bar{  B} \subseteq \ZZ/p^k\ZZ$ by the reduction map $\red_k$. We will also
set $  {\mathcal A}= \bigcup_{k=1}^{ \infty} {\mathcal A}_k$. In view of Stone-Weierstrass theorem for $p$-adic fields
(\cite{Dieudonne}), ${\mathcal A}$ is dense in the space of all continuous functions on $  \U_p^n \times \ZZ_p$. Let us denote the reduction maps $\ZZ_p \to \ZZ_p/p^k \ZZ_p$ (and also $\ZZ_p^n \to (\ZZ_p/p^k \ZZ_p)^n$) by $\red_k$.
Let $h$ be a real-valued continuous function with compact support contained in $ \U_p^n \times \ZZ_p$.
We will do some reduction steps:
we may assume that $h(\ov,x)=f(\ov)g(x)$, where $f(\ov)=\1_A( \ov)$ and $g(x)= \1_B(x)$, with
\begin{equation}\label{level}
\begin{split}
 A& = \red_k^{-1}( \overline{ A}) \subseteq \U_p^n,  \\
 B &= \red_k^{-1} ( \overline{B}) \subseteq \ZZ_p.
\end{split}
\end{equation}

In other words,
\[ \ov \in A, \ow \in \ZZ_p^n \ \text{with} \ \|\ow \|< p^{-k} \implies \ov+\ow \in A, \]
\[ x \in B, y \in \ZZ_p \ \text{with} \  |y | < p^{-k} \implies x+y \in B. \]

It thus suffices to show that
\begin{align*} L_1(h)
&= \lim_{t \to \infty} p^t  \mu_n( \{\ov\in A, \q(\ov) \in p^t B)\nonumber\\
&=\vol( K.\ve_1) m_{ K}( \{ k \in  K: k\ve_1 \in A \}) \mu_1(B).  \end{align*}

Let us define
\begin{align*}
C_t(A,B)&=\{\ov \in \ZZ_p^n: \ov \in A, p^{-t} \q(\ov) \in   B \},\\
 C'_t(A,B)&=\{ \ou + p^t\ow: \ou \in A,  \q(\ou)=0, \ow \in \ZZ_p^n, 2\beta(\ou,\ow) \in B \}.
\end{align*}

Here $\beta$ is the associated bilinear form to the quadratic form $\q$. Recall the polarization identity:
\[ \q(\ov+\ow)=\q(\ov)+ 2\beta(\ov,\ow)+ \q(\ow). \]
We claim that $C_t(A,B)=C'_t(A,B)$ for any $t>k$. First, we show that $C'_t(A,B ) \subseteq C_t(A,B)$. Assume that $\ov=\ou+p^t\ow$ with $\ou \in A, \ow \in \ZZ_p^n$ satisfying $\q(\ou)=0,  2\beta(\ou,\ow) \in B$. This implies that
\[ \q(\ou+p^t\ow)= 2p^t\beta(\ou,\ow)+ p^{2t} \q(\ow)^2= p^t( 2 \beta(\ou,\ow)+  p^t \q(\ow)^2) \]
Since $\ou \in A$ and $\| p^t\ov\| \le p^{-t}$, for $t>k$, we have
$\ov \in A$. Moreover,
\[ p^{-t} \q(\ov)=  2\beta(\ou,\ow)+ p^t \q(\ow)^2 \in
B. \]
Conversely, if $\ov \in C_t(A,B)$, set
$\ov'= \ve_1+  \frac{1}{2}\q(\ov)\ve_n$. 
Note that $\q(\ov')= \q(\ov)$ and $\|\ov \|= \|\ov'\|=1$. Hence by Proposition \ref{transitivity of K}, there exists
$k \in  K$ such that $\ov=k\ov'$. Set $\ou=k\ve_1$ and $\ow= \frac{1}{2}p^{-t} \q(\ov) k\ve_n$, so that $\ov$ decomposes as
$\ov= k(\ve_1+  \frac{1}{2}\q(\ov)\ve_n) = \ou+ p^t\ow$. 
Furthermore, since $\ov \in A$ and
$ \|\ou- \ov \|= \|   \frac{1}{2} \q(\ov) k\ve_n  \| \le p^{-t} < p^{-k}$,
we have $\ou \in A$. Clearly, we have
$\q(\ov)=\q(k\ve_1)=0$ and
\[ 2 \beta(\ou,\ow)= 2 \beta( k\ve_1, \frac{1}{2} p^{-t} \q(\ov) k\ve_n)=
p^{-t} \q(\ov) \beta(k\ve_1, k\ve_n) =p^{-t} \q(\ov) \in B. \]
This equality implies that \[ \mu_n(C_t(A,B))= \mu_n( C'_t(A,B) )= \lim_{\ell \to \infty} \frac{| \red_{\ell}(C'_t(A,B) )|}{p^{n\ell} }. \]
Set $ A_{\ell}= \red_{\ell}(A) \subseteq \ZZ/p^{\ell}\ZZ$ and $ B_{\ell}=\red_{\ell}(B) \subseteq \ZZ/p^{\ell}\ZZ$. We will denote the induced quadratic form over $\ZZ_p/p^{\ell}\ZZ_p$ by $\q$ as well. First observe that for $\ell/2>t>k$, we have
\[ \red_\ell(C'_t(A,B))=  \{ \overline{u}+ p^t \overline{w} :  \bar{u} \in A_{\ell},  \q(\bar{u})=0,  2\beta(\bar{u},\bar{w}) \in B_{\ell}  \}. \]

For any choice of $ \overline{u} \in A_{\ell} \cap \{ \overline{u}: \q( \overline{u})=0 \}$ there exist exactly
$| B_{\ell}|$ choices for $ \overline{w}$ such that $ 2\beta(\bar{u},\bar{w}) \in B_{\ell} $.
 Hence, in order to find
$ \mu_n(C_t(A,B))$, it suffices to find the cardinality of the fibers of the map $\Phi: (\ZZ/p^\ell\ZZ)^n \times (\ZZ/p^\ell\ZZ)^n \to (\ZZ/p^\ell\ZZ)^n$ defined by
\[ \Phi( \overline{u}, \overline{w})=  \overline{u}+ p^t \overline{w}. \]

We will show that all the fibers have the same size. If $\Phi( \overline{u}_1, \overline{w}_1)=  \Phi( \overline{u}_2, \overline{w}_2)$, then $\overline{u_1} \equiv \overline{u_2} \pmod{p^t}$. Conversely, suppose that $\overline{u_1} \in A_{\ell}$, $\q(\overline{u_1})$ and $\overline{v}=\overline{u_1}+ p^{t}\overline{w_1}$ for some $\overline{w_1}$ which satisfies $2\beta(\bar{u_1},\bar{w_1}) \in B_{\ell}$. Let $ \overline{u_2} \in A_{\ell}$ be such that $ \q( \overline{u_2})=0$, $ u_1 \equiv u_2 \pmod{p^t}$. Write $ \overline{u_2}- \overline{u_1}=p^t z$. Set $ \overline{w_2}= \overline{w_1}-z$, hence $\overline{v}=\overline{u_2}+p^t\overline{w_2}$. 
First note that $\q( \overline{ u_1})= \q( \overline{u_2})=0$ combined with 
\[ \q( \overline{u_2})= \q(u_1+ p^t z)= q(u_1)+2 p^t \beta( \overline{u_1}, z)+ p^{2t} \q(z) \] implies that $ \beta( \overline{u_1}, z)$ is divisible by $p^t$. 
This implies that 
 \[ 2\beta( \overline{u_2}, \overline{w_2})= 2 \beta( \overline{u_1}- p^tz , \overline{w_1}+ z) =
 2\beta(\overline{u_1},\overline{w_1})-2p^t \beta(z, \overline{w_1}+z)+ 2 \beta ( \overline{u_1},z)   \in B_{\ell} \]
  This shows that the fibers of $\Phi$ are in one-to-one correspondence with the equivalence classes of $ \overline{u} \in A_{\ell}$ module $p^t$. This implies that for all sets $A,B$ satisfying \eqref{level}, for $\ell \ge t>k$, we have
\[ |\red_\ell(C_t(A,B))|= N_t | A_{\ell} \cap \{ \overline{u} : \q(\overline{u})=0 \}| \times |B_{\ell}|. \]
Note that $N_t$ does not depend on $\ell$.
In the special case $A=\U_p^n$ and $B=\ZZ_p$, $C_t(A,B)$ is precisely the $p^{-t}$-neighborhood of the set $\{ \ou \in \U_p^n:  \q(\ou)=0 \}$. It follows that

\[
 \lim_{t \to \infty} p^t \mu_n(C_t(\U_p^n, \ZZ_p))=  \mu_{n-1}(\{ \ou \in \mathcal \U_p^n:  \q(\ou)=0 \})= \vol( K.\ve_1). \]

%
%From here, we have
%\begin{equation}
%\begin{split}
%\lim_{T \to \infty} p^T \mu_n(C_T(A,B)) & = \lim_{T \to \infty} p^T  \frac{| \red_{T}(C'_T(A,B) )|}{p^{nT} \cdot p^{T}} \\
%& = \lim_{T \to \infty} \frac{R_{T} \cdot | A_{\ell} \cap \{ u : Q_0(u)=0 \}| \times |B_{\ell}|}{p^{nT}}.
%\end{split}
%\end{equation}

From here, for general $A$ and $B$ we obtain
\begin{equation*}
\begin{split}
\lim_{t \to \infty} p^t \mu_n( C_t(A,B) ) &=
 \lim_{t \to \infty}  p^t \mu_n(C_t(\mathcal \U_p^n, \ZZ_p))
 \frac{\mu_n(C_t(A,B))}{\mu_n(C_t(\mathcal \U_p^n, \ZZ_p))}  \\
&= \vol( K.\ve_1) \lim_{ \ell \to \infty}  \frac{|A_{\ell} \cap \{\ov\in (\ZZ/p^{\ell} \ZZ)^n: \q(\ov)=0 \} \cdot |B_{\ell} |  }{    |\{ \ov \in  (\ZZ/p^{\ell} \ZZ)^n : \q(\ov)=0 \} \cdot | \ZZ/p^{\ell} \ZZ|      } \\
&= \vol( K.\ve_1)m_K\{k \in  K: k\ve_1 \in A \} \mu_1(B).
\end{split}
\end{equation*}
\end{proof}

\section{Counting results}\label{sec:counting}
In this section, we prove Theorems \ref{uniform-upper-bound} and 
\ref{main:asymptotics}. The proofs depend on an ergodic result which will be proven in the next section. 
%In the course of these proofs we will deal with compactly supported continuous function $h: (\QQ_S^n-\{ 0 \})\times \QQ_S \to \RR$ of the form
%$h(\ov, \zeta)= \prod_{p \in S} h_p( \ov_p, \zeta_p)$. 
%For $h$ in this space, define
%\[
%L(h)=\lim_{\t\rightarrow \infty} \prT^{-(n-2)} \int_{\QQ_S^n} h\left( \frac{\ov}{\T^\c}, \q(\ov) \right) \ d\ov.\]

\begin{proof}[Proof of Theorem \ref{uniform-upper-bound}]
Since any bounded set $\Omega$ is contained in a ball of radius $r$ for some $r>0$, for an upper bound, it suffices to prove the theorem for the case when $\Omega$ is a unit ball.

Without loss of generality, we may assume that there exists a compact set $C\subset  \SG$ so that every quadratic form in $\mathcal D$ is of the form $\q^g$ for some $g\in C$, where $\q$ is a fixed standard form, since there are finitely many equivalence classes of quadratic forms over $\QQ_S$.
Since $C$ is compact, there exists $\beta$ such that for every
$g=(g_p)_{p \in S} \in C$ and every $\ov\in \QQ_S^n$, $\beta^{-1}\|\ov_p\|_p\leq
\|g_p \ov_p\|_p \leq \beta\|\ov_p\|_p$ for every $p\in S$. Let
$\epsilon>0$ be given and $g\in C$ be arbitrary. Let $f=\prod_{p \in  S} f_p$ be a continuous
non-negative function with compact
support on $\QQ_S^n$  such that $J_{f_\infty}\geq (1+\epsilon)^{1/|S|}$ on $[\beta^{-1},
2\beta]\times I_{ \infty}$ and $J_{f_p}\geq (1+\epsilon)^{1/|S|}$
on $p^{\ZZ} \cap \{ x \in \QQ_p:  \beta^{-1} \le |x|_p \le 2\beta \} \times I_p$ for $p\in
S_f$. It is clear that if $\ov\in \QQ_S^n$ satisfies the conditions
\begin{equation}\label{eq2.3}
T_{ \infty} \leq \|\ov_\infty\|\leq 2 T_{ \infty} ,\|\ov_p\|_p = T_p,  \quad \q(g\ov) \in I,
\end{equation}
then $J_f(\|g\ov \|^\c /|\T|^\c , \q(g \ov)) \geq
1+\epsilon$. 
By Lemma \ref{L:JF} and Lemma 3.6 in
\cite{EMM}, for sufficiently large $\t$,
\[c(\SK)\prT^{n-2}\int_\SK f(a_\t \sk \sg\ov)dm(\sk)\geq 1
\]
if $\ov$ satisfies \eqref{eq2.3}. Summing over $\ov\in \ZZ_S^n$,
it follows that the number of vectors $\ov \in \ZZ_S^n $ satisfying  \eqref{eq2.3} is bounded from above by 
\begin{equation}\label{eq2.3bel}
 \sum_{\ov\in\ZZ_S} c(\SK)\prT^{n-2}\int_{\SK}
f(a_t\sk\sg\ov)dm(\sk)
=c(\SK)\prT^{n-2}\int_{\SK} \tilde{f}(a_{\t}\sk\sg)dm(\sk).
\end{equation}

Since \eqref{eq2.3bel} holds for all sufficiently large $\T$, by summing over all $(T_\infty/2^{k_\infty}, T_p/p^{k_1}, \ldots, T_s/p^{k_s})$, $k_\infty, k_1, \ldots, k_s \in \NN$, the statement follows.
\end{proof}
%%%%%%%%%%%%%%%%%%%%%%%%%%%%%%%%%%%%%%%%%%%%%%%%%%%%%%%%%%%%%%%%%%%%%%%%%%%%

%%%%%%%%%%%%%%%%%%%%%%%%%%%%%%%%%%%%%%%%%%%%%%%%%%%%%%%%%%%%%%%%%%%%%%%%%%%%
\begin{proposition}\label{prop 3:7}
 Let $f_p: \QQ_p^n \to \RR$ be a positive continuous function with compact support in $\QQ_{p,+}^n$, and 
$f=\prod_{p\in S}f_p$. Assume further that $f_p$ satisfies the property \eqref{IF} when $p\in S_f$. Let
$\nu_p: \U^n_p \to \RR$ be positive continuous functions satisfying the condition  \eqref{IV} and
$\nu=\prod_{p\in S}\nu_p$. For a non-exceptional quadratic form $\q$, $\epsilon>0$ and every $\sg\in  \SG$, there exists a
positive $S$-time $\t_0$ so that for
$\t \succ \t_0$,
\begin{equation}
\left| \prT^{-(n-2)}\sum_{\ov\in\ZZ_S^n} 
J_f \left( \frac{\| \sg \ov \|^\c}{\T^\c}, \q( \sg \ov) \right) \nu \left( \frac{\sg\ov}{\| \sg \ov \|^\c} \right)
-c(\SK)\int_\SK
\tilde{f}(a_\t \sk \sg)\nu(\sk^{-1}\ve_1)dm(\sk)\right|\leq \epsilon, 
\end{equation}
where $c(\SK)=\prod_{p\in S}c(K_p)$. 

\end{proposition}
\begin{proof} Since $J_f=\prod_{p \in S}J_{f_p}$ has compact support,
by Theorem \ref{uniform-upper-bound} for given $\sg\in \SG$, the number of
$\ov\in \ZZ_S^n$ such that $J_f(\sg, \ov, \t)\nu(\sg, \ov)\neq 0$ is
bounded by $c\prT^{n-2}$ for some $c>0$. Take $\t_0$ such that Lemma
\ref{L:JF} and Lemma 3.6 in \cite{EMM} hold for $\sg\ov$
and $\epsilon/c$. Then for $\t\succ\t_0$,
\begin{align*}
&\left|\prT^{-(n-2)}\sum_{\ov\in\ZZ_S^n} 
J_f \left( \frac{\| \sg \ov\|^\c}{\T^\c}, \q( \sg \ov) \right) \nu \left( \frac{\sg\ov}{\| \sg \ov \|^\c} \right)-c(\SK)\int_\SK
\tilde{f}(a_{\t}\sk \sg)\nu(\sk^{-1}\ve_1)dm(\sk)\right|\\
&\hspace{0.5cm}\leq \prT^{-(n-2)}\sum_{\ov\in\ZZ_S^n} \left| 
J_f \left( \frac{\| \sg \ov\|^\c}{\T^\c}, \q( \sg \ov) \right) \nu \left( \frac{\sg\ov}{\| \sg \ov \|^\c} \right)
-c(\SK)\prT^{n-2} \int_\SK f(a_{\t}\sk \sg \ov)\nu(\sk^{-1}\ve_1)
dm(\sk)\right|\\
&\hspace{0.5cm}\leq C(f) e^{-(n-2)t_0}\Sf p^{-(n-2)t_p}\cdot \epsilon/c
\cdot c\prT^{n-2}= C(f) \epsilon.
\end{align*}
Here $C(f)$ is a constant depending on $f$ which comes from the fact that if $|a_i - b_i|<\epsilon$ and $|a_i| < A, |b_i|<B$, then $| \prod a_i - \prod b_i | < C(A, B) \epsilon$. Note that both terms in the second line of the equations are bounded by a constant depending on $f$. 
\end{proof}
%%%%%%%%%%%%%%%%%%%%%%%%%%%%%%%%%%%%%%%%%%%%%%%%%%%%%%%%%%%%%%%%%%%%%%%%%%%%
In the rest of this section, we will deal with compactly supported continuous function $h: (\QQ_S^n-\{ 0 \})\times \QQ_S \to \RR$ of the form
$h(\ov, \zeta)= \prod_{p \in S} h_p( \ov_p, \zeta_p)$.

For $h$ as above, define
\[
L(h)=\lim_{\t\rightarrow \infty} \prT^{-(n-2)} \int_{\QQ_S^n} h\left( \frac{\ov}{\T^\c}, \q(\ov) \right) \ d\ov.\]

%%%%%%%%%%%%%%%%%%%%%%%%%%%%%%%%%%%%%%%%%%%%%%%%%%%%%%%%%%%%%%%%%%%%%%%%%%%%
\begin{lemma}\label{lem 3.9}
%Suppose $f=\prod_{p \in S} f_p$ be a continuous function of
%compact support on $\QQ^n_{S,+}$, $\nu=\Sf \nu_p$ a
%non-negative continuous function on $\SS^{n-1}\times \Sf
%\U^n_p$. 
Let $f$ and $\nu$ be as in Theorem \ref{prop 3:7}.
Let $h_\infty ( \ov_{ \infty}, \zeta_{ \infty})=J_{f_\infty}(\|\ov_\infty\|_{\infty},
\zeta_\infty)\nu_\infty(\ov_\infty/\|\ov_\infty\|_\infty)$, 
$h_p( \ov_p, \zeta_p)  =J_{f_p}(\|\ov_p\|_p^{\c}, \zeta_p)\nu_p(\ov_p/\|\ov_p\|_p^{\c})$ for 
$p \in S_f$. Set $h(\ov, \zeta)= \prod_{p \in S}  h_p( \ov_p, \zeta_p)$. 
Then we have
\begin{align}
%\lim_{T\rightarrow\infty}\Sf \lim_{p^t\rightarrow\infty}
L(h)= c(\SK)
 \int_{\SG/\Gamma}\tilde{f}(g)\, dg  \prod_{p \in S} \int_{K_p}
 \nu_p(k_p^{-1}\ve_1)dm(k_p)\nonumber
\end{align}
\end{lemma}
\begin{proof} Since
\begin{align*}
\int_{\QQ_p^{n-1}} f(p^{-r}, x_2, \ldots, x_{n-1},
x_n)dx_2\ldots dx_n
=\int_{\QQ_p^{n-1}} f(p^{-r}u, x_2, \ldots,
x_{n-1}, x_n)dx_2\ldots dx_n
\end{align*}
for any unit $u$, we have that
\begin{align*}
\int_{\QQ_p^n} f dx_1\ldots dx_n
&=\int_{\QQ_p}\int_{\QQ_p^{n-1}} f(x_1, x_2, \ldots x_{n-1},
x_n)dx_2\ldots dx_n\cdot dx_1\\
&=p^r(1-p^{-1}) \sum_{r\in\ZZ} \int_{\QQ_p^{n-1}} f(p^{-r}, x_2, \ldots
x_{n-1}, x_n)dx_2\ldots dx_n.
\end{align*}
The change of variables $p^{-r}x_n+ \q(0, x_2, \ldots, x_{n-1}, 0)=\zeta$ yields $p^r dx_n=d\zeta$, and hence
\begin{align*}\int_{\QQ_p^n} f dx_1\ldots dx_n
&=\sum_{r\in\ZZ}\int_{\QQ_p}\int_{\QQ_p^{n-2}} f(p^{-r}, x_2,
\cdots, x_{n-1}, x_n)dx_2\ldots dx_{n-1} (1-p^{-1})d\zeta\\
&=\sum_{r\in\ZZ} \int_{\QQ_p} J_f(p^{-r},
\zeta)p^{r(n-2)}(1-1/p)d\zeta.
 \end{align*}

The result will now follow from Proposition \ref{Siegel integral formula},  Proposition \ref{L:h}, and Lemma 3.9 in \cite{EMM}. 
Computations proceed by a verbatim repetition of the proof of Lemma 3.9 in \cite{EMM}.
%Computations proceed verbatim those in Lemma 3.9 in \cite{EMM}.  

\end{proof}
%%%%%%%%%%%%%%%%%%%%%%%%%%%%%%%%%%%%%%%%%%%%%%%%%%%%%%%%%%%%%%%%%%%%%%%%%%%%

\begin{proof}[Proof of Proposition \ref{volume-asym}] Since $\Omega=\prod_{p \in S} \Omega_p \subset \QQ_S^n$, it suffices to show that for any choice of intervals $I_p \subseteq \QQ_p$, $p \in S$,
there exists a constant $ \lambda(\q_p, \Omega_p)$ such that $\vol \{ \ov \in  T_p\Omega_p : \q_p(\ov) \in I \}$ is asymptotically
$ \lambda(\q_p, \Omega_p)T_p^{n-2}$. The case $p= \infty$ is in Lemma 3.8 (ii) of \cite{EMM}. Hence, we may assume that $p\in S_f$, and write
$I_p= a+p^b\ZZ_p$, for $a \in \QQ_p$ and $b \in \ZZ$. 
  
%\begin{equation} 
%\vol(\{\ov\in \QQ_p^n : \q_p(\ov)\in a+p^b\ZZ_p \}\cap p^{-t}\Omega_p)\sim \lambda_{\q_p, \Omega_p}p^{-b}p^{t(n-2)}
%\end{equation}
Recall that $\Omega_p=\{\ov\in \QQ_p^n : \|\ov\|_p \leq \rho_p(\ov/\|\ov\|_p^{\c})\}$, where $\rho_p(u\ov)=\rho_p(\ov)$ for any $\ov\in \U_p^n$ and $u\in \U_p$. Substituting $g^{-1}\Omega_p$ by $\Omega_p$, we may assume that $\q_p$ is standard. Moreover, the function $\rho_p'$ determining $g^{-1}\Omega_p$ satisfies $\rho_p'(\ov)=\rho_p'(u\ov)$ for $u\in \U_p$. It would be slightly more convenient to work with the set 
$\hat{\Omega}_p
=\{\ov\in\QQ_p^n : \|\ov\|_p=\rho_p(\ov/\|\ov\|_p^{\c})\}$. 

We now apply Proposition \ref{L:h} to find sequences $(h_m)$, $(h'_m)$ of compactly supported continuous positive functions such that $h_m\leq \1_{\hat{\Omega}_p\times I_p} \leq h'_m$ and $|h_m-h'_m|$ converges uniformly to zero as $m\rightarrow \infty$. Since $h_m, h'_m$ are compactly supported, by definition, $\lim_{m\rightarrow\infty}L_1(h_m)= \lim_{m\rightarrow \infty}L_1(h'_m)$. Hence
\begin{align*}
&\lim_{t\rightarrow\infty}p^{-(n-2)t}\vol(\{\ov\in p^{-t}\hat\Omega_p : \q_p(\ov)\in  I_p \}  ) 
=\lim_{t\rightarrow\infty}p^{-(n-2)t}\int_{\QQ_p^n}\1_{\hat{\Omega}_p\times I_p}(p^t\ov, \q_p(v))d\ov\\
&\hspace{2cm}=\vol(K.\ve_1)p^{-b}\sum_{z \in \ZZ}p^{(n-2)z}\int_{K}\1_{\hat{\Omega}_p}(p^{-z}k\ve_1)dm(k)=:\lambda_{\q_p, \hat\Omega_p}p^{-b}.
\end{align*}
Note that $\sum_{z \in \ZZ}p^{(n-2)z}\int_{K}\1_{\hat{\Omega}_p}(p^{-z}k\ve_1)dm(k)$ is a summation over a finite set of $z$'s.
The result will now follow from the fact that $ \Omega_p= \bigcup_{i \ge 0} p^i\hat{\Omega_p}$ is a disjoint union.  
\end{proof}
 
\begin{proof}[Proof of Theorem~\ref{main:asymptotics}.]
We will start by introducing two function spaces: let $\L$ denote the space of real-valued functions $F(\ov, \zeta)$ on  $\QQ_S^n- \{ 0 \} \times \QQ_S$ with support in an implicitly fixed compact set $C$, satisfying $F( u \ov, \zeta)= F( \ov, \zeta)$ for all $u \in \U_p, p \in S_f$. We equip $\L$ with the topology of uniform convergence on $C$. By $\L_0$, we denote the subspace of $\L$ consisting of the functions of the form
$F( \ov , \zeta)= J_f(\| \ov \|^\c, \zeta) \nu( \ov/\| \ov\|^\c)$, where $f= \prod_{p \in S }  f_p$, $\nu= \prod_{p \in S} \nu_p$, $f_p$ is continuous with compact support on $\QQ_{p,+}^n$ satisfying \eqref{IF} and $\nu_p$ is a non-negative continuous function on $\U_p^n$ satisfying \eqref{IV}.
We claim that $\L_0$ is dense in $\L$. In fact, it is easy to see from a $p$-adic Stone-Weierstrass Theorem (see \cite{Dieudonne}) that the approximation holds if 
the invariance properties are dropped from $\L$ and $\L_0$. With the invariance properties, one only needs to integrate the approximating function $\nu$ over
$\prod_{p \in S} \U_p$ with respect to the Haar measure to obtain the desired approximating function. 
This implies that for each $x\in \mathcal D$, there are $h_x$, $ h'_x$ in $\mathcal L_0$ such that for all $\ov\in \QQ_S^n$ and $\zeta\in \QQ_S$,
\begin{equation*}
h_x(x\ov, \zeta) \leq \1_{\hat{\Omega}} \leq h'_x(x\ov, \zeta)
\end{equation*}
and
\begin{equation*}
|L(h_x)-L(h_x')|<\epsilon,
\end{equation*}
where $\1_{\hat{\Omega}}$ is the characteristic function of
\begin{equation*}
\hat{\Omega}=\left\{\ov\in\QQ_S^n :\begin{array}{c}
\rho_\infty(\ov_\infty/\|\ov_\infty\|_\infty)/2 < \|\ov_\infty\|_\infty<\rho_\infty(\ov_\infty/\|\ov_\infty\|_\infty)\  \text{and} \\
 \|\ov_p\|= \rho_p(\ov_p/\|\ov_p\|_p^{\c}), p\in S_f \end{array} \right\}.
\end{equation*}

By Proposition \ref{prop 3:7}, Theorem \ref{main-DM}, and Proposition \ref{lem 3.9}, there exist points $x_1, \ldots, x_{\ell}\in  \SG/\Gamma$ so that $ \SH x_i$ is closed, $1\leq i\leq \ell$, and for each compact subset $\mathcal F$ of $\mathcal D - \cup_{i=1}^{\ell}  \SH x_i$, there exists $\T_0\succ 0 $ so that for all $\T\succ\T_0$ and $x=\sg\Gamma \in \mathcal F$,
\begin{equation*}
\begin{split}
\left|\prT^{-(n-2)}\sum_{\ov\in\ZZ_S^n}  h \left(  \frac{\sg \ov}{\T^\c}, \q(\sg \ov ) \right)
- L(h) \right| 
&\leq\left|\prT^{-(n-2)}\sum_{\ov\in\ZZ_S^n}  h \left(  \frac{\sg \ov}{\T^\c}, \q(\sg \ov ) \right)
-c(\SK)\int_\SK \tilde{f}(a_{\t} \sk x)\nu(k^{-1}\ve_1) dm(\sk)\right| \\
&+\left|c(\SK)\int_\SK \tilde{f}(a_{\t} \sk x)\nu(k^{-1}\ve_1) dm(\sk)-c(\SK)\int_{ \SG/\Gamma} \tilde{f} dg \int_\SK \nu dm(\sk) \right|\\
&+\left|c(\SK)\int_{ \SG/\Gamma} \tilde{f} dg \int_\SK \nu dm(\sk)-\prT^{n-2}\int_{\QQ_S^n} 
h \left(  \frac{\sg \ov}{\T^\c}, \q(\sg \ov ) \right)d\ov\right|<\epsilon.\\
\end{split}
\end{equation*}

From the definition of $L(h)$, it follows that for $\T$ large enough, we have
\begin{equation*}
\left|\prT^{-(n-2)}\int_{\QQ_S^n} h \left(  \frac{\sg \ov}{\T^\c}, \q(\sg \ov ) \right)  d\ov- L(h)\right|\leq \epsilon.
\end{equation*}

  By approximating argument with $h_x$, $h'_x$ with a suitable $\epsilon$, for every $\theta>0$, there exist finitely many points $x_1, \ldots, x_{\ell}$ with $ \SH x_i$ closed and for an arbitrary compact subset $\mathcal F\subseteq \mathcal D-\cup_{i=1}^{\ell}  \SH x_i$, there exists $\T_0\succ 0$ such that for every $x=\sg\Gamma \in \mathcal F$ and every $\T\succ \T_0$, we obtain
%\begin{align*}
%\sum_{\ov\in\ZZ_S^n} h_0(g\ov_0 e^{-t_0}, \q^0(g_0\ov)\Sf h_p(p^{t_p}g\ov_p, \q_p(g\ov_p))
%\end{align*}
%\begin{align*}
%\int_{\QQ_S^n} h_0(e^{-t}\ov_0, \q^0(\ov_0))\times\Sf h_p(p^t\ov_p, \q_p(\ov_p)) d\ov
%\end{align*}
\[(1-\theta)\vt{\I,\q, \hat{\Omega}}{\T} \leq \ct{\I, \q, \hat{\Omega}}{\T}\leq (1+\theta)\vt{\I,\q, \hat{\Omega}}{\T}.
\]

The claim will now follow from Proposition \ref{volume-asym} and a standard geometric series argument. 
\end{proof}

% S-aritmetic DANI-MARGULIS 

\section{$S$-arithmetic uniform equidistribution of unipotent flows}\label{s:DM}
In this section we will prove an $S$-arithmetic version of a theorem in \cite{DM}, which was mentioned in the introduction. Our proof builds upon ideas and results that were originated in \cite{DM} and were partially extended in the $S$-arithmetic setup in \cite{GT}. The general line of argument is similar to the one in \cite{DM}. In a number of places, technicalities arise that need to be handled differently. 

Let $\GG$ be a $\QQ$-algebraic group.
With $S$ and $S_f$ as above, set $\SG=\GG(\QQ_S)= \prod_{p \in S } \GG(\QQ_p)$. Let $\Gamma$ be an $S$-arithmetic subgroup of $\SG$, i.e. a group commensurable to $\GG( \ZZ_S)$. 
Note that $\GG(\QQ_p)$ is naturally embedded in $\SG$.  
For each $p \in S$, let $U_p=\{u_p(z_p): z_p \in \QQ_p \}$ be a one-parameter unipotent $\QQ_p$-subgroup, 
which means that  $u_p: \QQ_p \to \GG(\QQ_p)$ is a non-trivial $\QQ_p$-rational homomorphism.
\cite{Ra}.

 We will need a uniform version of Ratner's theorem for the action of unipotent groups on $\SG/\Gamma$, in the special
case that $\Gamma$ is an $S$-arithmetic subgroup. 
Let us mention that Ratner's main result in \cite{Ra} was independently obtained in Margulis-Tomanov paper \cite{MT}.
The following class of groups introduced by Tomanov plays
an important role in this theorem.

\begin{definition}\cite{GT}\label{classF}
A connected $\QQ$-algebraic subgroup $\GP$ of $\GG$ is a subgroup \emph{of class $\F$ relative to $S$} if for each proper normal 
$\QQ$-algebraic subgroup $\GQ$ of $\GP$, there exists $p \in S$ such that $(\GP/\GQ)(\QQ_p)$ contains a non-trivial unipotent element. 
\end{definition}

%Let $\Gamma$ be an $S$-arithmetic lattice in $\SG$. 
If $\GP$ is a subgroup of class $\F$ in $\GG$, then for any subgroup $P'$ of finite index in $\GP(\QQ_S)$,  $P' \cap \Gamma$ is an $S$-arithmetic lattice in $P'$. For a closed subgroup $U$ of $G$, and a proper subgroup $\GP$ of $\GG$ of class $\F$, we write
$X(P,U)= \{ g \in G: Ug \subseteq gP \}$. One can see that $X(P,U)$ is a $\QQ_S$-algebraic subvariety of $G$.
We will prove the following $S$-arithmetic version of \cite[Theorem 3]{DM}.

\begin{theorem}\label{main-DM}
Let $\GG$ be a $\QQ$-algebraic group, $\SG=\GG(\QQ_S)$, $\Gamma$ an $S$-arithmetic lattice in $\SG$, and let $\mu$ denote the 
$\SG$-invariant probability measure on $\SG/\Gamma$. Let $\SU=\{ (u_p(z_p))_{p \in S}| z_p \in \QQ_p \}$ be a one-parameter unipotent  
$\QQ_S$-subgroup of $\SG$, and let $\phi: \SG/ \Gamma \to \RR$ be a bounded continuous function. Let $\SKK$ be a compact subset of
$\SG/\Gamma$, and let $ \epsilon>0$. Then there exist finitely many proper subgroups $\GP_1, \dots, \GP_k$ of $\GG$ of class $\F$,
and compact subsets $C_i \subseteq X(\SP_i, \SU)$, where $\SP_i= \GP_i(\QQ_S)$, $1 \le i \le k$, such that the following holds:
given a compact subset $\SF$ of $\SKK \setminus \cup_i C_i\Gamma/\Gamma$,  for all $x \in \SF$ and all $\T$ with $m(\T) \gg 0$,  we have
\[ \left|   \frac{1}{\lambda_S(\I(\T)) } \int_{\I(\T)} \phi( u_\sz x) d\lambda_S(\sz) - \int_{\SG/\Gamma} \phi d\mu       \right| \le \epsilon. \]
\end{theorem}

Following 
\cite{DM} and \cite{GT}, we denote the set of singular points with respect to $U$ by 
$\Sing(U)= \bigcup_{\GP \in \F, \GP \neq \GG} X(P,U)\Gamma/\Gamma $. Its complement $\Gen(U)= G/\Gamma \setminus \Sing(U)$ will be called the set of generic points.

For a subset $S'$ of $S$, the direct sum $u_{S'}: \QQ_{S'} \to G$ of one-parameter unipotent subgroups defined by $u_{S'}((t_p)_{p \in S'})= ( u_p(t_p) )_{p \in S'}$ is called a one-parameter unipotent $\QQ_{S'}$-subgroup of $G$. If the subset $S'$ is known from the context, we sometimes write  $u(t)$ instead of $ ( u_p(t_p) )_{p \in S'}$.

When $r>0$, we will use the notation $I_p(r)$ for the closed ball of radius $r$ in $\QQ_p$ centered at zero. 
When $p$ is a non-archimedean place, we will implicitly assume that $r$ is a power of $p$. Any translation of $I_p(r)$, that is any set of the form $b+I_p(r)$, will be called an interval. When $p \in S_f$, the ultrametric property of the norm implies that if $J_1$ and $J_2$ are two intervals with a non-empty intersection, then $J_1 \subseteq J_2$ or $J_2 \subseteq J_1$. In the rest of this article, for a $p$-adic interval $L=I_p(r)$, we will write $ \hat{L}= I_p( p r)$.   We will start by recalling the following $S$-arithmetic version of the quantitative non-divergence theorem of Dani-Margulis which will later be needed in this paper:

% We will need the following easy lemma.
%
%\begin{lemma}\label{}
%Let $L_1, \dots, L_r \subseteq \QQ_p$ be disjoint intervals. Then $\sum_{i=1}^r \lambda_p( L_i) \le \lambda_p( \bigcup_{i=1}^{r} \hat{L_i})$. 
%\end{lemma}

\begin{theorem}[\cite{GT},Theorem 3.3]\label{recurrence}
Let $\GG$ be a $\QQ$-algebraic group, $G=\GG(\QQ_S)$, and $\Gamma$ an $S$-arithmetic lattice in $G$, and $\mu$ denote the 
$G$-invariant probability measure on $G/\Gamma$. Let $S' \subseteq S$, and $U=\{u(t)| t \in \QQ_{S'} \}$ be a one-parameter unipotent  
$\QQ_{S'}$-subgroup of $G$. Let $ \epsilon >0$ and $\SK \subseteq G/\Gamma$ be a compact set. Then there exists 
a compact subset $\SK_1$ such that for any $x \in \SK_1$ and any $S'$-interval $I(\T)$ in $\QQ_{S'}$, we have 
\[ \frac{1}{\lambda_{S'}(I(\T))} \lambda_{S'} \{ t \in I(\T)|u(t)x \not\in \SK_1 \} < \epsilon . \]
\end{theorem}

%\begin{definition}\label{cf}
%A connected $k$-algebraic subgroup $\GP$ of $\GG$ is a subgroup of class $\F$ relative to $S$ if for each proper normal 
%$k$-algebraic subgroup $\GQ$ of $\GP$ there exists $p \in S$ such that $(\GP/\GQ)(\QQ_p)$ contains a non-trivial unipotent element. 
%\end{definition}

Let $\Gamma$ be an $S$-arithmetic lattice in $G$. 
If $\GP$ is a subgroup of class $\F$ in $\GG$ then for any subgroup $P'$ of finite index in $\GP(\QQ_S)$, 
we have $P' \cap \Gamma$ is an $S$-arithmetic lattice in $P'$. 
The following proposition has been proven in \cite{GT}.

\begin{proposition}[\cite{GT}, Theorem 4.2] \label{small}
Let $M \subseteq \QQ_p^m$ be Zariski closed. Given a compact set $A \subseteq M$ and $ \epsilon>0$, there exists 
a compact set $B \subseteq M$ containing $A$ such that the following holds: for a compact neighborhood $\Phi$ of
$B$ in $\QQ_p^m$, there exists a neighborhood $\Psi$ of $A$ in $\QQ_p^m$ such that for any one-parameter unipotent subgroup $\{u(t) \}$ in $\GL_m(\QQ_p)$, any $\w \in \QQ_p^m - \Phi$, and any interval $I \subseteq \QQ_p$ containing $0$, we have 
\[ \lambda_p \{ t \in I: u(t) \w  \in \Psi \} \le \epsilon \cdot \lambda_p \{ t \in T: u(t) \w \in \Phi \}. \]
\end{proposition}

We will also need the following $S$-adic analogue of Theorem 2 in \cite{DM}. 

\begin{theorem}\label{limit}
Let $\GG$ be a $k$-algebraic group, $G=\GG(\QQ_S)$, and $\Gamma$ be an $S$-arithmetic lattice in $G$. Let $U^{(i)}=\{ u^{(i)}(t)| t \in \QQ_S\}$ be a sequence of one-parameter unipotent $\QQ_S$-subgroup of $G$, such that 
$ u^{(i)}(t) \to u(t)$ for all $t \in \QQ_S$, as $ i \to \infty$. Assume that the sequence $(x_i)_{ i \ge 1} \in G/\Gamma$
converges to the point $x \in \Gen( u(t) )$, and let $m(\T_i) \to \infty$. For any bounded continuous function 
$\phi: G/\Gamma \to \RR$, we have 
\[ \frac{1}{| \T_i |} \int_{I(\T_i)} \phi( u^{(i)}(t)x_i)  \; d\lambda_S(t) \to \int_{G/\Gamma} \phi \; d\mu. \]
\end{theorem}
 
Let us briefly sketch the proof of this theorem. The main ingredient of the proof is the 
quantitative non-divergence theorem, whose $S$-adic analogue, Theorem \ref{recurrence}, 
is proven in \cite{GT}. Arguing by contradiction, one assumes that there exists a sequence
$x_i$ of points for which the statement is not true. Using the density of the set of generic points, one may 
assume that $x_i$ are generic for $u_t$. Invoking the 
quantitative non-divergence again, one then shows that there is no escape of mass to infinity. Once we 
verify that the limiting measure is $u(t)$-invariant, the measure
classification theorem of Ratner will finish the proof. For details, we refer the reader to \cite{DM}. 

Let $\GP$ be a subgroup of class $\F$ in $\GG$. By Chevalley's theorem, there exists a $k$-rational representation $\rho_{\GP}: G \to \GL(V_\GP)$ such
that $N_{\GG}(\GP)$ equals the stabilizer of a line in $V$ spanned by a vector $\m \in V(\QQ)$. 
This representation and the vector $\m$ are fixed throughout this paper. Let $\chi$ be the 
$k$-rational character of $N_\GG(\GP)$ defined by $\chi(g)\m= g\m$, for $g \in N_\GG(\GP)$. 
We denote $\GN= \{ g \in \GG: g \m = \m \}$ and $N= \GN(\QQ_S)$. We also set 
$\Gamma_N=\Gamma \cap N$ and $\Gamma_P= \Gamma \cap N_\GG(\GP)$. The orbit map $\eta_{\GP}:
 \GG \to \GG \m \subseteq V_\GP$ is defined by $\eta_{\GP}(g)=g \m$. The orbit $\GG \m$ is isomorphic to the quasi-affine variety $\GG/\GN$ and $ \eta$ is a quotient map. Set $\GX= \{ g \in \GG: Ug \subseteq g\GP \}$ and let $A_\GP$ denote the Zariski closure of  $\eta_{\GP}( X(P,U))$. 
Clearly $\GX$ is an algebraic variety of $\GG$ defined over $k_S$ and $\GX(\QQ_S)=X(P,U)$. 
It is not hard to show (see \cite{GT}) that $\eta^{-1}( A_\GP )= X(P,U)$.
It will be useful to consider the map $\Rep: G/\Gamma \to V_{\GP}$ defined as follows. For each $x \in G/\Gamma$, we define
$\Rep(x)= \{ \eta_{\GP}(g): g\in G, x=g\Gamma \}$. For $D \subseteq A_{\GP}$ and for $ \gamma \in \Gamma$, we define the $\gamma$-overlaps of $D$ by 
$\ol^{\gamma}(D)=\{ g\Gamma: \eta_{\GP}(g) \in D, \eta_{\GP}(g \gamma) \in D \} \subseteq G/\Gamma.$
Finally, we set 
\[ \ol(D)= \bigcup_{\gamma \in \Gamma - \Gamma_{\GP}} \ol^{\gamma}(D) \subseteq G/\Gamma. \]

The proof of the following lemma is straightforward:

\begin{lemma}\label{}
For $\gamma \in \Gamma$ and $ \gamma_1 \in \Gamma_{\GP}$, and $D \subseteq A_{\GP}$ we have
\begin{enumerate}
\item  $\ol^{e}(D)= \{ x \in G/\Gamma: \Rep(x) \cap D \neq \emptyset \}$.
\item $\ol^{\gamma}(D)=\ol^{\gamma \gamma_1}(D)$. 
\end{enumerate}
\end{lemma}

For each subgroup  $\GP$ of class $\F$ relative to $S$, we will
denote $I_{\GP}= \{ g \in \GG: \rho_{\GP}(g)m_{\GP}=m_{\GP} \}$. 
A proof for the following proposition can be given along the same lines as the proof of Proposition 7.1 in \cite{DM}. 
\begin{proposition}\label{cover}
Suppose $\GP$ is a subgroup of class $\F$ relative to $S$, and $C \subseteq V_{\GP}$ is compact. Assume also that 
$ \SK \subseteq G/\Gamma$ is compact. Then there exists a compact set $ \widetilde{ C} \subseteq G$ such that 
\[ \pi( \widetilde{ C})= \{ x \in \SK: \Rep(x) \cap C \neq \emptyset \}. \]
\end{proposition}

\begin{proposition}\label{overlap}
Let  $\GP$ be a subgroup of class $\F$ relative to $S$ and $D \subseteq A_{\GP}$ be compact. Let $\SK \subseteq G/\Gamma$ be 
compact. Then the family $\{ \SK \cap \ol^{\gamma}(D) \}_{\gamma \in \Gamma}$ contains only finitely many distinct elements. 
Moreover, for each $\gamma \in \Gamma$, there exists a compact set $ \widetilde{ C}_{\gamma} \subseteq \eta_{\GP}^{-1}(D) 
\cap \eta_{\GP}^{-1}(D)\gamma$ such that $\SK \cap \ol^{\gamma}(D)= \pi( \widetilde{ C}_k)$.
\end{proposition}

\begin{proof}
The argument for finiteness from Proposition 7.2. in \cite{DM} can be carried over verbatim to this case. 
\end{proof}

Let us denote by $\evit$ the class of subsets of $G$ of the form $E= \bigcap_{i=1}^r \eta_{ \GP_i}^{-1}( D_i)$,
where $\GP_i$ are subgroups of class $\F$ and $D_i \subseteq A_{\GP_i}$ are compact. For such a set $E$ (together with the 
given decomposition), we denote $\N(E)$ to be the family of all neighborhoods of the form $\Phi=  \bigcap_{i=1}^r \eta_{ \GP_i}^{-1}( \Theta_i)$, where $ \Theta_i \supset D_i$ are neighborhoods in $V_{\GP_i}$. We will refer to these neighborhoods as components
of $\Phi$.  The following definition from \cite{GT} will also be used in this paper. Let $ \delta >0$ be such that 
for any $z \in \ZZ_S^{\ast}$, if for all $p' \in S \setminus \{ p \}$, we have $|z-1 |_{p'}< \delta$, then $z=1$. The existence of $\delta$ follows from the fact that the map from $\ZZ_S^{\ast}$ to $\RR^{S}$ defined by 
$ \lambda(x)=(\log |x|_p)_{p \in S}$ has a discrete image in a codimension-one subspace of $\RR^S$. A subset $ A = \prod_{p \in S} A_p \subseteq G$ is said to be
{\it $S(p)$-small} if for all $ p' \in S \setminus \{ p \}$ the following holds: if $c \in \QQ_{p'}^{\ast}$ is such that 
$c(A_{p'} \m) \cap A_{p'} \m \neq \emptyset$, then $|c-1|_{p'} < \delta$. We will now prove a stronger version of a theorem in \cite{GT}. 
\begin{theorem}
Let $p \in S$, $ \epsilon >0$, and let $\SK \subseteq G/\Gamma$ be compact. For $E \in \evit$, there exists $E' \in \evit$ such that 
the following holds: given $\Phi \in \N(E')$, there exists a neighborhood $ \Omega \supseteq \pi(E)$ such that for any 
one-parameter unipotent subgroup $\{ u_p(t) \}$ of $G$, and any $g \in G$, and any $r_0>0$, one of the following holds:
\begin{enumerate}
\item A component of $\Phi$ contains $\{ u_p(t)g\gamma: t \in I_p(r) \}$ for some $\gamma \in \Gamma$. 
\item For all $r>r_0$, we have
$\lambda_p \{ t \in I_p(r) \setminus I_p(r_0): u_p(t)g\Gamma \in \Omega \cap \SK \} \le \epsilon \lambda_p( I_p(r) \setminus I_p(r_0))$. 
\end{enumerate}
\end{theorem}

\begin{proof} It is clear that we can assume that $E= \eta_{\GP}^{-1}(C)$ and that $E$ is $S(p)$-small. We will now proceed by the induction on $\dim \GP$. The result is clearly valid for $\dim \GP=0$. 
Let us assume that it is known for all $\GP$ with dimension at most $n-1$ and that $C \subseteq A_{\GP}$, with 
$\dim \GP=n$. Applying Proposition \ref{small} to $C$ (viewed as a compact
subset of the Zariski closed set $A_{\GP}$), we obtain a compact subset 
$D$ of $A_{\GP}$ such that for any compact neighborhood $\Phi$ of
$D$ in $A_{\GP}$, there exists a neighborhood $\Psi$ of $C$ in $A_{\GP}$ such that for any one-parameter subgroup $\{u_p(t) \}$ of $\GL(V_{\GP})$, any point $\w \in V_{\GP} \setminus \Phi$, and any interval $I \subseteq \QQ_p$ containing $0$, we have 
\[ \lambda_p \{ t \in I: u_p(t) \w  \in \Psi \} \le \epsilon \cdot \lambda_p \{ t \in T: u_p(t) \w \in \Phi \}. \]
Note that since the set of the roots of unity in $\QQ_p$ is finite, we can choose 
$D$ such that $ \omega D= D$ for every root of unity $ \omega \in \QQ_p$, and thus $D$ can be chosen to be $S(p)$-small. Now, let $B= \eta_{\GP}^{-1}(D)$. 

By Proposition \ref{overlap} the family
of sets $\{ \SK \cap \ol^{\gamma}(D) \}_{\gamma \in \Gamma}$ is finite, hence
consists of the sets $ \SK \cap \ol^{\gamma_j}(D)$, for $ 1 \le j \le k$. We 
assume that $\gamma_1=e$. Moreover, we can write $  \SK \cap \ol^{\gamma_j}(D)= \pi(C_j)$ for some
compact subset $C_j \subseteq B \cap B\gamma_j^{-1} \subseteq
X( \GP \cap \gamma_j \GP \gamma_j^{-1}, W)$. We claim that 
$\gamma_j \not\in \Gamma_{\GP}$ for $j \ge 2$. Assuming the contrary, 
we obtain $\rho(\gamma_j)\m_{\GP}=\chi(\gamma_j)\m_{\GP}$. Since
$\chi(\gamma_j) \in \O^{\ast}$, we have
$ \eta_{\GP}(b\gamma_j)= \chi(\gamma_j) \eta_{\GP}(b) \in D$. Since $D$ is $S(p)$-small, we obtain that $\chi(\gamma_j)$ is a root of unity in $\QQ_p^{\ast}$. 
This shows that $B\gamma_j \subseteq \eta^{-1}(D)=B$, which is a contradiction to the choice of $\gamma_j$. From here we see that for $j \ge 2$, $\GP 
\cap \gamma_j \GP \gamma_j^{-1}$ is a proper subgroup of $\GP$. Hence there exists a subgroup $\GP_j$ of class $\F$ which is contained in the connected component of  $\GP \cap \gamma_j \GP \gamma_j^{-1}$. Note that $\GP_j$ is of dimension less than $n$, and $C_j \subseteq X(\GP_j, W)$. We now set 
$E_j= \eta_{\GP_j}^{-1}(\eta_{\GP_j}(C_j))$ and apply the induction hypothesis to obtain $E_j' \in \evit$ such that for any choice of $\Phi_j 
\in \N(E_j')$, we can find neighborhoods $ \Omega_j$ of $E_j$ such that 
for any one-parameter subgroup $(u_p(t))_{t \in \QQ_p}$ of $G$, $g \in G$ and
$r>0$, we have 
$$\lambda_{v} \{ t \in I_p(r) \backslash I_p(r_0): u_p(t)g\Gamma \in \Omega_j 
\cap \SK \} \le (\epsilon/2k)r$$
unless there exists $\gamma \in \Gamma$ such that the set $\{ u_p(t)g\gamma: 
t \in I_p(r) \} $ is contained in a component of $\Phi_j$. Set 
$E''= \bigcup_{j=2}^n E'_j \in \evit$ and $E'=E'' \cap B$. Consider 
$\Phi \in \N(E')$. This shows that there exists a neighborhood 
$ \Omega'$ of $\pi(E'')$ such that for any one-parameter unipotent subgroup
$\{ u_p(t) \}_{t \in \QQ_p}$, every $g \in G$ and $r_0>0$, we have
\[ \lambda_p( t \in I_p(r_0): u_p(t)g\Gamma \in \Omega' \cap \SK \}
\le \frac{\epsilon r}{2}\]
unless $\{ u_p(t)g\gamma: t \in I_p(r_0) \}$ is in a component of $\Phi$ for some $\gamma \in \Gamma$. 
Set $\SK_1 = \SK \setminus \Omega'$, and choose a compact subset $K' \subseteq G$ 
such that $\pi(K')=\SK_1$.  Let $\Phi_1$ be a neighborhood of $D$ in $V$ 
such that $ \eta_{\GP}^{-1}(\Phi_1) \subseteq \Phi$ and
$O( \Phi_1) \cap \SK_1 =\emptyset$. Since $D$ is $S(p)$-small, we can clearly choose $\Phi_1$ to be $S(p)$-small. As $  \rho(u_p(t))$ is
a one-parameter unipotent subgroup of $\GL(V)$, and $ \eta_{\GP}(C)$
is of relative size less than $ \epsilon/4$ in $D$, we can find a 
neighborhood $\Psi$ of $C$ in $V$ satisfying
\[ \lambda_p( \{ t \in I_p(r): \rho( u_p(t) )v \in \Psi \}) 
\le \frac{\epsilon}{4} \lambda_p( \{ t \in I_p(r): \rho( u_p(t))\ov \in \Phi_1 \}),\]
for all $\ov \in V \setminus \Phi_1$, $r>0$ and unipotent subgroups $\{ u_p(t) \}_{t \in \QQ_p}$. Let $ \Omega= \pi( \eta_{\GP}( \Psi) ) \subseteq G/\Gamma$. Assume that (1) does not hold for $g \in G$, a one-parameter subgroup $\{ u_p(t) \}_{t \in \QQ_p }$, and $r_0>0$. This implies that for every $\gamma \in \Gamma$, there exists $t \in I_p(r_0)$ such that $u_p(t)g\gamma \in G \setminus \Phi$. 
For $q \in \mathbf{M}$, we consider the following sets:
\[ J_1(\oq)= \{ t \in I_p(r) \setminus I_p(r_0): \rho( u_p(t)g)\oq \in \Phi_1 \}, \]
\[ J_2(\oq)= \{ t \in I_p(r) \setminus I_p(r_0): \rho( u_p(t)g)\oq \in \Psi, \pi(u_p(t)g) \in
\SK_1 \}. \]
Since $J_1(\oq)$ is an open subset of $\QQ_p$, it is a disjoint union of intervals. 
We will also define $J_3(\oq) \subseteq J_1(\oq)$ as follows: if $v$ is an archimedean place, then $J_3(\oq)$ consists of those $ t \in J_1(\oq)$ such that 
for some $a \ge 0$, we have $[t,t+a] \subseteq J_1(\oq)$ and $\pi( u(t+a)g) 
\in \SK_1$. If $v$ is a non-archimedean place, then $J_3(\oq)$ consists of those $ t \in J_1(\oq)$ such that there exists an interval $J \subseteq \QQ_p$ containing $t$ and $t' \in J$ such that $\pi(u(t')g) \in \SK_1$. Clearly $J_3(\oq)$ is open
in $\QQ_p$, and is hence a disjoint union of intervals. 
We first make the following claim:

{\bf Claim}: $J_3(\oq_1) \cap J_3(\oq_2) = \emptyset$ for $\oq_1, \oq_2 \in \M$, unless
$\oq_2= \omega \oq_1$ for some root of unity $ \omega \in \QQ_p^{\ast}$. 

In the archimedean case, if $t \in J_3(\oq_1) \cap J_3(\oq_2)$, then there exists
$a \ge 0$ such that $[t, t+ a] \subseteq J_1(\oq_1) \cap J_1(\oq_2)$ and 
$\pi(u(t+a)g) \in \SK_1$. If $q_j= \eta(\gamma_j)$ for $j=1,2$, then
$ \eta( u(t+a)g\gamma_1)= \eta( u(t+a) g \gamma_2) \in \Phi_1$, we will 
have $\oq_1, \oq_2 \in \ol( \Phi) \cap \SK_1 =\emptyset$, unless $ \gamma_1^{-1}\gamma_2 \in \Gamma_{\GP}$, which implies that $\oq_2= \omega \oq_1$.

In the non-archimedean case, if $t \in J_3(\oq_1) \cap J_3(\oq_2)$, then $ t \in J_1(\oq)$ and there exist intervals 
$J(\oq_1), J(\oq_2) \subseteq \QQ_p$ containing $t$ and $t_1' \in J(\oq_1)$ and $t_2' \in J(\oq_2)$ 
such that $\pi(u(t'_1)g), \pi( u(t'_2)g) \in \SK_1$. Note that since $J(\oq_1)$ and $J(\oq_2)$ intersect one contains the other, hence,
without loss of generality, we can assume that $t'_1 \in J(\oq_1) \cap J(\oq_2)$, and $\pi(u(t'_1)g) \in \SK_1$. The rest 
of the argument is as in the archimedean case. 

Denote $I_p(a,r)=a+I_p(r)$, and let $\mathcal{L}_1$ be the family of those components $L=I_p(a,r_1)$ of $J_1(\oq)$ such that $L \cap I_p(r_0) =\emptyset$, 
and $\mathcal{L}_2$ the rest of components.  Note that $ \hat{L} \not\subseteq I_p(r_0)$. This implies 
that $\lambda_p( \hat{L} \cap J_2(\oq) ) \le \lambda_p( \hat{L} \cap J_3(\oq) )$, which, in turn, shows that
\[ \sum_{L \in \mathcal{L}_1 } \lambda_p( \hat{L} \cap J_2(\oq) ) \le 
\sum_{   L \in \mathcal{L}_1} \lambda_p( \hat{L} \cap J_3(\oq) )  \le  \lambda_p ( I_p(r) \setminus I_p(r_0) ).   \]

We now claim that $\sum_{L \in \mathcal{L}_2} \lambda_p(L) \le \lambda_p( I_p(r))$. 
In fact, if $L \in \mathcal{L}_2$, then either $L \subseteq I_p(r_0)$ or $I_p(r_0) \subseteq L$. 
If $I_p(r_0) \subseteq L$ for some $L \in \mathcal{L}_2$, then since components are disjoint, $\mathcal{L}_2$ has precisely one element and the result follows. So, assume that for each $L \in  \mathcal{L}_2$, we have
$L \subseteq I_p(r_0)$. Then the disjointness of components implies that $\sum_{L \in \mathcal{L}_2} \lambda_p(L) \le \lambda_p( I_p(r_0))$. 
\end{proof}

%%%%%%% TWO BASIC PROPOSITIONS %%%%%%%%%%

\begin{proposition}\label{}
Let $\GG$ be a $\QQ$-algebraic group, $G=\GG(\QQ_S)$, and $\Gamma$ be an $S$-arithmetic lattice in $G$.
Let $U$ be a one-parameter unipotent subgroup of $G$. Assume that 
$\GP_1, \dots, \GP_k$ are subgroups of class $\F$ for $ 1 \le i \le k$, and let 
$D_i$ be a compact subset of $A_{\GP_i}$, and $\Theta_i$ be a compact neighborhood of
$D_i$ in $V_{\GP_i}$. For a given compact set $\SK \subseteq G/\Gamma$, there exists 
$\GP'_1, \dots, \GP'_k$ of class $\F$ and compact subsets 
$D'_i$ of $A_{\GP'_i}$, $ 1 \le  i \le k$, such that for any 
compact set $ \SF \subseteq \SK \setminus \bigcup_{i=1}^{k} ( \eta_{\GP_i}^{-1}( D_i) \cup
\eta_{\GP'_i}^{-1}(D'_i) )\Gamma/\Gamma$, there exists $T_0$ such that 
for any $g \in G$ with $g\Gamma \in \SF$, and $ 1 \le i \le k$, there exists $t\in I_p(T_0)$ such that 
$u_p(t)g \not\in \eta_{ \SP_i}^{-1}( \Theta_i)$. 
\end{proposition}

\begin{proof}
The proof of this proposition is very similar to the proof of Proposition 8.1 in \cite{DM}. 
Let us denote by $I_\GP$ the set of $g \in G$ with $g.\m_{\GP}=\m_{\GP}$. We first claim that there exists a subgroup $\GP'$ of class $\F$ such that $X(I_{\GP}, U) \subseteq X(\GP',U)$. In fact, let 
$\GP'$ be the smallest connected algebraic subgroup of $\GG$ which contains all the unipotent elements of $I_{\GP}$. Note that since $\GP'$ is generated by unipotent subgroups, 
we have $X_k(\GG)=\{ 1 \}$, where $X_k(\GG)$ denotes the group of characters of $G$ defined over $k$. It follows from Theorem 12.3 of \cite{BH} that $\GP' \cap \Gamma$ is a lattice in $\GP'$ and $\GP'$ is of class $\F$. 
Let $\GP'_1, \dots, \GP'_k$ be chosen as above such that $X(I_{\GP_i}, U) \subseteq X(\GP', U)$
for all $1 \le i \le k$. We will also define $Q_i= \{ w \in \Theta_i: \rho_{\GP_i}(u_p(t)) w=w, \ \forall t \in \QQ_p \}$. 
Using Proposition \ref{cover}, we can find compact subsets $C_i \subseteq \eta_{\GP_i}^{-1}(Q_i)$, $1 \le i \le k$ such that 
\[ \pi( \widetilde{ C_i})= \{ x \in Q_i: \Rep(x) \cap C_i \neq \emptyset \}. \]
This implies that $C_i \subseteq X( \GP_i, U)$. Consider the compact sets $D_i'= \eta_{\GP'_i}(C_i) \subseteq A_{\GP'_i}$ and assume that $\SF$ is a compact subset of 
$\SK \setminus \pi( \bigcup_{i=1}^{k} ( \eta_{\GP_i}^{-1}(D_i) \cup \eta_{\GP_i'}^{-1}(D'_i))$. Fix a compact subset $F' \subseteq G$ such that $\SF= \pi(F')$. From the fact that 
$ \rho_{\GP_i}(\Gamma) \m_{\GP_i}$ is a discrete subset of $V_{\GP_i}$, it follows that there
are only finitely many $ \gamma \in \Gamma$ such that $\eta_{\GP_i}( \gamma) \in \rho_{\GP_i}(F')^{-1}\Theta_i$ for some $1 \le i \le k$. It thus suffices to show that for each 
$1 \le i \le k$ and $\gamma \in \Gamma$, for all large enough $t$ we have
$ \Theta_i \cap \rho_{\GP_i}(u_p(t))( \Theta_i \cap \rho_{\GP_i}(F'\gamma) \m_{\GP_i}) =\emptyset$.
Note that $ \rho_{\GP_i}(F'\gamma) \m_{\GP_i} \cap \Theta_i$ is a compact subset of $ V_{\GP_i}$
which does not contain any fixed point of the flow $ \rho_{\GP_i}(u_p(t))$.  
Since $ \rho_{\GP_i}(u_p(t))$ is a unipotent one-parameter subgroup of $\GL(V_{\GP_i})$ the result follows.

\end{proof}

\begin{proof}[Proof of Theorem \ref{main-DM}]

%
%
%\begin{theorem}\label{main-DM}
%Let $\GG$ be a $k$-algebraic group, $G=\GG(\QQ_S)$, and $\Gamma$ an $S$-arithmetic lattice in $G$, and $\mu$ denote the 
%$G$-invariant probability measure on $G/\Gamma$. Let $U=\{ (u_p(t_p))_{p \in S}| t_p \in \QQ_p \}$ be a one-parameter unipotent  
%$k_T$-subgroup of $G$, and let $\phi: G/ \Gamma \to \RR$ be a bounded continuous function. Let $\SK$ be a compact subset of
%$G/\Gamma$, and let $ \epsilon>0$. Then there exist finitely many proper subgroups $\GP_1, \dots, \GP_k$ of class $\F$,
%and compact subsets $C_i \subseteq X(P_i, U)$, where $P_i= \GP_i(\QQ_S)$, $1 \le i \le k$, such that the following holds:
%for any compact subset $\SF$ of $\SK \setminus \cup_i C_i\Gamma/\Gamma$ there exists $T_0$ such that for all $x \in \SF$ and 
%$\T $ with $m(\T )> T_0$, we have
%\[ \left|   \frac{1}{\lambda_T(I(\T) )} \int_{I(\T)} \phi( u_tx) d\lambda_{T}(t) - \int_{G/\Gamma} \phi d\mu       \right| \le \epsilon. \]
%\end{theorem}

For a bounded continuous function $\phi$ defined on $G/\Gamma$, the one-parameter unipotent group $(u(t))$ and 
$x \in G/\Gamma$, we define
\[ \Delta(\phi,u(t),x,\T )=  \left|   \frac{1}{\lambda_S(I(\T) )} \int_{I(\T)} \phi( u(t)x) d\lambda_{S}(t) - \int_{G/\Gamma} \phi d\mu       \right|. \] 
Let us consider the above statement with $S$ replaced by $S' \subseteq S$ everywhere. We will prove the statement first for the case $|S'|=1$. Then we will show that if 
the statement holds for $S_1$ and $S_2$, then it must also hold for $S_1 \cup S_2$. Suppose that $S'= \{ p \}$ for some $p \in S$. 
We argue by contradiction. Assume that the statement of the theorem is not true. This implies the existence of a
bounded continuous function $\phi: G/\Gamma \to \RR$, a compact subset $\SK_1 \subseteq G/\Gamma$, and $ \epsilon>0$ such that 
for any proper subgroups $\GP_1, \dots, \GP_k$ of class $\F$, and any compact subsets $C_i \subseteq X(P_i, U)$, where $P_i= \GP(\QQ_p)$, $1 \le i \le k$, there exists a compact set  $\SF$ of $\SK_1 \setminus \cup_i C_i\Gamma/\Gamma$  such that for all $T_0>0$, 
there exists $T>T_0$, and $x \in F$ such that $\Delta(\phi,u(t),x,T ) > \epsilon$. 
Without loss of generality, we can assume that $\phi$ has compact support and is bounded in absolute value by $1$. There exists a compact subset $\SK \subseteq G/\Gamma$
such that for all $x \in \SK_1$ and all $T $, we have 
\begin{equation}\label{time-in-K}
 \lambda_p \{ t \in I_p(T): u_p(t)x \not\in \SK \} < \frac{1}{3} \lambda_p(I_p(T)).  
\end{equation}
We can now apply this to construct an increasing sequence $E_i \subseteq E_{i+1}$ in $\evit$ such that 
\begin{enumerate}
\item The family $\{ E_i \}_{i \ge 1}$ exhausts the singular set of $U$, i.e., $ \bigcup_{i=1}^{ \infty} E_i=
\sing(U)$. 
\item For each $i \ge 1$, there exists an open neighborhood $ \Omega_i \supset E_i\Gamma/\Gamma$ such that for any compact 
set $F \subseteq \SK \setminus E_{i+1}\Gamma/\Gamma$, there exists $T _{i+1}$ such that for all $x \in F$ and  
$T > T '_i$ we have 
\begin{equation}\label{time-in-sing}
 \lambda_p \{ t \in I_p(T): u_p(t)x \in \Omega_i \cap \SK \} \le \frac{1}{4^i} \lambda_p(I_p(T)). 
\end{equation}

\end{enumerate}

For $i \ge 1$, write $\SK \cap \pi(E_i)= \bigcup \pi(C_{j})$ for some compact sets $C_j \in X(P_j, U)$, $1 \le j \le k$. 
As we are arguing by contradiction, we can find a compact subset $\SF_i \subseteq \SK_1 \setminus \pi(E_{i+1})$ such that for 
each $T _0$ there exists $x \in \SF_i$ and $T  \ge T _0$ such that $ \Delta(\phi,u_p(t),x,T ) > \epsilon$. 
Without loss of generality, assume that $T_1< T_2< \cdots $. This implies that there exists $x_i \in F_i$ and $\sigma_i $ such that $\Delta(\phi, u_p(t),x_i, \sigma_i)> \epsilon. $
From \eqref{time-in-K} and \eqref{time-in-sing}, for each $j \ge 1$ we obtain $t_j \in I(T _j)$ such that $u_p(t_{j} )  x_j \in \SK \setminus \bigcup_{i=1}^{j} \Omega_i$. This implies that 
\[ \Delta( \phi, u_p(t), \sigma_j, y_j) \ge \epsilon  - 2 \frac{T _j}{ \sigma_j}   \ge \frac{\epsilon}{3}. \] 
As $y_j \in \SK$ and $\SK$ is compact, there exists a limit point $y \in \SK$. By construction, $y \not\in \Omega_j$ for all $j \ge 1$. This shows that $y$ is not a singular point for $U$. Now, we can apply Theorem \ref{limit} to a convergent subsequence of 
$\{ y_j \}$ and the corresponding subsequence of $ \sigma_i$, to obtain a contradiction. 
 Let us now turn to the general case. 
Assume that the statement is known for $S_1, S_2 \subseteq S$, and $S_1 \cap S_2 =\emptyset$. We write $U_1=( u_p(t_p))_{ p \in S_1}$ 
and $U_2=( u_p(t_p) )_{p \in S_2}$. 
Note that there exists a compact subset $\SK_1 \subseteq G/\Gamma$ such that for all
$x \in \SK$ and any interval $I(\T_j) \subseteq \QQ_{S_j}$, $j=1,2$, containing zero we have 
\[ \frac{1}{\lambda_{j}(I(\T_j))} \lambda_{j} \left\{  t \in I(\T_j): u_j(t) \in \SK_1  \right\} \ge 1 - \epsilon/4. \]
Here, we have used the shorthands $u_1(t)=(u_p(t_p))_{p \in S_1}$ and $d\lambda_1$ for the Haar measure on $ \prod_{p \in S_1} \QQ_p$.
By the induction hypothesis, there exist finitely many proper subgroups $\GP_1, \dots, \GP_k$ of class $\F$,
and compact subsets $C_i \subseteq X(P_i, U)$, where $P_i= \GP_i(\QQ_S)$, $1 \le i \le k$, such that the following holds:
for any compact subset $\SF$ of $\SK_1 \setminus \cup_i C_i\Gamma/\Gamma$ there exists $T_0$ such that for all $x \in \SF$ and 
$S_1$-vector $\T_1$ with $m(\T_1 )> T_0$, we have
\[ \left|   \frac{1}{\lambda_1(I(\T_1) )} \int_{I(\T_1)} \phi( u_1(t)x) d\lambda_1(t) - \int_{G/\Gamma} \phi d\mu       \right| \le \frac{\epsilon}{4}. \]

Since $C_i\Gamma/\Gamma \subseteq G/\Gamma$ has measure zero, we can choose neighborhoods $N_i$ of $C_i\Gamma/\Gamma$, each of measure at most $ \epsilon/16k$. Now, let $\phi_i$, $1 \le i 
\le k$ be a continuous function such whose restriction to $N_i$ is $1$, and 
$\int_{G/\Gamma} \phi_i < \epsilon/8k$. By applying the induction hypothesis to $ \phi_1, \dots, \phi_k$, we can find 
 finitely many proper subgroups $\GQ_1, \dots, \GQ_l$ of class $\F$,
and compact subsets $D_i \subseteq X(Q_i, U)$, where $Q_i= \GQ_i(\QQ_{S_2})$, $1 \le i \le l$, such that the following holds:
for any compact subset $\SF$ of $\SK \setminus \cup_i D_i\Gamma/\Gamma$, for all $x \in \SF$ and $m(\T _2)$ sufficiently large, we have 
\[ \left|   \frac{1}{\lambda_2(I(\T_2) )} \int_{I(\T_2)} \phi_j( u_2(t)x) d\lambda_2(t) - \int_{G/\Gamma} \phi_j d\mu       \right| \le \frac{\epsilon}{8k}, \quad 1 \le i \le k. \]
Since $\phi_i(x)=1$ for all $x \in N_i$, we obtain
\begin{equation}\label{b2}
\frac{1}{\lambda_2(I(\T_2) )} \lambda_2 \left\{ t_2 \in I(\T _2): u_2(t_2)x \in \bigcup_{i=1}^{k} N_i \right\} \le \frac{ \epsilon}{4}.
\end{equation}

Let $A= \{ t_2 \in I(\T _2): u_2(t_2)x \in \SK_1 \}.$ Note that by the choice of $\SK_1$, 
we have $\lambda_2(A) \ge (1- \epsilon/4) \lambda_2(I(\T _2))$. Combining this with \eqref{b2}, we obtain 
\[\frac{1}{\lambda_2(I(\T_2) )} \lambda_2 \left\{ t_2 \in I(\T _2): u_2(t_2)x \in \SK_1 \setminus \bigcup_{i=1}^{k} N_i \right\}
\ge 1- \frac{\epsilon}{2}. \]

Since  $ \SK_1 \setminus \bigcup_{i=1}^{k} N_i$ is a compact subset of $\SK_1$, disjoint from 
$ \bigcup_{i=1}^{k} C_i\Gamma/\Gamma$, we have 
there exists $T_2$ such that for all $x \in   \SK_1 \setminus \bigcup_{i=1}^{k} N_i$ and 
$m(\T_1)> T_2$, we have
\[ \left|   \frac{1}{\lambda_1(I(\T_1) )} \int_{I(\T _1)} \phi( u_1(t)x) d\lambda_1(t) - \int_{G/\Gamma} \phi d\mu       \right| \le \frac{\epsilon}{4}. \]
Decomposing $u(t)=u_1(t_1)u_2(t_2)$, and using Fubini's theorem, we obtain 
 \[ \left|   \frac{1}{\lambda(I(\T) )} \int_{I(\T)} \phi( u(t)x) d\lambda(t)- 
\int_{G/\Gamma} \phi   \right| \le {\epsilon}. \]

It follows that the collection of $X(\GP_i, U), 1 \le i \le k$ and $X(\GQ_j,U), 1 \le j \le l$ will satisfy the conditions of the theorem. 
\end{proof}

\section{Equidistribution of translated orbits}
Let $\SG$, $\SH$, and $\SK$ be as before. In this section, we prove the equidistribution of the translated orbit $a_\t \SK \sg$ in $\SG/\Gamma$, where $\sg \in G$ is fixed (or varies over a compact set), and $\t$ tends to infinity. The main result will be proven in two steps. In the first step--dealt with in this section--we assume that the test function in Theorem \ref{ergodic-a} is bounded. The main argument, which is purely dynamical, is based on approximating the integral above, with a similar one involving orbits shifted by unipotent elements. The advantage is that the latter can be handled using an $S$-arithmetic version of Dani-Margulis theorem, which is generally false for unbounded functions. We will eventually need to prove Theorem \ref{ergodic-a} for test functions with a moderate blowup at the cusp. This step requires an analysis of certain integrals involving the $\alpha$-function, which is carried out in a later section.

\begin{theorem}\label{ergodic-a-bounded}
Let $n \ge 3$ be a positive integer, and let $\SG$, $\SH$, $\SK$ and $a_\t$ be as above. 
Let $\phi: \SG/\Gamma \to \RR$ be a bounded continuous function, and $q_p$ be a positive continuous function on $\U_p^n$, and $q=  \prod_{p \in S} q_p$. Let $x_0 \in \SG/\Gamma$ be such that the orbit $\SH x_0$ is not closed. Then 
\[ \lim_{ \t \to \infty} \int_\SK \phi(a_\t \sk x_0) q(\sk) dm(\sk) =
\int_{\SG/\Gamma} \phi(y) d\mu(y) \int_\SK q(k) \, dm(\sk). \]
Here $d\mu$ denotes the $\SG$-invariant probability measure on $\SG/\Gamma$. 
\end{theorem}

The proof of this theorem closely follows \cite{DM}. %Among other things, the proof uses a version of a theorem of Dani and Margulis, which is adapted to the $S$-arithmetic setting. More precisely, 
The integral of $\phi$ along the $a_\t$-translates of the $\SK$-orbits is approximated by a similar
integral with $a_\t$ replaced by the unipotent flow to which a generalized theorem of Dani and Margulis will be applied. One difference between the situation here and the one in \cite{EMM} is that as $\{a_\t\}$ is a $\ZZ$-flow, the orbit needs to be ``fattened up'' in order to obtain rough approximations of the orbit of a unipotent flow, which is similar to intervals in $\QQ_p$.

Fix a prime $p \in S_f$. For simplicity, let us suppress the dependence of matrices introduced below on $p$. For $ \alpha \in \U_p= \{ \alpha: |\alpha|=1 \}$ and $t \in \ZZ$, we set

\[ d_t= \begin{pmatrix}
p^{-t}  & 0    \\
0 & p^t \\
\end{pmatrix} ,
\quad
u_{t, \alpha}= \begin{pmatrix}
1  &  \alpha p^{-t}    \\
0 & 1 \\
\end{pmatrix}, \quad
k_{t, \alpha}= \begin{pmatrix}
1  &  \alpha p^t    \\
- \alpha^{-1} p^t & 1-p^{2t} \\
\end{pmatrix},
 \]
 \[ 
 b_{t, \alpha}= \begin{pmatrix}
1  & 0    \\
\alpha^{-1} (p^{3t}-p^t) & 1 \\
\end{pmatrix}, \quad w_{ \alpha}= \begin{pmatrix}
0  & - \alpha    \\
 \alpha^{-1} & 0 \\
\end{pmatrix}.
\]
Note that $d_t= b_{t, \alpha} u_{t, \alpha} k_{t, \alpha} w_{ \alpha}$, $\{k_{t,\alpha} : t \in \NN, \alpha \in \CA\}$ and $\{ w_{ \alpha}: \alpha \in \CA \}$ 
{embeds into $K_p$ (with the embedding of $\SL_2(\QQ_p)$ into $\SO(\q)$).}  This is parallel to a similar formula in the real case. The only different is that $w$ is constant in the real case, whereas here it depends on $\alpha$.

\begin{proof}[Proof of Theorem \ref{ergodic-a-bounded}]
Without loss of generality, we may assume that $\phi$ and $q$ are uniformly 
continuous. 
We first prove an analogous statement for $(u_\sz)_{ \sz \in \QQ_S}$ instead of $(a_\t)_{\t \in \T_S}$. We will then deduce the statement for the latter from a similar statement for the former by an approximation process. 
First, we will show that under the assumptions of Theorem \ref{ergodic-a}, we have

{\bf Claim 1:} There exist $x_1, \dots, x_{\ell} \in \SG/\Gamma$ such that the orbits $\SH x_i$ are closed and each carries a finite $\SH$-invariant measure, such that
for each compact set $\SF \subseteq \SD - \bigcup_{i=1}^{\ell} \SH x_i$, for any 
$\T$ with $m(\T)$ sufficiently large, and 
all $x \in \SF$, we have
\[ m \left\{ \sk \in \SK:   \left|  \frac{1}{\lambda_S(\I(\T))}
\int_{\I(\T)} \phi( u_\sz \sk x) \ d\sz - \int_{\SG/\Gamma} \phi \ d\mu  \right|  > \epsilon   \right\} 
< \epsilon. \]
Indeed, 
let  $\SP_1, \dots, \SP_k$ be the subgroups provided by Theorem \ref{main-DM} 
for the data $(\phi, \SK \SD, \epsilon)$, and let $C_i \subseteq X(\SP_i, \SU)$ be the corresponding compact sets. Define 
\[ Y_i= \{ \sg \in \SG: \SK \sg \subseteq X(\SP_i, \SU) \}. \]
Let $M$ be the subgroup of $\SG$ generated by $\sk^{-1}\SU \sk$, where $\sk \in \SK$. The normalizer $L=N_\SG (M)$ contains $\SK$, hence is of the form $L=E\SK$ where $E$ is a normal subgroup of $\SH$. Since $\SU \subseteq L$ and $U$ is not included in any proper normal subgroup of $\SH$, we obtain $M=\SH$. 

Note that for any $\sy \in Y_i$, we have $\sk^{-1}\SU \sk \subseteq \sy \SP_i \sy^{-1}$, for all $\sk \in \SK$. So $ \SH \subseteq  \sy \SP_i \sy^{-1}$. 
It follows that there exists a subgroup of finite index $\SP_i'$ in $\SP_i$ such that $\SH= \sy \SP_i' \sy^{-1}$ for 
all $ \sy \in Y_i$. This implies that $\sy_1^{-1}\sy_2 \in N_\SG(\SH)$ for all $\sy_1, \sy_2 \in Y_i$. Hence $Y_i$ can be covered with a
finite number of right cosets of  $\SH$. Therefore,
\[ \bigcup_{ 1 \le i \le k} Y_i\Gamma/\Gamma \subseteq \bigcup_{ 1 \le j \le l} \SH x_j. \]
Note that $\SP_i \cap \Gamma$ is a lattice in $\SP_i$. Since $X(\SP_i, \SU)$ is an analytic subset of $\SG$, we have
\[ m( \{ \sk \in \SK: \sk\sg \in X( \SP_i, \SU) \gamma \})=0  \]
for each $\sg \in \SG - Y_i \gamma$. Therefore
\[ m \left\{ \sk \in \SK: \sk \sg \in \bigcup_{1 \le i \le k} X( \SP_i, \SU)\Gamma \right\}=0 \]
holds for any $\sg \in \SG - \bigcup_{1 \le i \le k} Y_i\Gamma$. 
Since $C_i \subseteq X(\SP_i,\SU)$ and $\SF \subseteq \SD- \bigcup_{j=1}^{\ell} \SH x_j$, 
we have
\[  m  \left\{ \sk \in \SK: \sk x \in \bigcup_{i=1}^{k} C_i \Gamma/\Gamma \right\} =0, \]
for every $x \in \SF$. From here it easily follows (using an analogue of Lemma 4.2
in \cite{EMM}) that there exists an open subset $W \subseteq \SG/\Gamma$ containing 
$ \bigcup_{i=1}^{k} C_i \Gamma/\Gamma$ such that 
\[ m ( \{ \sk \in \SK: \sk x \in W \})< \epsilon \]
for any $x \in \SF$. From the choice of $C_1, \dots, C_k$, we have that for 
every $ y \in \SK \SD - W$ and all $\T$ with $m(\T) \gg 0$, we have 
\begin{equation}\label{eq:DM}
 \left|   \frac{1}{\lambda_S(\I(\T)) } \int_{\I(\T)} \phi( u_\sz x) d\lambda_S(\sz) - \int_{\SG/\Gamma} \phi d\mu       \right| \le \epsilon. 
\end{equation}

Let $q$ be a bounded continuous function on $\SK$. It follows that
\begin{equation}
\left|   
 \frac{1}{\lambda_S(\I(\T)) }  \int_\SK  \int_{\I(\T)} \phi( u_\sz x) q(\sk) d\lambda_S(\sz) - \int_{\SG/\Gamma} \phi d\mu  \int_\SK q(\sk) \; dm(\sk) \right| 
\le \sup_{\sk \in \SK} |q(\sk)| (1 + 2 \sup_{y \in \SG/\Gamma} |\phi(y)|) \epsilon. 
\end{equation}

%In the rest of this proof, we will use the following notation:
%\[ d_t= \begin{pmatrix}
%p^{-t}  & 0    \\
%0 & p^t \\
%\end{pmatrix} ,
%\quad
%u_{t, \alpha}= \begin{pmatrix}
%1  &  \alpha p^{-t}    \\
%0 & 1 \\
%\end{pmatrix}, \quad
%k_{t, \alpha}= \begin{pmatrix}
%1  &  \alpha p^t    \\
%- \alpha^{-1} p^t & 1-p^{2t} \\
%\end{pmatrix},
% \]
% \[ 
% b_{t, \alpha}= \begin{pmatrix}
%1  & 0    \\
%\alpha^{-1} (p^{3t}-p^t) & 1 \\
%\end{pmatrix}, \quad w_{ \alpha}= \begin{pmatrix}
%0  & - \alpha    \\
% \alpha^{-1} & 0 \\
%\end{pmatrix}
%\]

We will now turn to the approximation part. Fix $p \in S_f$.  
In what follows $d\alpha$ is the Haar measure on $\ZZ_p$ normalized such that 
$\int_\CA d\alpha =1$. First note that 
\begin{equation}
\begin{split}
\int_{K_p} \phi(d_tkx) q(k) \ dm(k) &= \int_\CA \int_{K_p}  \phi(b_{t, \alpha} u_{t, \alpha} k_{t, \alpha}
w_{ \alpha} kx) q(k) dm(k) d\alpha \\
&= \int_\CA \int_{K_p} \phi(b_{t, \alpha} u_{t, \alpha} k_{t, \alpha}
w_{ \alpha} kx) q(k) dm(k_{t, \alpha} w_{ \alpha} k) d\alpha \\
&= \int_\CA \int_{K_p} \phi(b_{t, \alpha} u_{t, \alpha} kx) q( w_{ \alpha}^{-1} k_{t , \alpha}k) dm(k) d\alpha. 
\end{split}
\end{equation}

%For $y \in \QQ_p$, we can write
%\begin{equation}
%\begin{split}
% \left|  \int_K \phi( d_t kx) q(k) \ dm(k)-   \int_A  \int_K    \phi( u_{t+y, \alpha} kx)
% q( w_{ \alpha}^{-1}k) \ dm(k)  \right|  &  \\
% &\le \left| (\phi( d_{t+y} kx ) - \phi(d_t kx) ) q(k) \ dm(k) \int_K   \right| +
% \left|   \int_K (\phi( d_{t+y} kx ) - \phi( u_{r+y, \alpha} kx) ) q( w_{ \alpha}^{-1}k)
% \ dm(k).  \right|  \\
%\end{split}
%\end{equation}

We can now write
\begin{equation}
\begin{split}
&\left|   \int_{K_p} \phi(d_tkx)q(k) \ dm(k)- \int_\CA \int_{K_p} \phi( u_{t, \alpha} kx)
q( w_{ \alpha}^{-1}k) \ dm(k) d\alpha
  \right|  \\
& \hspace{1in}=
\left|   \int_\CA \int_{K_p} \phi(b_{t, \alpha}u_{t, \alpha}kx)  q(w_{\alpha}^{-1}
k_{t, \alpha} k) - 
 \phi( u_{t, \alpha} kx)
q( w_{ \alpha}^{-1}k) \ dm(k) d\alpha
  \right|. 
\end{split}
\end{equation}
Note that as $t \to \infty$, $b_{t, \alpha}$ and $k_{t, \alpha}$ converge 
to $1$ uniformly in $ \alpha$. Hence, for large values of $t$, we have
\[ \left|   \int_{K_p} \phi(d_tkx)q(k) \ dm(k)- \int_\CA \int_{K_p} \phi( u_{t, \alpha} kx)
q( w_{ \alpha}^{-1}k) \ dm(k) d\alpha
  \right| < \frac{ \epsilon}{2}. \]
  Since $q$ is continuous, we can partition $\CA$ into intervals
  $I_1, \dots, I_r$ such that for each $1 \le i \le r$, there exists
  $w_i \in K_p$ such that $|q( w_{ \alpha}^{-1}k)-q(w_i k) |< \epsilon/4$ for every $ \alpha \in I_i$. This implies that 
  \[ \left|   \int_{I_i} \int_{K_p}  \phi( u_{t, \alpha} kx) q( w_{ \alpha}^{-1}k)
 -  \phi( u_{t, \alpha} kx)  q( w_i k)  \ dm(k) d\alpha  \right|< \lambda_p(I_i) \epsilon/4. \]
It is clear that $I_{i,t}= \{ p^{-t} \alpha: \alpha \in  I_i \}$ has a measure comparable to $p^t$. So, for sufficiently large $t$, we  
  have
  \[ \int_{I_i} \int_{K_p}  \phi( u_{t, \alpha} kx)  q( w_i k)  \ dm(k) d\alpha  
  = \int_{I_{i,t}} \int_{K_p} \phi( u_{t} kx)  q( w_i k) d\theta(t). \]
  We can now apply \eqref{eq:DM} to obtain the result.

\end{proof}

\section{Establishing upper bounds for certain integrals}
In this section, we will prove Theorem \ref{bound}. In the course of the proof we will need to consider a number of integrals that are intimately connected to the integral
featured in Theorem \ref{bound}. In the rest of this section, we fix $p \in S_f$ and 
denote by {$a_t $ the linear transformation $a_t^p$ defined in equation (\ref{eqn:a_t}): $a_t\ve_1=p^t \ve_1, \;  a_t\ve_n=p^{-t}\ve_n,$  and $a_t \ve_j= \ve_j , \   2 \le j \le n-1. $}
%\begin{theorem}[$p$-adic Implicit Function Theorem]\seon{Cite Serre. Se2}
%Let $X$ and $Y$ be $m$- and $n$-dimensional $p$-adic manifolds, $x
%\in X, y \in Y$ and $\phi: X \to Y$ be a continuous function that is
%locally given by power series, such that $ \phi(x)=y$. Then the
%followings are equivalent:
%\begin{enumerate}
%\item The rank of the linear map $d_x\phi: T_xX \to T_y Y$ is $r$.
%\item There exist coordinate systems $(x_i)$ and $(y_i)$ at $x$ and $y$, respectively, such that
%\[ \phi(x_1, \dots , x_m)= (x_1, \dots , x_r, \psi_{r+1}(x_1, \dots ,
%x_m), \ldots, \psi_{n}(x_1, \dots , x_m) ), \] where $\psi_{r+1} , \dots,
%\psi_n$ are analytic functions.
%\end{enumerate}
%\end{theorem}

%%%%%%%%%%%%%%%%%%%%%%%%%%%%%%%%%%%%%%%%%%%%%%%%%%%%%%%%%%%%%%%%%%%%%%%%%%%%
%\begin{lemma}\label{trans}
% Let $V$ be a finite dimensional vector space over $\QQ_p$ and $ K$ be
% a compact subgroup of $\GL(V)$. Let $W$ be a proper subspace of
% $V$ such that for any finite index subgroup $ K'$ of $ K$ and any subspace $W'$ of $W$,
% $K'W'\nsubseteq W'$. Then for any $v\in V$, the subset
% \begin{equation}
%\tran(v, W)=\{k\in  K : kv\in W\}\subseteq  K
% \end{equation}
% has Haar measure zero.
%\end{lemma}

\begin{lemma}\label{trans}
 Let $V$ be a finite dimensional vector space over $\QQ_p$ and let $K$ be
 a compact subgroup of $\GL(V)$. Let $v \in V$ and $W$ be a proper subspace of
 $V$ such that for any finite index subgroup $ K'$ of $ K$ and any non-zero subspace $W' \subset W$, $K' W' \not\subseteq W'$. Then  the subset
 \begin{equation}
\tran(v, W):=\{k\in  K : kv\in W\}\subseteq  K
 \end{equation}
 has Haar measure zero.
\end{lemma}

\begin{proof} Denote the Haar measure on $ K$ by $m$. Since $ K$ is compact, $m( K)<\infty$.
Suppose that $m(\tran(v, W))>0$. Take
$W'\subseteq W$ of the smallest dimension such that $\tran(v,
W')$ has a positive measure. For each $k\in  K$, set $k W'=\{kw : w\in W'\}$. For each $k'\in  K$, we have
$\tran(v, k'W')=k'\tran(v, W')
$ and hence
\[m(\tran(v, k'W'))=m(\tran(v, W'))\]
 for all $k' \in  K$. If
$k'W'\neq k'' W'$, since $W'\cap  {k'}^{-1} k'' W'$ has dimension strictly lower
than that of $W'$, by minimality, the intersection
\begin{equation*}
\begin{split}
\tran(v, k'W')\cap\tran(v, k'' W')=&\; k' \tran(v, W') \cap k' \tran(v, {k'}^{-1} k''W')   \\
\subseteq&\; k' \tran(v, W'\cap {k'}^{-1} k'' W')
\end{split}
\end{equation*}
has measure zero. 

Thus $\{k W' : k\in  K\}$ is a finite set and $K' : = \{ k: kW' =W' \}$ is of finite index in $K$. Then $K'W' = W' \subset W'$, contradicting to the condition that $K'W' \nsubseteq W'$. 
\end{proof}
%%%%%%%%%%%%%%%%%%%%%%%%%%%%%%%%%%%%%%%%%%%%%%%%%%%%%%%%%%%%%%%%%%%%%%%%%%%%%

%%%%%%%%%%%%%%%%%%%%%%%%%%%%%%%%%%%%%%%%%%%%%%%%%%%%%%%%%%%%%%%%%%%%%%%%%%%%%
\begin{lemma}\label{05:1}
 Let $\rho$ be an analytic representation of $\SL_n(\QQ_p)$ on a finite dimensional normed space $(V, \| \ \|)$ over
$\QQ_p$ such that all elements of $\rho\left(\SL_n(\ZZ_p)\right)$
preserve the norm on $V$. For $g\in\SL_n(\QQ_p)$ and $v\in V$, we
will denote $\rho(g)(v)$ by $gv$. Let $ K$ be a closed analytic subgroup of $\SL_n(\ZZ_p)$.   Assume that $V$ has a
decomposition 
\begin{equation}\label{V-split}
V=W^{-}\oplus W^0\oplus W^{+}, 
\end{equation}
where
\begin{eqnarray*}
%W^{-}&=&\left\{\ov\in V : a_t\ov=p^t \ov\right\},\\
W^0&=&\left\{v\in V : a_tv=v, t \in \ZZ\right\}  \\
W^{\pm}&=&\left\{v\in V : a_tv=p^{\mp t} v , t \in \ZZ\right\}.
\end{eqnarray*}

 Let $Q$ be a closed subset of $\left\{v\in V: \|v\|_p=1\right\}$.
 Suppose that the following conditions are satisfied:
 \begin{enumerate}
\item The subspace $W^{-}+W^0$ satisfies the assumption of Lemma
\ref{trans};
\item There is a positive integer $\ell$ such that for any nonzero $v\in W^-$,
\begin{equation}
\mathrm{codim}\left\{x\in \Lie( K): xv\in W^-\right\}\geq
\ell.\label{eq05:1}
\end{equation}
\end{enumerate}

  Then for any s, $0<s<\ell$,
  \begin{equation*}
\lim_{t\rightarrow+\infty}\sup_{v\in Q}\int_{ K}\frac
  {{dm}(k)}{\|a_t kv\|^s}=0, \label{eq05:2}
 \end{equation*}
  where $ m$ is the normalized Haar measure on $ K$.
\end{lemma}

\begin{proof}
Let $\pr:V\rightarrow W^+\oplus W^0$ and
$\pr^+:V\rightarrow W^+$ denote the natural projections associated to the decomposition 
\eqref{V-split}. For any $v\in V$ and $r\geq 0$, define
\begin{equation*}
\begin{split}
D(v,r) & =\left\{k\in  K: \|\pr(kv)\|\leq r\right\} \\
D^+(v,r) &=\left\{k\in  K: \|\pr^+(kv)\|\leq r\right\}.
\end{split}
\end{equation*}

We will show that the measure of $D(v, r)$ is bounded
by $C r^\ell$, where $C$ is uniform over all $v \in Q$.
Recall that the Lie algebra $\Lie(K)$ is defined as the tangent space to the
$p$-adic analytic manifold $K$ at the identity. Consider the map $f_{v}:  K \to W^+
\oplus W^0$ defined by $ f_{v}(k)=\pr(kv)$. Since $ K$ acts by
linear transformations, the derivative of $f_{v}$ at the identity is
given by
\begin{equation} d_e f_{v}(x)= \pi(xv), \quad x \in
\Lie(K)=T_e K \subseteq \sl_n(\QQ_p).
\end{equation}

For every $v \in V$, set
\[ L_{v}= \{ x \in \Lie(K): xv \in W^- \}= \ker d_e f_{v}. \]
For $k \in  K$, also define the map $m_k:  K \to  K$ by $m_k(k')=k'k$.
Note that $m_k(e)=k$ and $d_e m_k: T_e K \to T_k K$ is an isomorphism.
Differentiating $f_{kv}= f_{v} \circ m_k$ yields
$d_e f_{kv}= d_k(f_{v}) \circ d_em_k$, showing that
$ \rk d_e f_{kv}= \rk d_k f_{v}.$ From the assumption, if $f_{v}(k)=0$, then $kv \in W^-$, hence by
the assumption of the theorem, $\codim \ker d_ef_{kv} \ge \ell$,
which implies that $\codim \ker d_k f_{v} \ge \ell$.

Let us denote by $\PP(V)$ the projective space on $V$, and by $[v]$ {}{the} image of the non-zero 
vector $v$ in $\PP(V)$. We know that if $f: U \to V$ is a $p$-adic analytic function, then the map $x \mapsto \rk d_xf$ is lower semi-continuous, that is, $\liminf_{ y \to x} \rk d_yf  \ge \rk d_xf$. This implies that for every $v \in W^-$, there exists an open subset
$U$ containing $[v] \in \PP(V)$ such that for $ [w] \in U$, we have
$\codim \ker d_k f_{w} \ge \ell$. By compactness, there exists $\epsilon>0$ such if $\| p(kv) \| \le \epsilon \|v \|$, then the map $f_{v}$ has a rank at least $ \ell$.
Using the implicit function theorem for local fields (e.g. \cite{Se2}), for any $k \in K$ for which $\| \pr
(kv) \| \le \epsilon \| v \|$, we can choose local coordinate
systems for $\mathcal U_k\subset K$ at $k$ and for $W^+ \oplus W^0$
in which $f_{v}$ is given by
\[ f_{v}( u_1, \dots , u_m)= (u_1, \dots , u_\ell, \psi_{\ell +1}, \dots,
\psi_{m'}), \] where $m=\dim K$, $m'=\dim W^0+\dim W^+$ and
$\psi_{j}(u_1, \dots , u_m)$ is an analytic function for
$j=\ell+1,\dots,m'$ depending smoothly on $v$.
If $r\leq \epsilon$, for such a neighborhood $\mathcal U_k$, $k\in
D(v,r)$, there is a nonnegative $\alpha=\alpha(k)$ satisfying
\[m\left(D(v,r)\cap \mathcal U_k\right)\sim \left(\frac r
\epsilon\right)^{\ell+\alpha}m\left(D(v,\epsilon)\cap \mathcal
U_k\right)
\]
so that
\[\frac {m\left(D(v,r)\right)}{r^{\ell}}\leq \frac 1
{\epsilon^\ell}m\left(D(v,r)\right).
\]

 If $r\geq \epsilon$, {}{since $m(D(v,r)) \leq 1$,}
\[\frac {m\left(D(v,r)\right)}{r^{\ell}}\leq \frac 1
{\epsilon^{\ell}}.
\]

 By compactness of the ball $\left\{v\in V: \|v\|=1\right\}$, we have
\begin{equation}
C:= \sup_{\|v\|=1,
r>0}\frac{m\left(D(v,r)\right)}{r^{\ell}}<\infty.\label{eq5:03}
\end{equation}
We will now have to consider the set
\[ D^+(v,0)= \{ k \in  K: p^+(kv)=0 \}= \{ k \in  K: kv \in W^0 \oplus W^- \}. \]

By Lemma \ref{trans}, we have $m( D^+(v,0))=0$. We claim that
this and the compactness of $Q$ imply that
\begin{equation}
 \lim_{ r\to 0} \sup_{v \in Q} m( D^+(v,r))=0. \label{eq5:04}
 \end{equation}

Otherwise, pick a sequence $r_n \to 0$, and $v_n \in Q$ such that
$m(D^+(v_n, r_n))> \epsilon$ for some $ \epsilon$. Upon passing to a
subsequence, we can assume that $v_n$ converges to $v$.
\[ \epsilon \le m(\{ k: \| \pr^+(kv_n) \| \le r_n  \})  \subseteq
m(\{ k: \| \pr^+(kv) \| \le r_n+ \| v_n - v\| \}). \]

 Since $ r_n+ \| v_n
- v\| \to 0$, we obtain $ m(D^+(v, 0)) \ge \epsilon$, which is a
contradiction.\\

  Since $a_t|_{W^0}={Id}_{W^0}$ and \eqref{eq5:03}, we see that for a positive integer
$t$,
%  $a_t|_{W^+}=p^{-t}\;{Id}_{W^+}$, we see that for $t>0$,
%\[ \|a_t kv \|^s \ge \max \{ \| P(kv) \|_p^s, p^{ts} \| p^+(kv) \|_p^s \}. \]
$\|a_t kv\|^s\geq \|\pr(kv)\|^s$ hence
\[\int_{D(v,r)- D\left(v, \frac r p\right)}\frac {dm(k)}{\|a_t kv\|^s}\leq
\frac{m\left(D(v,r)\right)}{(r/p)^{s}}\leq
\left(p^{-s}C\right)r^{\ell-s},
\]
for any $v\in V$, $\|v\|=1$ and any $r>0$. From the fact that
$m\left(D(v,0)\right)=0$,
\begin{eqnarray}
\int_{D(v,r)}\frac {dm(k)}{\|a_t kv\|^s}&=&\sum_{n=0}^\infty
\int_{D\left(v, p^{-n}r\right)\backslash D\left(v,
p^{-n-1}r\right)}\frac {dm(k)}{\|a_t kv\|^s}\label{eq5:-5}\\
&\leq&\frac {C}{p^{-s}-p^{-\ell}}r^{\ell-s}.\nonumber
\end{eqnarray}

On the other hand, since $a_t|_{W^+}=p^{-t}\;{Id}_{W^+}$, we
get that $\|a_t kv\|^s\geq p^{ts}\|\pr^+(kv)\|^s$ so that for any
$v\in V$, $r>0$ and $r_1>0$,
\begin{eqnarray*}
\int_{ K\backslash D(v,r)}\frac{dm(k)}{\|a_t
kv\|^s}&\leq&\int_{ K-D^+(v,r_1)}\frac {dm(k)}{\|a_t
kv\|^s}+\int_{D^=(v,r_1)\backslash D(v,r)}\frac {dm(k)}{\|a_t
kv\|^s}\\
&\leq&p^{-ts}r_1^{-s}+m\left(D^+(v,r_1)\right)r^{-s}.
\end{eqnarray*}

  Therefore for a given $\epsilon_0>0$, since $\ell-s>0$, take $r>0$ and $r_1>0$ such
  that $Cr^{\ell-s}/\left(p^{-s}-p^{-\ell}\right)<\epsilon_0/3$ and
  $m\left(D^+(v,r_1)\right)r^{-s}<\epsilon_0/3$. Then we can take
  $t'>0$ such that if $t>t'$, then $p^{-ts}r_1^{-s}<\epsilon_0/3$ so
  that
  \[\int_{ K} \frac{dm(k)}{\|a_t kv\|^s}<\epsilon_0.
  \]
\end{proof}
%%%%%%%%%%%%%%%%%%%%%%%%%%%%%%%%%%%%%%%%%%%%%%%%%%%%%%%%%%%%%%%%%%%%%%%%%%%

We now write $V_i$ for the $i$-th exterior power $\bigwedge^i \QQ_p^n$, and
denote by $\rho_i$  the $i$-th exterior power of the natural representation 
$\SL_n(\QQ_p)$ on $\QQ_p^n$. Recall that for an $i$-tuple $J=(j_1, \dots, j_i)$, we define
$\ve_J=\ve_{j_1} \wedge \cdots \wedge \ve_{j_i}$. It is easy to see that $V_i$ 
decomposes into $a_t$-eigenspaces $V_i=W_i^-\oplus W_i^0\oplus W_i^+$, and the eigenspaces
are given by
\begin{eqnarray*}
W_i^-&=& \spn \left\{e_{j_1 \cdots j_i} : j_1=1, j_i<n \right\},\\
W_i^+&=& \spn \left\{e_{j_1 \cdots j_i} :
1<j_1, j_i=n \right\}, \\
W_i^0&= & \spn \left\{e_{j_1 \cdots j_i} : j_1=1, j_i=n \text{ or } j_1>1, j_i<n
 \right\}. \\
\end{eqnarray*}
%%%%%%%%%%%%%%%%%%%%%%%%%%%%%%%%%%%%%%%%%%%%%%%%%%%%%%%%%%%%%%%%%%%%%%%%%%%%%

\begin{proposition}\label{05:2+4} Let $\q_p$ be the standard quadratic form on $\QQ_p^n$, $n\geq 4$, defined as in Subsection~\ref{subsec:quad}.
% defined by
%\begin{equation}\q(x_1, x_2, \ldots,x_n)= x_1x_n+\alpha_2 x_2^2+\cdots+\alpha_{n-1}x_{n-1}^2,\label{quadform}
%\end{equation}
%where $\alpha_j\in\left\{1, u_0, p, u_0p\right\}$, $2\leq j\leq n-1$.
Set $K = \SL_n(\ZZ_p) \cap \SO(\q_p)$.
 Let {}{$A=A_p$} and $V_i$, $W_i^-$, $W_i^0$ and $W_i^+$ be as above. For each
 $V_i$, define
 \begin{equation}F(i)=\left\{\ox_1\wedge \ox_2 \wedge \cdots \wedge \ox_i : \ox_1, \ox_2, \cdots, \ox_i\in
 \QQ_p^n\right\}\subset V_i \label{F(i)}
 \end{equation}
  and $Q_i=F(i)\cap \left\{v\in V_i : \|v\|_p =1\right\}$. Then for any $s$, $0<s<2$,
  \begin{equation}
\lim_{t\rightarrow\infty}\sup_{v\in Q_i}\int_{ K}\frac
  {dm(k)}{\|a_tkv\|^s}=0. \label{eq05:2+4}
  \end{equation}
\end{proposition}

\begin{proof}
Let us first consider the special case of $n=4$, in which $\q_p$ is given by $\q_p(x_1, x_2, x_3, x_4)=x_1x_4+a_2 x_2^2+a_3 x_3^2$, where $a_i \in \{\pm 1, \pm u_0,
\pm p, \pm u_0p\}$. Here $u_0 \in \U_p$ is a representative of a quadratic non-residue in the quotient $\ZZ_p/ p\ZZ_p$. 
Define the following subgroups of $K$.

\begin{equation} \mathbf{U}=\left\{ u_z= \left(\begin{array}{cccc}
1 & 0 & 0  & 0 \\
z & 1 &  0 & 0\\
 0 & 0 & 1  & 0 \\
-a_2z^2 & -2a_2z& 0  & 1\\
 \end{array}\right) : z \in \ZZ_p\right\}, 
  \label{U2}
\end{equation}
\begin{equation} \mathbf{U}'=\left\{ u'_z= \left(\begin{array}{cccc}
1 & 0  & 0  & 0 \\
 0 & 1 & 0 & 0\\
z  & 0 & 1 &  0 \\
-a_3z^2 & 0  & -2a_3z& 1\\
 \end{array}\right) : z \in \ZZ_p \right\}.\label{U3}
\end{equation}

Note that if $ K'$ is a finite index subgroup of $ K$, then it is open in $ K$, thus
 $ K' \cap  \mathbf{U} $, $K' \cap \mathbf{U}'$ are open sets in $\mathbf{U}$ and $\mathbf{U'}$, respectively. 

For $i=3$, $W^- = \langle \ve_{123} \rangle$ and $W^0 = \langle \ve_{124}, \ve_{134} \rangle$.
For $\ov = x_{123} \ve_{123} + x_{124} \ve_{124} + x_{134} \ve_{134},$
\[ u_z\ov= x_{123}  \ve_{123}+(2a_2z x_{123} + x_{134}) \;\ve_{134}+(a_2z^2 x_{123} + z x_{134})\;\ve_{234} + x_{124} \ve_{124} ,\]
\[u'_z\ov 
= x_{123} \ve_{123} +( x_{124}-2a_3z x_{123} ) \ve_{124}+(a_3z^2 x_{123}-zx_{124}) \ve_{234} + x_{134} \ve_{134}.
 \]
If there is a subspace $W' \subset W^- \oplus W^0$ such that $K'W' \subset W'$, then for any sufficiently small $z_1, z_2$,
the elements $u_{z_1} \ov, u'_{z_2} \ov \in W^- \oplus W^0$, for any $\ov \in W'$, which implies that the coefficients of $\ve_{234}$ vanish. It implies that $x_{123} = x_{134} = x_{124}$, thus $W' = \{0\}$ and the condition (1) is satisfied.

For $i=2$, $W^- = \langle \ve_{12}, \ve_{13} \rangle$ and $W^0 = \langle \ve_{23}, \ve_{14} \rangle$. For $\ov= x_{12}\ve_{12} +  x_{13} \ve_{13} +x_{23}\ve_{23}+ x_{14} \ve_{14}$, we have
\begin{equation*}
\begin{split}
u_z\ov
=  
 & x_{12}\ve_{12} + 2a_2zx_{12}\ve_{14}+(z x_{14}-a_2z^2x_{12}) \ve_{24}   \\
& +x_{13}\ve_{13}+(zx_{13}+x_{23})\ve_{23} +a_2(z^2x_{13}+2zx_{23}) \ve_{34} + x_{14} \ve_{14},\\
u'_z\ov = 
 &x_{12}\ve_{12}-(zx_{12}+x_{23})\ve_{23}+a_3(z^2x_{12}-2z x_{23})\ve_{24}\\
  &+x_{13}\ve_{13}- 2a_3zx_{13}\ve_{14}+(zx_{14}-a_3z^2x_{13})\ve_{34} + x_{14} \ve_{14}.
\end{split}
\end{equation*}
As above, if for any sufficiently small $z_1, z_2$,
the elements $u_{z_1} \ov, u'_{z_2} \ov \in W^- \oplus W^0$, then
the coefficients of $\ve_{24} , \ve_{34}$ vanish, i.e.
$$x_{14}-a_2 z_1 x_{12} = z_1 x_{13} + 2 x_{23}= z_2 x_{12} - 2 x_{23} = x_{14} - a_3 z_2 x_{13} = 0.$$
Thus if $x_{14}=0$, then $x_{12} = x_{13} = x_{23} =0$. If $x_{14} \neq 0$, then there is no solution. Thus the condition (1) of Lemma~\ref{05:1} is satisfied.
The case $i=1$ follows immediately.

Now let us check the condition (2) of Lemma~\ref{05:1}.
Define the following elements of $Lie(K)$:

\begin{equation}  u_1= \left(\begin{array}{cccc}
0 & 0 & 0  & 0 \\
1 & 0 &  0 & 0\\
 0 & 0 & 0 & 0 \\
0 & -2a_2 & 0  & 0\\
 \end{array}\right), 
  \label{U2}
\end{equation}
\begin{equation} u_2= \left(\begin{array}{cccc}
0 & 0  & 0  & 0 \\
 0 & 0 & 0 & 0\\
1  & 0 & 0 &  0 \\
0 & 0  & -2a_3& 1\\
 \end{array}\right).
\end{equation}

%
%\begin{equation} U_2=\left\{ u_2(t_2)= \left(\begin{array}{cccc}
%1 & 0 & 0  & 0 \\
%t_2 & 1 &  0 & 0\\
% 0 & 0 & 1  & 0 \\
%-a_2t_2^2 & -2a_2t_2& 0  & 1\\
% \end{array}\right) : t_2 \in \QQ_p\right\} 
%  \label{U2}
%\end{equation}
%\begin{equation} U_3=\left\{u_3(t_3)=\left(\begin{array}{cccc}
%1 & 0  & 0  & 0 \\
% 0 & 1 & 0 & 0\\
%t_3  & 0 & 1 &  0 \\
%-a_3t_3^2 & 0  & -2a_3t_3& 1\\
% \end{array}\right) : t_3 \in \QQ_p \right\}.\label{U3}
%\end{equation}

For $i=3$, for $\ov= x_{123}\ve_{123}$
\[ (z_1 u_1 + z_2 u_2) (\ov) = 2a_2 z_1 \ve_{134}  - 2a_3 z_2 \ve_{124} .\]
It is clear that any non-zero linear combination $z_1u_1+ z_2 u_2$ is not in the set
$\{ x \in Lie(K) : x\ov \in W^- \}.$

For $i=2$, for $\ov= x_{12}\ve_{12} +  x_{13} \ve_{13} $, we have
\begin{equation*}
(z_1 u_1 + z_2 u_2) (\ov)
= x_{12}(-z_2 \ve_{23} - 2 a_2 z_1 e_{14}) + x_{13}(z_1 \ve_{23}-2a_3z_2 \ve_{14}).  
\end{equation*}
Suppose that the form $\q$ is not exceptional, in other words, $a_2 \neq - a_3.$
For an arbitrary $\ov$, we claim that any non-zero linear combination $z_1u_1+ z_2 u_2$ is not in the set
$\{ x \in Lie(K) : x\ov \in W^- \},$ 
thus we get codimension of the set at least $\ell =2$.
Indeed, if we solve the conditions for $z_1, z_2$, we obtain 
$(x_{12}/x_{13})^2 = -a_3/a_2,$ and the right hand side is not a square unless $a_2= -a_3.$
The case $i=1$ is immediate.
Thus the condition (2) of Lemma~\ref{05:1} is satisfied.
By Lemma \ref{05:1}, for $0<s<2$, we can see that
  \[\lim_{t\rightarrow\infty}\sup_{v\in Q_i}\int_{K}\frac
  {dm(k)}{\|a_tkv\|^s}=0,
  \]
 for any $i$. 
 
 Now let us consider $n\geq 5$.
Denote the embedding from
$\SO(x_1x_4+a_{j_2}x_2^2+a_{j_3}x_3^3)$ to $\SO(\q_p)$
induced by 
$$
    \ve_1\mapsto \ve_1, \;
    \ve_2\mapsto \ve_{j_2},\;
    \ve_3\mapsto \ve_{j_3},\;
    \ve_4\mapsto \ve_n. 
  $$
by  $\iota(j_2, j_3)$. We will denote the induced embedding on their Lie algebras by $\iota(j_2, j_3)$ again.
For a non-exceptional isotropic standard quadratic form $\q_p$ on $\QQ_p^n$, for any 
given $i$ and any vector $$\ov = \underset{J = \{ j_1 < \cdots < j_i \}}{\sum} a_J \ve_J \in V_i,$$ 
choose
$j_2$, $j_3$ such that 
$$-a_{j_2}/ a_{j_3} \; \mathrm{is \; not \; a \; square},$$
and 
$\iota(j_2, j_3) (\SO(x_1x_4+a_{j_2}x_2^2+a_{j_3}x_3^3))$ acts nontrivially on $\ov$.
  Clearly, the embedding $\iota(j_2, j_3)$ maps subgroups $ \mathbf U$ and
$ \mathbf U'$ into subgroups of the maximal compact subgroup $K$ of $\SO(\q_p)$.

For any $\ov\in W'$, one can choose such an embedding $\iota(j_2, j_3)$
so that at least one of the orbits $\iota(j_2, j_3) \mathbf U \ov$ and $\iota(j_2, j_3)\mathbf U' \ov$
are not contained in $W'$. Similarly, for any $\ov \in W^-$, one can choose such an embedding $\iota(j_2, j_3)$
so that for $(z_1, z_2) \neq (0,0)$, $(z_1\iota(j_2, j_3) u_1 + z_2 \iota(j_2, j_3) u_2) \ov \notin W^-$.
Lemma follows from the case $n=4$.
\end{proof}

In the next lemma, we will combine the results over different places $p \in S$. Recall 
that for $t=(t_p)_{ p \in S}$, we set $a_\t=(a_{t_p})_{ p \in S}$.
%%%%%%%%%%%%%%%%%%%%%%%%%%%%%%%%%%%%%%%%%%%%%%%%%%%%%%%%%%%%%%%%%%%%%%%%%%%%
%Lemma 5.7
For $p \in S$, we denote by $ \epsilon_p$ the element of $\ST$ for which $t_p=1$ and $t_q=0$ for $q \neq p$. Note that for $p \in S_f$, 
$a_{\epsilon_p} = a_1^p = \diag\left(p, 1, \ldots, 1, p^{-1}\right)$ as an element in $\SO(\q_p)$ (or as an element embedded in $\SO(\q)$).
%Similarly, $a_{\epsilon_\infty}=
%\diag\left(e^{-1}, 1, \ldots, 1, e \right)$ in place $ \infty$.  

\begin{lemma}\label{lem5:07} Let $n \ge 4$, $\q$ be a non-degenerate isotropic non-exceptional quadratic form  over $\QQ_S^n$, and $\SK$ be the maximal compact subgroup of $\SO(\q)$ defined above. If $\q$ is not exceptional, then for any $s \in (0,2)$ and any $c>0$, 
there exists a positive integer $m$ 
such that for every lattice $\Delta$ in $\QQ_S^n$, and every $\t \in R$, with
\begin{equation}\label{df:R}
R=\{ m \epsilon_p: p \in S \}.
\end{equation}

the following inequality holds:
\begin{equation}
\int_\SK \alpha_i(a_{\t}\sk \Delta)^s dm(\sk) < \frac c 2 \alpha_i(\Delta)^s + \omega^2 \max_{0<j<\min\{n-i,i\}}\left(\alpha_{i+j}(\Delta)\alpha_{i-j}(\Delta)\right)^{s/2}.\label{eq5:28}
\end{equation}

\end{lemma}
\begin{proof} First, let us assume that $\q$ is not exceptional, and let $c>0$ be given. {We will use} Proposition \ref{05:2+4} and the analogous statement in \cite{EMM}. For each $p \in S$, one can find an integer $m_p>0$ such that for any $v\in F(i)$, where $F(i)$ is defined in \eqref{F(i)} with $\|v\|_p=1$, and any $t_p>m_p$. 
\[\int_{K_p} \frac {dm(k_p)}{\|a^p_{t_p} k_pv\|_p^s} < c/2.
\]
 
Note that
\[\int_{K_p} \frac {dm(k_p)}{\|a^p_{0} k_pv\|_p^s} \le 1. \]
This implies that for $\t=m \epsilon_p$, $p \in S$, we have
\[\int_{\SK} \frac {dm(\sk)}{\|a_{\t} \sk v\|^s} 
= \prod_{p \in S} \int_{K_p} \frac {dm(k)}{\|a^p_{t_p} k_pv\|_p^s} 
< \frac{c}{2}.
\]
%The argument in the case where $\q_{ \infty}$ is of signature $(2,2)$ is similar. In this case, one cannot find a multiple of $ \epsilon_{ \infty}$ for which the inequality holds, but using Proposition \ref{05:2+4}, one can choose $m$ large enough so that 
%
%\[  \int_{K_{p}} \frac {dm(k)}{\|a^p_{m} k v\|_{p}^s} \cdot \int_{K_{ \infty}} \frac {dm(k')}{\|a^{ \infty}_{1} k' v\|_{ \infty}^s} 
%< \frac{c}{2}. \]
We deduce that for every nonzero $v \in F(i)$, and $ \t \in R$, we have
  \begin{equation}
  \int_\SK \frac {dm(k)}{\|a_{\t}{\sk}v\|^s} < \frac c 2 \; \frac 1 {\|v\|^s}.\label{eq5:29}
  \end{equation}

  For a given $S$-lattice $\Delta$ in $\QQ_S^n$ and each $i$, there exists a $\Delta$-rational subspace $L_i$ of dimension $i$ satisfying that
  \begin{equation}
  \frac 1 {d(L_i)}=\alpha_i(\Delta).
  \end{equation}

  By substituting a wedge product of $\ZZ_S$-generators of $L_i\cap \Delta$ for $v$ in \eqref{eq5:29}, we have that
  \begin{equation}
  \int_\SK \frac {dm(\sk)}{d_{a_\t\sk\Delta}(a_\t \sk L_i)^s} < \frac c 2 \; \frac 1 {d_{\Delta}(L_i)^s}.\label{eq5:31}
  \end{equation}
  
Let $\omega>1$ be such that for all $v\in F(i)$, $1 \le i \le n-1$, and $\t \in R$ we have
  \begin{equation}
  \omega^{-1}\leq \frac {\|a_\t v\|}{\|v\|}\leq \omega.\label{eq5:32}
  \end{equation}
  Let $\Psi_i$ be the set of $\Delta$-rational subspaces $L$ of dimension $i$ for which 
  $$d_{\Delta}(L)<\omega^2 d_{\Delta}(L_i).$$ 
  We will prove inequality \eqref{eq5:28} by distinguishing two cases. First, assume that $\Psi_i=\{L_i\}$. Then by \eqref{eq5:32}, we have
  \begin{align}
  d_{a_\t \sk \Delta}(a_\t \sk L)&=\|a_\t \sk v\|\geq \omega^{-1}\|v\|=\omega^{-1}d(L)
  \geq \omega d(L_i)= \omega \|v'\|\geq \|a_\t \sk v'\|=d_{a_\t \sk \Delta}(a_\t \sk L_i), \label{eq5:33}
  \end{align}
  where $v$ and $v'$ are wedge products of $\ZZ_S$-generators of $L$ and $L_i$ respectively. By inequalities \eqref{eq5:31},\eqref{eq5:33} and the definition of $\alpha_i$,
  \begin{equation}
  \int_K \alpha_i(a_\t \sk \Delta)^s dm(\sk) < \frac c 2 \alpha_i(\Delta)^s.\label{eq5:34}
  \end{equation}
  Now we assume that $\Psi_i\neq \{L_i\}$. Let $M$ be an element of $\Psi_i$ different from $L$, and let $\dim(M+L_i)=i+j$ for some $j>0$. Using Lemma \ref{lem 5:06} and the inequality \eqref{eq5:32}, we get that for any $\sk\in \SK$,
  \begin{align*}
  \alpha_i(a_\t \sk\Delta)< \omega \alpha_i(\Delta)&=\frac {\omega}{d(L_i)}
  < \frac {\omega^2}{\left(d(L_i)d(M)\right)^{1/2}}\\
  &\leq \frac {\omega^2}{\left(d(L_i\cap M)d(L_i +M)\right)^{1/2}}\leq \omega^2\left(\alpha_{i+j}(\Delta)\alpha_{i-j}(\Delta)\right)^{1/2},
  \end{align*}
  so that
  \begin{equation}
  \int_\SK \alpha_i(a_\t \sk \Delta)^s dm(\sk) \leq \omega^2 \max_{0<j\leq \min\{n-i,i\}}\left(\alpha_{i+j}(\Delta)\alpha_{i-j}(\Delta)\right)^s.\label{eq5:36}
  \end{equation}
The result {follows from \eqref{eq5:34} and \eqref{eq5:36}.}
\end{proof}
%%%%%%%%%%%%%%%%%%%%%%%%%%%%%%%%%%%%%%%%%%%%%%%%%%%%%%%%%%%%%%%%%%%%%%%%%%%%%

%%%%%%%%%%%%%%%%%%%%%%%%%%%%%%%%%%%%%%%%%%%%%%%%%%%%%%%%%%%%%%%%%%%%%%%%%%%%%
%Lemma 5:10
For $p\in S_f$, we will set
\[D_n^+=\left\{{\diag}\left(p^{\lambda_1}, p^{\lambda_2},
\ldots, p^{\lambda_n}\right)\in \GL_n(\QQ_p) : \lambda_1\leq
\lambda_2 \leq \cdots \leq \lambda_n, \lambda_i\in \ZZ \right\}.\]

Every element $g\in \GL_n(\QQ_p)$ admits a Cartan decomposition
  $$g=k_1(g) d(g) k_2(g), \quad k_1(g), k_2(g)\in \GL_n(\ZZ_p), d(g)\in D_n^+.$$
  We will write $\lambda_i(g)$, $1\leq i\leq n$ when
  $d(g)={\diag}\left(p^{\lambda_1(g)}, p^{\lambda_2(g)}, \ldots, p^{\lambda_n(g)}\right)$.

\begin{proposition}\label{05:10} Let $U$ be a neighborhood of the identity element $e$ in
$\SL_n(\ZZ_p)$ given by;
\begin{equation}
U=\left\{g\in \SL_n(\ZZ_p)
: \|g-e\|_p < 1\right\},
\end{equation}
where $\|\cdot\|_p$ is the maximum $p$-norm on the vector space of $n \times n$ 
matrices with respect to the standard basis. Then for any $k\in U$ and $d, d'\in
D^+_n$, and any $1\leq i \leq n$, we have
\begin{equation}\lambda_i(dkd')=\lambda_i(d)\lambda_i(d').
\end{equation}
\end{proposition}
\begin{proof} First note that $\|g\|_p=p^{-\lambda_1(g)}$.
 Let us denote $(i,j)$-entry of $g$ by $g_{ij}$. Let
 $d={\diag}\left(p^{\lambda_1}, \ldots, p^{\lambda_n}\right)$
 and $d'={\diag}(p^{\lambda'_1}, \ldots,
 p^{\lambda'_n})$ be in $D^+_n$. Then
 $dkd'=(p^{\lambda_i+\lambda'_j}k_{ij})$. Since $|k_{ii}|_p=1$ for $1 \le i \le n$ and
 $|k_{ij}|_p< 1$ for any $i\neq j$, we obtain
 \begin{equation}
p^{\lambda_1(dkd')}=\frac 1 {\|dkd'\|_p}=\frac 1 {|p^{\lambda_1+\lambda'_1}k_{11}|_p}=
 p^{\lambda_1+\lambda'_1}=p^{\lambda_1(d)}p^{\lambda_1(d')}.
 \end{equation}

Consider the representation $\rho_i$ of $\GL_n(\QQ_p)$ on the
$i$-th exterior product $\bigwedge^i(\QQ_p^n)$ in the usual way
and denote the maximum $p$-norm on $\bigwedge^i \QQ_p^n$ with respect to the standard basis by
$\|\cdot\|_p$. When expressed in the standard basis of $\bigwedge^i(\QQ_p^n)$, the entries of $\rho_i(g)$ are given by pairs $(J,K)$, where $J$ and $K$ are $i$-element subsets of the set $\{1,2, \dots, n \}$. More precisely, $\rho_i(g)=(\det g_{JK})$, where $g_{JK}$ is the $i\times i$ minor of $g$ formed by the rows in $J$ and columns in $K$. Note that the $p$-norm on $\bigwedge^i \QQ_p^n$ is also invariant under the action of $\GL_n(\ZZ_p)$. For $k\in U$, since $|k_{ii}|_p=1$ and $|k_{ij}|<1$ for $i \neq j$, we readily see 
  that $ |\det(k_{JK})|=1$ if $J=K$
  and $ |\det(k_{JK})|<1$, otherwise. This implies that 
  \[\|\rho_i(g)\|_p= \|\rho_i(dkd')\|_p= p^{-(\lambda_1(g)+\cdots+\lambda_i(g))}.
  \]
  Hence we obtain
  \begin{equation}p^{\lambda_1(dkd')+\cdots+\lambda_i(dkd')}=\frac 1
{\|\rho_i(dkd')\|_p}=p^{\sum_{j=1}^i\lambda_j+\sum_{j=1}^i\lambda'_j}=p^{\lambda_1(d)+\cdots\lambda_i(d)}p^{\lambda_1(d')+\cdots+\lambda_i(d')}.
  \end{equation}
\end{proof}
%%%%%%%%%%%%%%%%%%%%%%%%%%%%%%%%%%%%%%%%%%%%%%%%%%%%%%%%%%%%%%%%%%%%%%%%%%%%%%

We remark that the above proposition also holds for
\[
D_-^n=\left\{{\diag}\left(p^{\lambda_1}, p^{\lambda_2},
\ldots, p^{\lambda_n}\right) \in \GL_n(\QQ_p) : \lambda_1 \geq
\lambda_2 \geq \cdots \geq \lambda_n, \lambda_i\in\ZZ \right\},
\] instead of $D_+^n$.

%%%%%%%%%%%%%%%%%%%%%%%%%%%%%%%%%%%%%%%%%%%%%%%%%%%%%%%%%%%%%%%%%%%%%%%%%%%%%
\begin{corollary}\label{lem5:11} Let $ H$ be a simply connected simple algebraic
group in $\GL_n(\QQ_p)$ and $K$ the maximal compact subgroup
$\SL_n(\ZZ_p)\cap  H$ of $ H$. Let $a_t={\diag}(p^{-t}, 1,
\ldots, 1, p^t)$ for $t\in \ZZ$. Then there is a neighborhood $U$ of $e$ in $K$
such that for $t, s \ge 0$, we have 
\begin{equation*}
a_t U a_s \subset K a_t a_s K.
\end{equation*}
\end{corollary}
\begin{proof}
The Cartan decomposition of $H$(see Theorem
3.14 in \cite{PR}) provides {that for each $g\in H$, there are} elements $k_1(g)$ and
$k_2(g)$ in $K$ such that
\[a_tga_s=k_1(g){\diag}\left(p^{\lambda_1(a_tga_s)}, \ldots,
p^{\lambda_n(a_tga_s)}\right)k_2(g).
\]
Take a neighborhood $U$ of $K$ satisfying the condition of
Proposition \ref{05:10}. Then
\begin{eqnarray*}{\diag}\left(p^{\lambda_1(a_tga_s)}, \ldots,
p^{\lambda_n(a_tga_s)}\right)={\diag}\left(p^{\lambda_1(a_ta_s)},
\ldots, p^{\lambda_n(a_ta_s)}\right)
= a_ta_s,
\end{eqnarray*}
 so we get the result.
\end{proof}
%%%%%%%%%%%%%%%%%%%%%%%%%%%%%%%%%%%%%%%%%%%%%%%%%%%%%%%%%%%%%%%%%%%%%%%%%%%%%%

%%%%%%%%%%%%%%%%%%%%%%%%%%%%%%%%%%%%%%%%%%%%%%%%%%%%%%%%%%%%%%%%%%%%%%%%%%%%%%
%proposition 5.12

\begin{proposition}\label{supharmonic} Let $\SH=\SO(\q)$, $\SK=\SL_n(\ZZ_p)\cap \SH$ 
%$\Theta$ as in Lemma \ref{lem5:07} 
and let $\mathcal F$ be a family of strictly positive functions on $\SH$ having the following properties:
\begin{enumerate}[(a)]
\item For every $ \lambda>1$, there exists a neighborhood $V(\lambda), \lambda>1$ of the identity in $\SH$ such that for all $f\in \mathcal F$,
$\lambda^{-1}f(\sh) < f(\sg \sh) < \lambda f(\sh)$ for any $\sh\in \SH$ and $\sg\in V(\lambda)$. 
\item There exists a constant $C>0$ such that for all $p \in S_f$, and all $ \sh \in \SH$, we have    
 \[ f( a_{\epsilon_p} \sh) < C f(\sh),\]
 where $a_{\epsilon_p}$ denotes the element $\diag(p, 1, \ldots, 1, p^{-1})$. 
\item The functions $f\in \mathcal F$ are left $\SK$-invariant, that is, $f(\sk \sh)=f(\sh)$, for $\sk \in \SK$ and $\sh \in \SH$.
\item $\sup_{f\in \mathcal F}f(1) < \infty$.
\end{enumerate}
%\seon{I changed $a_{\epsilon_p}$ to $a_1^p$ and t to $\t$} 
Then there exists $0<c=c(\mathcal F)<1$ such that for any $\t_0>0$ and $b>0$ there exists $B=B(\t, b)<\infty$ with the following property: If $f\in \mathcal F$ and
\begin{equation}
\int_\SK f(a_{\t_0} \sk \sh)dm(\sk) < c f(\sh)+ b \label{eq5:52}
\end{equation}
for any $\sh \in \SH$, then
\[\int_\SK f(a_{\t}\sk)dm(k) < B
\]
for all $\t>0$. 
\end{proposition}
\begin{proof} First note that we can replace each $f \in \mathcal F$ with its right $K$-average
\[\tilde{f}(\sh)=\int_\SK f(\sh\sk)dm(\sk)
\]
and $\tilde{\mathcal F}=\{\tilde{f} : f \in \mathcal F\}$ still satisfies properties (a), (b) and (c).
Hence, without loss of generality, we may and will assume that for all $f\in \mathcal F$ and $\sh \in \SH$, we have
\begin{equation}f(\SK\sh\SK)=f(\sh). \label{eq5:53}
\end{equation}
We need to prove
\begin{equation}
\sup_{\t >0} f(a_{\t}) < B. \label{eq5:54}
\end{equation}

By Corollary \ref{lem5:11} and Lemma 5.11 in \cite{EMM}, we can take a neighborhood $\SU$ of the identity in $\SH$ such that $$a_{\t_0}\SU a_{\t}\in \SK \SV a_{\t_0} a_{\t} \SK$$ 
for any $\t_0 , \t \geq 0$. From (a) and \eqref{eq5:53}, it follows that 
\begin{equation}
\int_\SK f(a_{\t_0} \sk a_{\t})dm(\sk)\geq \int_{\SU\cap \SK} f(a_{\t_0} \sk a_{\t}) dm(\sk)> \frac{1}{2} m_\SK (\SU\cap \SK) f(a_{\t_0} a_{\t}). \label{eq5:56}
\end{equation}
Suppose that for all $\t_0 \in R$ and $b>0$, and all $\sh\in \SH$ we have
\begin{equation}
\int_\SK f(a_{\t_0} \sk \sh)dm(\sk) < \frac{1}{4} m_{\SK} (\SU\cap \SK)f( \sh)+ b.\label{eq5:57}
\end{equation}
Let $R$ be as in \eqref{df:R}, and $\Theta$ denote the semigroup generated by $R$.
Then by taking $\sh=a_{\t}$ in \eqref{eq5:57} we see that for any $\t>0$, and $\t_0 \in\Theta$, we have
\begin{equation}
f(a_{\t_0} a_{\t})< \frac 1 2 f(a_{\t})+ b' \le  \max(f(a_\t), b'), \label{eq5:58}
\end{equation}
where $b'=\max(2b/m_\SK(U \cap \SK),f(e))$. A simple induction shows that for 
any $\t \in \Theta$, we have
\begin{equation}
f(a_\t) \le b'.
\label{eq5:59}
\end{equation}

Note that every $ \t >0$ can be decomposed as
$ \t= \t_1+ \t_2$, where $\t_1 \in \Theta$ and 
$\t_2$ is positive and bounded in every component. Since $\t_2$ can be written as 
{the product of} a bounded number of elements of $V(1/2)$ and $\epsilon_p$, with $p \in S_f$, the claim 
follows. 

\end{proof}
%%%%%%%%%%%%%%%%%%%%%%%%%%%%%%%%%%%%%%%%%%%%%%%%%%%%%%%%%%%%%%%%%%%%%%%%%%%%%

%%%%%%%%%%%%%%%%%%%%%%%%%%%%%%%%%%%%%%%%%%%%%%%%%%%%%%%%%%%%%%%%%%%%%%%%%%%%%

\begin{proof}[Proof of Theorem \ref{bound}] Define functions $f_0, f_1, \ldots, f_n$ on $\SH$ by
\[f_i(\sh)=\alpha_i(\sh \Delta), \quad \sh\in \SH.
\]
Since $\alpha \le \sum_{i=0}^{n} \alpha_i$, it suffices to show that for all $0\leq i\leq n$,
\begin{equation}
\sup_{\t>0} \int_\SK f_i^s(a_\t \sk)dm(\sk)<\infty. \label{eq5:64}
\end{equation}
Let us check that $f_0, f_1, \ldots, f_n$ have the properties (a), (b), (c), and (d)
of Proposition \ref{supharmonic}. 
First note that for $v \in F(i)$, $0\leq i\leq n$, and $\sh \in \SH$, we 
have the trivial bound $\| \sh v \| \le  \| \wedge^i \sh \| \| v \|$. This proves
(a) and (b).
Since the action of $\SK$ preserves the function $\|\cdot\|$ on $\bigwedge^i(\QQ_S^n)$, each $f_i$ is left $\SK$ invariant. 
It is clear that $f_i(1)$ are uniformly bounded as $ \Delta$ runs over a compact set of lattices. 
From Lemma \ref{lem5:07} with $h\Delta$, for any $i$, $0<i<n$ and $h\in \SH$, we see that
\begin{equation}
\int_\SK f_i^s(a_\t \sk \sh)dm(\sk) < \frac c 2 f_i^2 + \omega^2 \max_{0<j<\min\{n-i,i\}} \left(f_{i+j}f_{i-j}\right)^{s/2}. \label{eq5:65}
\end{equation}
 
Define
\[f_{\epsilon, s}=\sum_{0\leq i\leq n}\epsilon^{i(n-i)}f_i^s.
\]
Since $\epsilon^{i(n-i)}f_i^s < f_{\epsilon, s}$, $f_0=1$ and $f_n=d(\lambda)^{-1}$, by putting $\epsilon=c/(2n\omega^2)$, we have the inequality \eqref{eq5:52} of Proposition \ref{supharmonic}:
\begin{align}
\int_\SK f_{\epsilon, s}(a_\t \sk \sh) dm(k) &< 1 + d(\Delta)^{-s}+ \frac c 2 f_{\epsilon, s}+ n\epsilon \omega^2 f_{\epsilon, s}\nonumber\\
&=c f_{\epsilon, s}+ 1+ d(\Delta)^{-1}.\label{eq5:67}
\end{align}
 Let $\mathcal C$ be an arbitrary compact set of unimodular lattices $\Delta$ and let $\mathcal F$ be the family of $f_{\epsilon, s}$ as $\Delta$ runs over $\mathcal C$. Then $\mathcal F$ satisfies the conditions of Proposition \ref{supharmonic}.    
From \eqref{eq5:67} it follows that there the constants $b$ and $c$ can be chosen 
uniformly for the family $f_{\epsilon, s}\in \mathcal F$  constant $c$ and $b$. Since $\alpha_i(h\Delta)^s\leq \epsilon^{-i(n-i)}f_{\epsilon, s}(h)$, by Proposition \ref{supharmonic}, we conclude that there exists a constant $B>0$ such that for each $i$, all $\t\succ0$ and all $\Delta \in \mathcal C$,
  \[\int_\SK \alpha_i(a_\t \sk\Delta)^s dm(\sk) < B.
 \]
\end{proof}

%%%%%%%%%%%%%%%%%%%%%%%%%%%%%%%%%%%%%%%%%%%%%%%%%%%%%%%%%%%%%%%%%%%

We will now use this bound to prove the main result. Define the set $A(r)$ by
\begin{equation}
A(r)=\{x\in  \SG/\Gamma : \alpha(x)\leq r \}.\label{A(r)}
\end{equation}
%\seon{I put the paper of Kleinbock-Tomanov in reference where I could find the proof of $S$-adic Mahler's compactness criterion.}
Using Mahler's compactness criterion~\cite{KT}, we see that $A(r)$ is compact for any $r>0$.
%we need to prove this. It's in the note... I think...
%%%%%%%%%%%%%%%%%%%%%%%%%%%%%%%%%%%%%%%%%%%%%%%%%%%%%%%%%%%%%%%%%%%
%\begin{theorem}\label{thm 3:05} Let $\q$ be a quadratic form such that the signature of $\q_{ \infty}$ is $(r,s)$ with
%$r \geq 3$. Then Theorem \ref{thm 4:04} holds for a continuous function $\phi$ on $ G/\Gamma$ such that for some $s$, $0\leq s\leq 2$, there exists a constant $C>0$ so that
%\[|\phi(\Delta)|< C\alpha(\Delta)^s
%\]
%for all $\Delta$ in $ G/\Gamma$.
%\end{theorem}
\begin{proof}[{Proof of Theorem \ref{ergodic-a}}]
We may assume that $\phi$ is nonnegative. For each $r\in \RR_{>0}$, we choose a continuous function $g_r$ on $ \SG/\Gamma$ such that $0\leq g_r(x)\leq 1$, $g_r(x)=1$ if $x\in A(r)$ and $g_r(x)=0$ outside $A(r+1)$. Then
\[\phi(a_\t \sk x)= (g_r\phi)(a_\t \sk x) + ((1-g_r)\phi)(a_\t \sk x).
\]

 Following the proof of Theorem 3.5 in \cite{EMM}, let $\beta=2-s$. Since $((1-g_r)\phi)(y)=0$ if $y\in A(r)$,
\begin{align*}
((1-g_r)\phi)(y)&\leq  C(1-g_r)(y)\alpha(y)^{2-\beta}\leq C\alpha(y)^{2-\beta/2}(1-g_r)(y)\alpha(y)^{-\beta/2}\\
&\leq Cr^{-\beta/2}\alpha(y)^{2-\beta/2}.
\end{align*}
By {Theorem \ref{bound}} and the fact that $\|\nu\|_{\infty}<\infty$ , there exists $C'>0$ such that
\begin{equation}\label{eq5:81}
\begin{split}
\int_{ \SK} ((1-g_r)\phi)(a_\t \sk x)\nu(\sk)dm(\sk)&\leq C r^{-\beta/2}\int_{ \SK} \alpha(a_\t \sk x)^{2-\beta/2}\nu(\sk)dm(\sk)\\
&\leq C' r^{-\beta/2}.
\end{split}
\end{equation}

On the other hand, since $g_r\phi$ is compactly supported, we see that for sufficiently large $\t\in \RR_{>0}\times \NN^s$ in the cone generated by $R$
\begin{equation}\label{eq5:82}
\left|\int_{ \SK}(g_r\phi)(a_\t \sk x)\nu(\sk) dm(\sk) - \int_{ \SG/\Gamma}(g_r\phi)(y)dg(y)\int_\SK\nu(\sk)dm(\sk)\right|<\epsilon/2,
\end{equation}
where $d\sg$ is the normalized Haar measure on $ \SG/\Gamma$.
Since $g_r\phi\rightarrow \phi$ as $r\rightarrow \infty$, \eqref{eq5:81} and \eqref{eq5:82} show the theorem.
\end{proof}

%%%%%%%%%%%%%%%%%%%%%%%%%%%%%%%%%%%%%%%%%%%%%%%%%%%%%%%%%%%%%%%%%%%%%%%%%%%%%

\section{Counterexamples for ternary forms}
In this section, we will prove Theorem \ref{counter-example} which partially shows that Theorem \ref{main:asymptotics} is optimal in the sense that for some exceptional forms the asymptotics does not hold.

We will first introduce the family of quadratic forms which will be useful for the construction. 
Fix $  \alpha_p \in \QQ_p$, and define
 \[ \q^{ \alpha_p}(\ox)= x_1^2+ x_2^2 - \alpha_p^2 x_3^2. \] 
Similarly, for an $S$-vector $ \alpha=( \alpha_p)_{p \in S} \in \QQ_S$, we set $ \q^{ \alpha}=( \q^{ \alpha_p})_{p \in S}$.  
In the case $ \alpha \in \QQ$, with a slight abuse of notation, we will also use $\q^{\alpha}$ as a shorthand for the quadratic form 
on $\QQ_S^n$ associated to the constant $S$-vector with $\alpha_p= \alpha$ for all $p \in S$. 
We will first prove an $S$-adic version of a well-known fact from number theory adapted to fulfill our purpose.
 
\begin{lemma}\label{a lot}
Given $ \alpha \in \QQ$, there exists a constant $ c_{ \alpha}>0$ such that for a sufficiently large $S$-time 
$\T=(T_p)_{p \in S}$, there exists at least $ c_{ \alpha} \prT \log \prT$ vectors
$ \ox=(x_1,x_2,x_3) \in \ZZ_S^3$ which satisfy 

\begin{enumerate}
\item $ \q^{ \alpha}(\ox)=0 $.
\item $\| x_i \|_p= T_p$ for every $1 \le i \le 3$ and $p \in S_f$. 
\item $\| \ox  \|_{ \infty} \le T_{ \infty}$.
\end{enumerate}
\end{lemma}

Before we proceed to the proof, we emphasize that condition (2) is a stronger form of $\| x\|=T_p$ which
is needed for the following arguments.

\begin{proof}
It will be slightly more convenient to first work with the quadratic form $ \q'(\ox)= x_1x_2- x_3^2$, which is 
$\QQ$-equivalent to $ \q^{\alpha}$ for every $ \alpha$. For $p \in S_f$, write $T_p= p^{n_p}$, and set 
$a= \prod_{p \in S_f} p^{-n_p}= T_{ \infty} \prT^{-1}$. We will consider triples
$$x_1=aku^2, \quad x_2=akv^2, \quad x_3=akuv,$$ 
where $u,k, v \in \ZZ$ with $\gcd(u,v)=1$ and $\gcd(u,p)=\gcd(v,p)=\gcd(k,p)=1$ for all $p \in S_f$. 
It is clear that $\ox$ satisfies conditions (1) and (2). 
Moreover, the condition (3) will also be satisfied if $k$ is chosen such that 
$$|u|, |v| \le \sqrt{T_{ \infty} /3ak}.$$ It is well-known that the density of the pairs $(u,v)$ with $\gcd(u,v)=1$ is $c=6/\pi^2>0$. We can thus produce at least as many as 
\[ \sum_{\substack{ 1 \le k \le T_{ \infty} /3a \\ \gcd(k,p)=1, p \in S} }   c \frac{T_{ \infty }}{3ak}
= c \ \prT \sum_{\substack{ 1 \le k \le T_{ \infty} /3a \\ \gcd(k,p)=1, \ p \in S} } \frac{1}{k}   \gg \prT \log \prT \]
solutions. Note that the forms $\q^{ \alpha}$ can be obtained from $\q'$ applying a rational linear transformation $A$,
which changes the real and $p$-adic norms by a bounded factor. This completes the proof. 
\end{proof}

We will need the following lemma, which is an $S$-arithmetic version of Lemma 3.15 in \cite{EMM}. We will assume that $ \Omega = \prod_{p \in S}  \U_p^n$ is the product of unit spheres, and drop it from the notation for the counting function.

\begin{lemma}\label{dense}
Let $\I= \prod_{p \in S} I_p$ be an $S$-interval. Given $ \epsilon>0$ and an $S$-time $\T_0$, the set of vectors $\beta= ( \beta_p)_{p \in S} \in \QQ_S$ for which there exists $\T \succ \T_0$ such that 
\[   \ct{\I,\q^{\beta}}{\T} \ge \prT ( \log \prT)^{1- \epsilon} \]
is dense in $\QQ_S$.
\end{lemma}

\begin{proof} For the notational simplicity, we will set $I_{ \infty}=[1/4,1/2]$ and $I_p=\ZZ_p$ for $p \in S_f$. It will be clear
that the proof works for any other choice of intervals. For $ \alpha \in \QQ$ and $\T$ large, define
\[ L(\alpha, \T)= \{ \ox \in \ZZ_S^3: \q^{\alpha}(\ox)=0, \, \| \ox \|_{ \infty} \in [T_{ \infty}/2,T_{ \infty}], \quad \| \ox \|_p = T_p, \, \forall p \in S_f
\}.  \]
By Lemma \ref{a lot}, for large enough $\T$, we have 
\[ \card L(\alpha, \T)  \ge \frac{ c_{ \alpha}}{4} \prT \log \prT \ge \prT ( \log \prT)^{1 - \epsilon}. \]

Note that if $\ox=(x_1, x_2, x_3) \in L( \alpha, \T)$ then
\[  \frac{T_{  \infty} ^2}{4( 1+ \alpha^2)} \le x_3^2 \le \frac{T_{ \infty} ^2}{1+  \alpha^2}. \] 
By condition (2) in Lemma \ref{a lot}, we also have $|x_3 |_p=T_p=p^{m_p}$ for every $p \in S$. 
Set 
\[ \beta_{ \infty}^2= \alpha^2 - (1+ \alpha^2)T_{ \infty}^{-2}, \quad  \beta_p^2= \alpha^2+ u_p p^{m_p},\]
where $|u_p|_p=1$.  Note that for sufficiently large $T_p$, $ \alpha^2+ u_p p^{m_p}$ is a square in 
$\QQ_p$ and $\beta$ is well-defined. A simple computation shows that for any $\ox \in L( \alpha, \T)$ we have
\begin{equation}
\begin{split}
\q^{ \beta_{ \infty}}(\ox) & = \q^{\alpha}(\ox)+ (1+ \alpha^2)T_{ \infty}^{-2} x_3^2 \in [1/4,1],  \\
\q^{ \beta_p}(\ox) & = \q^{ \alpha}(\ox)+ u^2 x_3^2 \in \ZZ_p.  
\end{split}
\end{equation}

This implies that the 
\[ \ct{\I, \q^{\beta}}{\T} \ge \card L( \alpha, \T) \gg \prT \log \prT^{1- \epsilon}.  \]
It is easy to see that the set of $ \beta$ obtained in this way is dense in $\QQ_S$. 

\end{proof}

We can now prove Theorem \ref{counter-example}. 
Assume that 
for a given $\T_0$, we denote by $W(\T_0)$ the set of $\gamma \in \QQ_S$ such that there exists 
$ \beta \in \QQ_S^n$ satisfying $ | \beta- \gamma| < \prT^{-3}$ and 
\[   \ct{\I, \q^{\beta}}{\T} \ge \prT \log \prT^{1- \epsilon}. \]
From the definition and Lemma \ref{dense}, it is clear that $W(\T_0) \subseteq \QQ_S$ is open and dense. 
Let $\T_j$ be a sequence of $S$-times going to infinity. Then $W= \cap_{j=1}^{ \infty} W(\T_j)$ is a set of 
second category, and hence contains an irrational form. For any $\gamma \in W$, there exists an infinite 
sequence $\beta^i$ and $\T_i$ satisfying $|\gamma_p- \beta^i_p|_p< \prTi^{-3}$ for all $p \in S$  and 
\[   \ct{\I, \q^{\beta^i}}{\T_i} \ge \prTi \log \prTi^{1- \epsilon}. \]
Set $R_i= \prTi$.  If $\| \ox \|_{ \infty}< R_i$, it is easy to see that
\[ \| \q_{ p}^{ \beta^i}(\ox)- \q_{ p} ^{\gamma}(\ox) \|_{ \infty} =O(R_i^{-1})< 1/8 \]
for $i \gg 1$ and $p \in S$. 
This implies that $\q^{\gamma}_{ \infty}(\ox) \in [1/8, 2]$ and $\q^{\gamma}_p(\ox) \in \ZZ_p$ for $p \in S_f$. 
The claim follows from here.

The argument when all $\q_p$ are equivalent to $x_1x_2-x_3x_4$ is similar. An argument along the same lines as the one given above, using the parametrization of the solutions to $x_1x_2-x_3x_4$ given by
\[ x_1= auv, \quad x_2= azw, \quad x_3=auz, \quad x_4= avw \]
establishes the result in this case. 

\appendix
\section{An extension of the Witt theorem} In this appendix, we prove Proposition~\ref{transitivity of K} as stated in Section 2.
\begin{proposition} For given $c_1 \in \QQ_p$ and $c_2 \in p^\ZZ$,
$ K_p$ acts transitively on $$\{\ov_p\in \QQ_p^n : \q(\ov_p)=c_1 \ \text{and} \
\| \ov_p\|_p=c_2\}.$$ 
%\seon{$C_1, C_2$ changed to $c_1, c_2$ to be compatible with section 3.2.}
\end{proposition}
\begin{proof} For simplicity, let us denote $\q_p$, $ K_p$ and $\ov_p$ by $\q$, $ K$ and $\ov$, etc. We need to prove that for given two vectors $\ov_1$ and
$\ov_2$ in $\mathbb Q_p^n$ with the same $p$-adic norm and the same value of $\q$, there is an element $k$ in $ K$ such that
$k.\ov_1=\ov_2$. For convenience, we may assume that $\ov_1=\ve_1$ and let
$\ov_2=\ov\in \mathbb Z_p^n$. Since the same argument will be applied to both isotropic or nonisotropic vectors, we will think of a quadratic form $\q(x_1, \ldots, x_n)$ as one of
\[    \begin{array}{c}
      u_1x_1^2+\cdots+u_ix_i^2+p(u_{i+1}x_{i+1}^2+\cdots+u_nx_n^2) \quad \text{or} \\
      x_1x_2+u_3x_3^2+\cdots+u_ix_i^2+p(u_{i+1}x_{i+1}^2+\cdots+u_nx_n^2).
    \end{array}
\]
depending on the case we want to treat. Here $u_i$'s are units in $\mathbb Z_p$.

Then we can write the corresponding symmetric matrix $B=B_{\q}$ as
$\left(
                                                                  \begin{array}{cc}
                                                                    B' &  \\
                                                                     & pB'' \\
                                                                  \end{array}
                                                                \right)$,
 where $B'$ and $B''$ are nondegenerate mod $p$.  The proposition demands to find a matrix $k\in  K$ satisfying that
\begin{enumerate}[(a)]
    \item $k.\ve_1=\ov$ and
    \item $^tkB k=B$.
\end{enumerate}

  Since $\mathbb Z_p$ is the inverse limit of $\mathbb Z/p^j\mathbb
  Z$, $j\rightarrow \infty$, we will construct a chain $k^j\in\GL_n(\mathbb Z/p^{j+1}\mathbb
  Z)$ such that
  \begin{enumerate}[(a')]
  \item $k^j.\ve_1=\ov$ mod $p^{j+1}$,
  \item $^tk^jBk^j=B$ mod $p^{j+1}$ and
  \item $k^j=k^{j+1}$ mod $p^{j+1}$.
  \end{enumerate}

 Then the inverse limit of $(k^j)_{j=0}^\infty$ will be an element satisfying the conditions (a) and (b).
  Let us denote $k^j=k_0+pk_1+p^2k_2+\cdots+p^jk_j$.

\emph{Step 1.} $j=0$. Let $k^0=k_0=\left(
                          \begin{array}{cc}
                            X_0 & Y_0 \\
                            Z_0 & W_0 \\
                          \end{array}
                        \right)$
depending on the size of $B'$ and $B''$.
%where $X_0$ is a $i$ by $i$ matrix and $W_0$ is $n-i$ by $n-i$
%                        matrix.
 By the condition (a'), the first column $\ov_0$ of $k_0$ is given by $\ov$ mod $p$. We want to find a solution of the
                        following equation;
  \begin{eqnarray*}
  \left(
    \begin{array}{cc}
      B' &  \\
       & 0 \\
    \end{array}
  \right)&=&\left(
    \begin{array}{cc}
      ^tX_0 & ^tZ_0 \\
      ^tY_0 & ^tW_0 \\
    \end{array}
  \right)\left(
    \begin{array}{cc}
      B' &  \\
       & 0 \\
    \end{array}
  \right)\left(
    \begin{array}{cc}
      X_0 & Y_0 \\
      Z_0 & W_0 \\
    \end{array}
  \right) \ \ \text{mod}\ p\\
  &=&\left(
    \begin{array}{cc}
      ^tX_0B'X_0 & ^tX_0B'Y_0 \\
      ^tY_0B'X_0 & ^tY_0B'Y_0 \\
    \end{array}
  \right) \ \ \text{mod}\ p.
  \end{eqnarray*}

  By the assumption of our quadratic
  form, $Q({pr}_i(\ve_1))=Q({pr}_i(\ov_0))$, where ${pr}_i:(x_1,\ldots,x_n)\in(\mathbb Z/p\mathbb Z)^n\rightarrow(x_1,\ldots,x_i)\in(\mathbb
  Z/p\mathbb Z)^i$. Applying the Witt theorem for finite fields (\cite{Artin}) to $\left((\mathbb Z/p\mathbb
  Z)^i,B'\right)$, we can get an isometry $X_0$ satisfying
  that the first column is ${pr}_i(\ov_0)$ and $^tX_0B'X_0=B'$ mod
  $p$. Since $^tX_0B'$ is invertible, we should take $Y_0$ as $0$. Note that in this step, we can not determine $Z_0$ and $W_0$.\\

\emph{Step 2.} $j=1$. The matrix $k^1=\left(
    \begin{array}{cc}
      X_0+pX_1 & Y_0+pY_1 \\
      Z_0+pZ_1 & W_0+pW_1 \\
    \end{array}
  \right)$ has the first column $\ov_0+p\ov_1$ and should satisfy the following equations.
\begin{eqnarray*}
^tX_0B'X_0+p(^tX_0B'X_1+^tX_1B'X_0+^tZ_0B''Z_0)&=&B' \ \text{mod}\ p^2,\\
^tX_0B'Y_0+p(^tX_1B'Y_0+^tX_0B'Y_1+^tZ_0B''W_0)&=&0\ \text{mod}\ p^2,\\
^tY_0B'Y_0+p(^tY_0B'Y_1+^tY_1B'Y_0+^tW_0B''W_0)&=&pB''\ \text{mod}\
p^2.
\end{eqnarray*}

  Since $Y_0=0$ and $^tX_0B'X_0=B'+pC_1^{11}$ mod $p^2$ for some symmetric matrix
  $C_1^{11}$, we can reduce the above equations as
\begin{eqnarray*}
C_1^{11}+^tX_0B'X_1+^tX_1B'X_0+^tZ_0B''Z_0&=&0 \ \text{mod}\ p,\\
^tX_0B'Y_1+^tZ_0B''W_0&=&0\ \text{mod}\ p,\\
^tW_0B''W_0&=&B''\ \text{mod}\ p.
\end{eqnarray*}

  Take any $Z_0$ such that the first column is
  $(v_{i+1},\ldots,v_n)$ and any $W_0$ satisfying $^tW_0B''W_0=B''\ \text{mod}\ p$ using the Witt
  theorem. Then it suffices to show that there is a matrix $X_1$ with the given first
  column satisfying the equation $C_1+^tX_0B'X_1+^tX_1B'X_0+^tZ_0B''Z_0=0 \ \text{mod}\
  p$.\\

\emph{Step 3.} We claim that for a given invertible symmetric
matrix $A$ and any symmetric matrix $C$, there is a solution $X$ of
the
equation $^tXA+AX=C$.

  By considering the space of $n$ by $n$ matrices as a $n^2$-
  dimensional vector space, we can rewrite the above equation by
\begin{equation}\left(
  \begin{array}{cccc}
    A_{11} & A_{12} & \cdots & A_{1n} \\
    A_{21} & A_{22} & \cdots & A_{2n} \\
    \vdots & \vdots & \ddots & \vdots \\
    A_{n1} & A_{n2} & \cdots & A_{nn} \\
  \end{array}
\right) \left(
         \begin{array}{c}
           {[X]^1} \\
           {[X]^2} \\
           \vdots \\
           {[X]^n} \\
         \end{array}
       \right)=\left(
                 \begin{array}{c}
                   {[C]^1} \\
                   {[C]^2} \\
                   \vdots \\
                   {[C]^n} \\
                 \end{array}
               \right),\label{eqtr1}
\end{equation}
where a block matrix $A_{ij}$ is defined by
\[A_{ij}=\left\{
         \begin{array}{ll}
           \left(
             \begin{array}{ccccc}
               a_{11} & a_{21} & \cdots & \cdots & a_{n1} \\
               \vdots & \vdots & \vdots & \vdots & \vdots \\
               2a_{1i} & 2a_{2i} & \cdots & \cdots & 2a_{ni} \\
               \vdots & \vdots & \vdots & \vdots & \vdots \\
               a_{1n} & a_{2n} & \cdots& \cdots & a_{nn} \\
             \end{array}
           \right) \ \text{when $A=(a_{st})$}
           , & \hbox{$i=j$;} \\ \\
           \text{\ \ All entries are zero except $j$-th row is}\\

          \hspace{5.5cm} \text{$(a_{1i}, a_{2i}, \ldots, a_{ni})$}
           , & \hbox{$i\neq j$.} \\
         \end{array}
       \right.\]

  Since $^tXA+AX$ and $C$ are both symmetric, after removing rows
  repeated(for example, one of second row and (n+1)th row corresponding to $c_{12}$ and
  $c_{21}$) we get a linear equation from $(\mathbb Z/p\mathbb Z)^{n^2}$ to $(\mathbb Z/p\mathbb Z)^{n(n+1)/2}$. Furthermore, from the fact that $A$ is
  invertible, the rank of this reduced linear equation is exactly
  $n(n+1)/2$ which tells us that there is a solution $X$.\\

  Now let us show that $C_1+^tX_0B'X_1+^tX_1B'X_0+^tZ_0B''Z_0=0 \ \text{mod}\
  p$ has a solution when $C_1$, $B'$, $X_0$, $B''$, $Z_0$ and the first column of $X_1$ are
  given. That is, the first $n^2$ by $n$ submatrix $^t(A_{11}, \ldots,
  A_{n1})$ of the equation (1) is erased together with the $n$ variables
  $[X]^1$. Before that, let us assume that the removed row among the repeated rows in the above
  argument is always the former one. That is, in the example, we
  will remove the second row and leave $(n+1)$th row. Hence the
  entries of the $n(n+1)/2$ by $n$ submatrix of the reduced matrix
  are zero except the first row, and consequently the rank of the linear equation $C_1+^tX_0B'X_1+^tX_1B'X_0+^tZ_0B''Z_0=0 \ \text{mod}\
  p$ is $n(n+1)/2-1$. On the other hand $c_{11}$ in the equation \eqref{eqtr1} can be also removed since it is determined by $A_{11}$ and
  $[X]^1$. Therefore if we check that the (1,1)-entry of $C_1+^tX_0B'X_1+^tX_1B'X_0+^tZ_0B''Z_0=0 \ \text{mod}\
  p$ holds, we can find a required matrix $X_1$. However this follows
  from the fact that $\q(\ve_1)=\q(\ov)=\q(\ov_0+p\ov_1)$ mod $p$.\\
  
\emph{Step 4.} In general, suppose that there exists a solution
$k=k_0+pk_1+p^2k_2+\cdots+p^nk_n+\cdots$ satisfying (a) and (b). Then from the condition
$^tkBk=B$,
\begin{align*}
\left(
  \begin{array}{cc}
    B' &  \\
     & pB'' \\
  \end{array}
\right)
&= \sum_{j=0}^{\infty}p^j\left(\sum_{i=0}^{j}\left(
  \begin{array}{cc}
    ^tX_i & ^tZ_i \\
    ^tY_i & ^tW_i \\
  \end{array}
\right)\left(
  \begin{array}{cc}
    B' &  \\
     & pB'' \\
  \end{array}
\right)\left(
  \begin{array}{cc}
    X_{j-i} & Y_{j-i} \\
    Z_{j-i} & W_{j-i} \\
  \end{array}
\right)\right)\\
&=\sum_{j=0}^{\infty}p^j\left(\sum_{i=0}^{j}\left(
  \begin{array}{cc}
    ^tX_iB'X_{j-i} & ^tX_iB'Y_{j-i} \\
    ^tY_iB'X_{j-i} & ^tY_iB'Y_{j-i} \\
  \end{array}
\right)\right)+\\
&\hspace{4cm} \sum_{j=0}^{\infty}p^{j+1}\left(\sum_{i=0}^j\left(
  \begin{array}{cc}
    ^tZ_iB''Z_{j-i} & ^tZ_iB''W_{j-i} \\
    ^tW_iB''Z_{j-i} & ^tW_iB''W_{j-i} \\
  \end{array}
\right)\right).
\end{align*}

  Hence we should find $X_j$, $Y_j$, $Z_{j-1}$ and $W_{j-1}$
  inductively. Take any $Z_{j-1}$ with the first column
  $^t(v_{j-1}^{i+1}, \ldots, v_{j-1}^{n})$, where
  $\ov=\ov_0+p\ov_1+\cdots+p^j\ov_j+\cdots$. Then by step3 with the fact
  that $^t\ov B\ov=\sum_{k=0}^{j}p^j\left(\sum_{i=0}^{k} {}^t\ov_i B\ov_{k-i})\right)$ mod
  $p^{j+1}$, we can find an appropriate $X_j$, $W_{j-1}$ and $Y_j$
  satisfying the following equations.
  \begin{eqnarray*}
  ^tX_0B'X_j+^tX_jB'X_0&=&-\sum_{i=1}^{j-1}{^tX_i}B'X_{j-i}-\sum_{i=0}^{j-1}
  {^tZ_i}B''Z_{j-i-1}+C_j^{11}\ \ \text{mod}\ p,\\
  ^tW_{j-1}B''W_0+^tW_0B''W_{j-1}&=&-\sum_{i=1}^{j-1}\left(^tW_iB''W_{j-i-1}+^tY_iB'Y_{j-i}\right)+C_n^{22}\
  \ \text{mod}\ p,\\
  ^tX_0B'Y_j&=&-\sum_{i=0}^{j-1}\left(^tX_iB'Y_{j-i}+^tZ_iB''W_{j-i-1}\right)+C_j^{12}\
  \ \text{mod}\ p,
  \end{eqnarray*}
  where $C_j^{11}$, $C_j^{22}$ and $C_j^{12}$ are obtained from the
  equations of the formal level $j-1$ (see the step2).

  Consequently, we can find $k\in \GL(n,\mathbb Z_p)$ such
  that $^tkBk=B$. Since $\det k=\pm1$, $k$ may not be an element of
  $ K$. However we can easily find $k'\in  K$ with $k'.\ve_1=\ov$ using reflections in $\mathbb Q_p^n$.
\end{proof}

\end{document}